\documentclass[nohyperref]{article}

\usepackage{amsmath}
\usepackage{amssymb}
\usepackage{amsthm}
\usepackage{mathrsfs}
\usepackage{amsfonts}
\usepackage{bm}
\usepackage{xparse}

\usepackage{nicefrac}       
\usepackage{booktabs}       
\usepackage{mathtools}
\usepackage{graphicx}
\usepackage{pifont}
\usepackage{appendix}
\usepackage{enumitem}
\usepackage{environ}         
\usepackage{etoolbox}        
\usepackage[nodisplayskipstretch]{setspace}
\usepackage{enumitem}
\usepackage{mdframed}
\usepackage{adjustbox}

\usepackage{dsfont}
\usepackage[normalem]{ulem}
\usepackage{thmtools} 
\usepackage{lipsum}
\usepackage{float}

\usepackage{tikz}
\usepackage{multicol}
\usepackage{subfigure}
\usepackage{thm-restate}
\usepackage{xcolor}
\definecolor{mydarkblue}{rgb}{0,0.08,0.45}
\allowdisplaybreaks
\usepackage{hyperref}
\hypersetup{
    unicode=false,          
    pdftoolbar=true,        
    pdfmenubar=true,        
    pdffitwindow=false,     
    pdfstartview={FitH},    
    pdfnewwindow=true,      
    colorlinks=true,       
    linkcolor=mydarkblue,          
    citecolor=mydarkblue,        
    filecolor=magenta,      
    urlcolor=cyan           
}

\usepackage[capitalise, nameinlink]{cleveref}

\definecolor{shadecolor}{gray}{0.90}
\declaretheoremstyle[
headfont=\normalfont\bfseries,
notefont=\mdseries, notebraces={(}{)},
bodyfont=\normalfont,
postheadspace=0.5em,
spaceabove=5pt,
mdframed={
  skipabove=3pt,
  skipbelow=3pt,
  hidealllines=true,
  backgroundcolor={shadecolor},
  innerleftmargin=2pt,
  innerrightmargin=2pt}
]{shaded}

\newtheorem{lemmanew}{Lemma}

\usepackage{xpatch}
\makeatletter
\xpatchcmd{\proof}{\topsep6\p@\@plus6\p@\relax}{}{}{}
\makeatother

\crefname{lemmanew}{Lemma}{Lemmas}
\crefname{eq}{Equation}{Equations}

\mdfdefinestyle{mdframedthmbox}{%
	leftmargin=.0\textwidth,
	rightmargin=.0\textwidth,%
	innertopmargin=0.75em,
	innerleftmargin=.5em,
	innerrightmargin=.5em,
}

\newenvironment{thmbox}
	{%
		\begin{mdframed}[style=mdframedthmbox]%
	}{%
		\end{mdframed}%
	}

\newcommand{\myquote}[1]{\null~\\{\null\hspace{.05\textwidth}\begin{minipage}[t]{.90\textwidth} #1 \end{minipage}}}

\makeatletter
\def\thm@space@setup{\thm@preskip=0pt
\thm@postskip=0pt}
\makeatother

\newcommand{\x}{w}
\newcommand{\xk}{w_{k}}
\newcommand{\xkp}{w_{k-1}}

\newcommand{\nul}{\nu_L}
\newcommand{\num}{\nu_\mu}

\newcommand{\yk}{y_{k}}

\newcommand{\z}{z}
\newcommand{\zk}{z_{k}}
\newcommand{\zkk}{z_{k+1}}

\newcommand{\ak}{r_{k}}

\newcommand{\akp}{r_{k-1}}

\newcommand{\bk}{b_{k}}

\newcommand{\q}{q}
\newcommand{\qk}{q_{k}}
\newcommand{\qkk}{q_{k+1}}

\newcommand{\alphak}{\alpha_{k}}
\newcommand{\gammak}{\gamma_{k}}

\newcommand{\gammakp}{\gamma_{k-1}}

\newcommand{\etakp}{\eta_{k-1}}

\newcommand{\inner}[2]{\langle #1, #2 \rangle}

\newcommand{\xkk}{w_{k+1}}
\newcommand{\xopt}{w^{*}}

\newcommand{\grad}[1]{\nabla f(#1)}

\newcommand{\norm}[1]{\left\|#1\right\|}
\newcommand{\normsq}[1]{\left\|#1\right\|^{2}}
\newcommand{\cI}[1]{\mathcal{I}\left\{#1\right\}}

\newcommand{\E}{\mathbb{E}}

\newcommand{\fk}{f_{ik}}
\newcommand{\fj}{f_{i}}
\newcommand{\fjopt}{f_{i}^*}
\newcommand{\etak}{\eta_{k}}

\newcommand{\Lk}{L_{ik}}

\newcommand{\gradk}[1]{\nabla f_{ik}(#1)}

\newcommand{\gradb}[1]{\nabla f_{\gB}(#1)}

\newcommand{\gradi}[1]{\nabla f_{i}(#1)}

\newcommand{\aligns}[1]{\begin{align*} #1 \end{align*}}

\def \nus/{\texttt{NUS}}

\newcommand{\nleft}{\mathclose\bgroup\left}
\newcommand{\nright}{\aftergroup\egroup\right}

\def\ceil#1{\lceil #1 \rceil}
\def\floor#1{\lfloor #1 \rfloor}
\def\1{\bm{1}}

\DeclareMathAlphabet{\mathsfit}{\encodingdefault}{\sfdefault}{m}{sl}
\SetMathAlphabet{\mathsfit}{bold}{\encodingdefault}{\sfdefault}{bx}{n}

\def\gB{{\mathcal{B}}}

\def\gN{{\mathcal{N}}}

\newcommand{\R}{\mathbb{R}}

\newcommand{\lin}[1]{\langle#1\rangle}

\newcommand{\abs}[1]{\nleft\vert#1\nright\vert}

\newcommand{\tightsub}[1]{{\kern -.1em \raise-.1em\hbox{\tiny$#1$}}{}}

\newcommand{\red}[1]{\textcolor{black}{#1}}
\newcommand{\blue}[1]{\textcolor{black}{#1}}

\newcommand{\Lmax}{L}

\newcommand{\lambdak}{\lambda_{k}}
\newcommand{\lambdakk}{\lambda_{k+1}}

\newcommand{\phik}{\phi_{k}}
\newcommand{\phikk}{\phi_{k+1}}
\newcommand{\gNk}{{\mathcal{N}_k}}
\newcommand{\gNkk}{{\mathcal{N}_{k+1}}}

\usepackage[accepted]{icml2022}

\icmltitlerunning{Towards Noise-adaptive, Problem-adaptive (Accelerated) SGD}

\begin{document}

\twocolumn[
\icmltitle{Towards Noise-adaptive, Problem-adaptive (Accelerated) SGD}



\icmlsetsymbol{equal}{*}

\begin{icmlauthorlist}
\icmlauthor{Sharan Vaswani}{sfu}
\icmlauthor{Benjamin Dubois-Taine}{inria}
\icmlauthor{Reza Babanezhad}{samsung}
\end{icmlauthorlist}

\icmlaffiliation{sfu}{Simon Fraser University}
\icmlaffiliation{inria}{DI ENS, Ecole normale supérieure, Université PSL, CNRS, INRIA, 75005 Paris, France}
\icmlaffiliation{samsung}{SAIT AI lab, Montreal}

\icmlcorrespondingauthor{Sharan Vaswani}{vaswani.sharan@gmail.com}

\icmlkeywords{Convex optimization, Stochastic gradient descent, Adaptive}

\vskip 0.3in
]



\printAffiliationsAndNotice{}  

\begin{abstract}
We aim to make stochastic gradient descent (SGD) adaptive to (i) the noise $\sigma^2$ in the stochastic gradients and (ii) problem-dependent constants. When minimizing smooth, strongly-convex functions with condition number $\kappa$, we prove that $T$ iterations of SGD with exponentially decreasing step-sizes and knowledge of the smoothness can achieve an $\tilde{O} \left(\exp \left( \nicefrac{-T}{\kappa} \right) + \nicefrac{\sigma^2}{T} \right)$ rate, without knowing $\sigma^2$. In order to be adaptive to the smoothness, we use a stochastic line-search (SLS) and show (via upper and lower-bounds) that SGD with SLS converges at the desired rate, but only to a neighbourhood of the solution. On the other hand, we prove that SGD with an offline estimate of the smoothness converges to the minimizer. However, its rate is slowed down proportional to the estimation error. Next, we prove that SGD with Nesterov acceleration and exponential step-sizes (referred to as ASGD) can achieve the near-optimal $\tilde{O} \left(\exp \left( \nicefrac{-T}{\sqrt{\kappa}} \right) + \nicefrac{\sigma^2}{T} \right)$ rate, without knowledge of $\sigma^2$. When used with offline estimates of the smoothness and strong-convexity, ASGD still converges to the solution, albeit at a slower rate. Finally, we empirically demonstrate the effectiveness of exponential step-sizes coupled with a novel variant of SLS.
\end{abstract}

\section{Introduction}
\label{sec:introduction}
We study unconstrained minimization of a finite-sum objective $f:\R^d\rightarrow\R$ prevalent in machine learning,  
\begin{align}
\min_{\x \in \R^d} f(w) := \frac{1}{n}\sum_{i=1}^n f_i(w).    
\label{eq:objective}
\end{align}
For supervised learning, $n$ represents the number of training examples and $f_i$ is the loss on example $i$. We assume $f$ to be a smooth, strongly-convex function and denote $\xopt$ to be the unique minimizer of the above problem.

We study stochastic gradient descent (SGD) and its accelerated variant for minimizing $f$~\citep{robbins1951stochastic, NemYudin1983book, nesterov2013introductory, bottou2018optimization}. The empirical performance and the theoretical convergence of SGD is governed by the choice of its step-size, and there are numerous ways of selecting it. For example,~\citet{moulines2011non,gower2019sgd} use a \emph{constant} step-size for convex and strongly convex functions. A constant step-size only guarantees convergence to a neighborhood of the solution. In order to converge to the exact minimizer, a common technique is to decrease the step-size at an appropriate rate, and such decreasing step-sizes have also been well-studied~\citep{robbins1951stochastic,ghadimi2012optimal}. The rate at which the step-size needs to be decayed depends on the function class under consideration. For example, when minimizing smooth, strongly-convex functions using $T$ iterations of SGD, the step-size is decayed at an $O(1/k)$ rate where $k$ is the iteration number. This results in an $\Theta(1/T)$ convergence rate for SGD and is optimal in the stochastic setting~\citep{nguyen2018tight}. 

On the other hand, when minimizing a smooth, strongly-convex function with condition number $\kappa$, deterministic (full-batch) gradient descent (GD) with a \emph{constant} step-size converges linearly at an $O(\exp(-T/\kappa))$ rate. Augmenting constant step-size GD with Nesterov acceleration can further improve the convergence rate to $\Theta(\exp(-T/\sqrt{\kappa}))$ which is optimal in the deterministic setting~\citep{nesterov2013introductory}. Hence, the stochastic and deterministic algorithms use different step-size strategies to obtain the optimal rates in their respective settings. 

\textbf{Noise-adaptive SGD}: Ideally, we want to design step-size schemes that make SGD adaptive to the noise in the stochastic gradients, matching the optimal convergence rates in both the deterministic and stochastic settings. Furthermore, in order for the algorithm to be practical, it should not require knowledge of the stochasticity (e.g. a bound on $\sigma^2$, the variance in stochastic gradients). Recently,~\citet{khaled2020better,li2020second} achieve the $\tilde{O} \left(\exp(-T/\kappa) + \nicefrac{\sigma^2}{T} \right)$ for smooth functions satisfying the Polyak-Lojasiewicz (PL) condition~\citep{karimi2016linear}, a generalization of strong-convexity. More importantly, these works are \emph{noise-adaptive} and do not require the knowledge of $\sigma^2$. For this,~\citet{li2020second} use SGD with an exponentially decreasing sequence of step-sizes, while~\citet{khaled2020better} use a constant then decaying step-size. There are two limitations with these works: (i) they require the knowledge of problem-dependent constants such as the smoothness and strong-convexity of the underlying function, and (ii) they do not match the optimal $\sqrt{\kappa}$ dependence (of the Nesterov accelerated method) in the linear convergence term, and are hence sub-optimal in the deterministic setting. We will address both these limitations in this work.  

\textbf{Towards noise and problem-adaptive SGD}: Typically, SGD requires the knowledge of problem-dependent constants to set the step-size. In practice, it is difficult to estimate these quantities, and one can only obtain loose bounds on them. Consequently, there have been numerous methods~\citep{duchi2011adaptive, li2018convergence, kingma2014adam, bengio2015rmsprop, vaswani2019painless, loizou2021stochastic} that can adapt to the problem, and adjust the step-size on the fly. We term such methods as \emph{problem-adaptive}. Unfortunately, it is unclear if such problem-adaptive methods can also be made noise-adaptive. On the other hand, as mentioned above, none of the noise-adaptive methods~\citep{li2020second,khaled2020better,stich2019unified} are problem-adaptive. Amongst these, the noise-adaptive algorithm in~\citet{li2020second} only requires knowledge of the smoothness constant and we try to relax this requirement. 

\textbf{Contribution}: In~\cref{sec:sc-unknown-correlated}, we use stochastic line-search (SLS)~\citep{vaswani2019painless} to estimate the smoothness constant on the fly. We prove that SGD in conjunction with exponentially decreasing step-sizes and SLS converges at the desired noise-adaptive rate but only to a \emph{neighbourhood of the solution}. This neighbourhood depends on the noise and the error in estimating the smoothness. We prove a corresponding lower-bound that shows the necessity of this neighbourhood term. Our lower-bound shows that if the SGD step-size is adaptively set in an online fashion (using the sampled function), no decreasing sequence of step-sizes can converge to the minimizer.

\textbf{Contribution}: In~\cref{sec:sc-unknown-decorrelated}, we consider estimating the smoothness constant in an offline fashion (before running the algorithm). We prove that SGD with an offline estimate of the smoothness and exponentially decreasing step-sizes converges to the solution, though its rate is slowed down by a factor proportional to the estimation error in the smoothness. In particular, our upper-bound shows that misestimating the smoothness constant can slow down the convergence rate. We complement this result with a lower-bound that shows that this slowdown is unavoidable.

Our results thus demonstrate the difficulty of obtaining noise-adaptive rates while being adaptive to problem-dependent parameters. 

\textbf{Noise-adaptive SGD with Nesterov acceleration}: We now turn to the second limitation of existing noise-adaptive methods, and aim to use Nesterov acceleration in order to obtain the optimal $\tilde{O} \left(\exp(-T/\sqrt{\kappa}) + \nicefrac{\sigma^2}{T} \right)$ rate, without the knowledge of $\sigma^2$. The work in~\citet{jain2018accelerating,arjevani2020tight} satisfies the desired criteria for quadratic functions. For general smooth, strongly-convex functions, ~\citet{ghadimi2013optimal,kulunchakov2019estimate} obtain the desired rate, but require the knowledge of $\sigma^2$, and are consequently not noise-adaptive. \citet{aybat2019universally} propose a multi-stage accelerated algorithm that does not require knowledge of $\sigma^2$. The authors use a dynamical systems analysis, and prove that their algorithm achieves the desired optimal rate \emph{only for $T \geq 2 \sqrt{\kappa}$}.

\textbf{Contribution}: In contrast, in~\cref{sec:acceleration}, we use SGD with a stochastic variant of Nesterov acceleration~\citep{cohen2018acceleration,vaswani2019fast} and the same exponentially decreasing step-sizes. We refer to the resulting method as Accelerated SGD (ASGD). Compared to~\citet{aybat2019universally}, ASGD is a more natural extension of the deterministic Nesterov accelerated gradient method. Under a bounded variance assumption on the stochastic gradients, we use the standard estimating sequences analysis, and prove that ASGD \emph{achieves the desired rate for all $T$ without the knowledge of $\sigma^2$}. Hence, exponentially decreasing step-sizes result in noise-adaptivity for both SGD and ASGD. 

\textbf{Contribution}: ASGD requires the knowledge of both the smoothness and strong-convexity parameters. As a step towards problem-adaptivity for ASGD, we analyze its convergence with offline estimates of these problem-dependent constants. To the best of our knowledge this is the first such result. We prove that, similar to SGD, misspecified ASGD converges to the minimizer, but its rate is slowed down by a factor proportional to the estimation errors. 

\textbf{Contribution}: Finally, in~\cref{sec:experiments}, we evaluate the performance of different step-size schemes on strongly-convex supervised learning problems. We show that (A)SGD consistently out-perform existing noise-adaptive algorithms. We propose a novel variant of SLS that guarantees convergence to the minimizer and demonstrate its practical effectiveness in making (A)SGD problem-adaptive.  

\textbf{Additional contributions}: In~\cref{sec:qsc-ub}, we show matching results for SGD on strongly star-convex functions~\citep{hinder2020near}, a class of structured non-convex functions. Finally, we prove upper-bounds for non-strongly-convex functions (\cref{sec:convex-ub}) and show that even when the smoothness constant is known, exponentially decreasing step-sizes converge to a neighbourhood of the solution. We give some justification as to why  polynomial or exponentially decreasing step-sizes are unlikely to be noise-adaptive in this setting. 
\section{Problem setup and Background}
\label{sec:background}

We assume that $f$ and each $\fj$ are differentiable and lower-bounded by $f^*$ and $\fjopt$, respectively. Throughout the paper, we assume that $f$ is $\mu$-strongly convex, and each $f_i$ is convex. We also assume that each function $\fj$ is $L_i$-smooth, implying that $f$ is $\Lmax$-smooth with $\Lmax := \max_{i} L_i$ (see~\cref{app:definitions} for the necessary definitions) and define $\kappa:= \frac{\Lmax}{\mu}$. 

We use stochastic gradient descent (SGD) or SGD with Nesterov acceleration~\citep{nesterov2013introductory} (referred to as ASGD) to minimize $f$ in~\cref{eq:objective}. 
In each iteration $k \in [T]$, SGD selects a function $\fk$ (typically uniformly) at random, computes its gradient and takes a descent step. Specifically, 
\begin{align}
\xkk = \xk - \gammak \alphak \gradk{\xk},
\label{eq:sgd}    
\end{align}
where $\xkk$ and $\xk$ are the SGD iterates, and $\gradk{\cdot}$ is the gradient of the loss function chosen at iteration $k$. Each stochastic gradient $\gradk{\x}$ is unbiased, implying that $\E_i \left[ \gradi{\x} \right] = \grad{\x}$. The product of scalars $\etak := \gammak \alphak$ defines the \emph{step-size} for iteration $k$. The step-size consists of two parts -- $\gammak$, a problem-dependent scaling term  that captures the (local) smoothness of the function; and $\alphak$, a problem-independent term that controls the decay of the step-size. Typically, $\alphak$ is a decreasing sequence of $k$, and $\lim_{k \rightarrow \infty} \alphak = 0$. The choice of the $\alphak$ sequence depends on the properties of $f$, for example, for smooth, strongly-convex functions, $\alphak$ is typically set to be $O(\nicefrac{1}{k})$. 

Throughout the paper, we will assume that $T$ is known in advance. In order to obtain noise-adaptive rates, we consider exponentially decreasing step-sizes~\citep{li2020second} of the form $\alphak := \alpha^{k}$ where $\alpha := \left[\frac{\beta}{T}\right]^{1/T} \leq 1$ for a constant $\beta \geq 1$. These step-sizes lie between the constant step-size used in the deterministic setting and the $\nicefrac{1}{k}$ decreasing step-sizes used in the stochastic setting, meaning that for $k \in [T]$, $\alphak \in \left[\frac{1}{k}, 1 \right]$.  

In the next section, we analyze the convergence of SGD with exponentially decreasing step-sizes for smooth, strongly-convex functions.
\section{Towards noise \& problem adaptive SGD}
\label{sec:adaptivity}
In this section, we consider approaches for developing noise and problem-adaptive SGD i.e. we aim to obtain the noise-adaptive rate matching~\citet{stich2019unified,li2020second, khaled2020better}, but do so without the knowledge of problem-dependent constants. 

Instead of the typical assumption of finite gradient noise $z^2 := \E_i [\norm{\nabla f_i(\xopt)}^2] < \infty$, we assume a finite optimal objective difference. Specifically, we define the noise as $\sigma^2 := \E_i [\fj(\xopt) - \fjopt] \geq 0$. This notion of noise has been used to study the convergence of constant step-size SGD in the \emph{interpolation} setting for over-parameterized models~\citep{zhang2019stochastic,loizou2021stochastic, vaswani2020adaptive}. Note that when interpolation is exactly satisfied, $\sigma = z = 0$. In general, if each function $f_i$ is $\mu$-strongly convex and $L$-smooth, then $\frac{1}{2L}z^2\leq \sigma^2 \leq\frac{1}{2\mu}z^2$. 

As a warm-up, we first assume knowledge of the smoothness constant in~\cref{sec:sc-known} and analyze the resulting SGD algorithm with exponentially decreasing step-sizes. In~\cref{sec:sc-unknown-correlated}, we consider using a stochastic line-search~\citep{vaswani2019painless,vaswani2020adaptive} in order to estimate the smoothness constant and set the step-size on the fly. Finally, in~\cref{sec:sc-unknown-decorrelated}, we analyze the convergence of SGD when using an offline estimate of the smoothness. 

\subsection{Known smoothness}
\label{sec:sc-known}
We use the knowledge of smoothness to set the problem-dependent part of the step-size for SGD, specifically, $\gammak = \nicefrac{1}{\Lmax}$. With an exponentially decreasing $\alphak$-sequence, we prove the following theorem in~\cref{app:sc-known-proof}. 

\begin{restatable}{theorem}{restatesck}
Assuming (i) convexity and $L_i$-smoothness of each $f_i$, (ii) $\mu$ strong-convexity of $f$, SGD (\cref{eq:sgd}) with $\gammak = \frac{1}{\Lmax}$, $\alphak = \left(\frac{\beta}{T}\right)^{k/T}$ converges as,
\begin{align*}
\E \normsq{\x_{T+1} - \xopt} & \leq \normsq{\x_{1} - \xopt} c_2 \, \exp\left( - \frac{T}{\kappa} \frac{\alpha}{\ln(\nicefrac{T}{\beta})}\right) \\ & + \frac{8 \sigma^2 c_2 \kappa}{\mu e^2} \frac{(\ln(\nicefrac{T}{\beta}))^2}{\alpha^2 T},
\end{align*}
where $c_2 = \exp\left( \frac{1}{\kappa} \cdot \frac{2\beta}{\ln(\nicefrac{T}{\beta})}\right)$.
\label{thm:sc-known}
\end{restatable}

Compared to~\citet{moulines2011non} that use polynomially decreasing step-sizes, exponential step-sizes result in a better trade-off between the bias (initial distance to the minimizer) and variance (noise) terms, achieving the desired $\tilde{O}\left(\exp(-T/\kappa) + \nicefrac{\sigma^2}{T} \right)$ noise-adaptive rate. In~\cref{lemma:polynomial-lb-1,lemma:polynomial-lb-2} in~\cref{app:polynomial}, we show that no polynomially decreasing step-size can result in the desired noise-adaptive rate. In order to interpolate between the stochastic (mini-batch size equal to $1$) and fully deterministic (mini-batch size equal to $n$) setting, we show the explicit dependence of $\sigma^2$ on the mini-batch size in~\cref{app:minibatch}.

Since strongly-convex functions also satisfy the PL condition~\citep{karimi2016linear}, the above result can be deduced from~\citep{li2020second}. However, unlike~\citep{li2020second}, our result does not require the growth condition and uses a weaker notion of noise. Moreover, we use a different proof technique, specifically, \citet{li2020second} use the smoothness inequality in the first step and obtain the rate in terms of the function suboptimality, $\E[f(\x_T) - f^*]$. In contrast, our proof uses an expansion of the iterates to obtain the rate in terms of the distance to the minimizer, $\E \normsq{\x_{T+1} - \xopt}$. This change allows us to easily handle the case when the smoothness constant is unknown and needs to be estimated.

Next, we use stochastic line-search techniques to estimate the unknown smoothness and set the step-size on the fly.

\subsection{Online estimation of unknown smoothness}
\label{sec:sc-unknown-correlated}
In this section, we assume that the smoothness constant is unknown, and aim to estimate it and set the step-size in an \emph{online} fashion. By online estimation, we mean that in iteration $k$, we use the knowledge of the sampled function $i_k$ to set the step-size, i.e. setting $\gammak$ depends on $i_k$. We only consider methods that use the knowledge of $i_k$ in iteration $k$ and are not allowed to access the other functions in $f$ (for example, to compute the full-batch gradient at $\xk$). Methods based on a stochastic line-search~\citep{vaswani2019painless, vaswani2020adaptive} or the stochastic Polyak step-size~\citep{loizou2021stochastic, berrada2020training} satisfy this criterion.    

We use stochastic line-search (SLS) to estimate the local Lipschitz constant and set $\gammak$, the problem-dependent part of the step-size. SLS is the stochastic analog of the traditional Armijo line-search~\citep{armijo1966minimization} used for deterministic gradient descent~\citep{nocedal2006numerical}. In iteration $k$, SLS estimates the smoothness constant $\Lk$ of the sampled function using $\fk$ and $\nabla f_{ik}$. In particular, starting from a guess ($\gamma_{\max}$) of the step-size, SLS uses a backtracking procedure and returns the largest step-size $\gammak$ that satisfies: $\gammak \leq \gamma_{\max}$ and, 
\begin{align}
\fk(\xk - \gammak \gradk{\xk}) & \leq \fk(\xk) - c \, \gammak \normsq{\gradk{\xk}}.
\label{eq:armijo-ls}    
\end{align}
Here, $c \in (0,1)$ is a hyper-parameter to be determined  theoretically. SLS guarantees that resulting the step-size $\gammak$ lies in the $\left[ \min \left\{ \frac{2 (1 - c)}{\Lk}, \gamma_{\max} \right\}, \gamma_{\max} \right]$ range (\cref{lem:sls-bounds}). If the initial guess is large enough i.e. $\gamma_{\max} > 1/\Lk$, then the resulting step-size $\gammak \geq \frac{2 (1 -c)}{\Lk}$. Thus, with $c = \nicefrac{1}{2}$, SLS can be used to obtain an upper-bound on $\nicefrac{1}{\Lk}$. 

In the interpolation ($\sigma = 0$) setting, a constant step-size ($\alphak = 1$ for all $k$) suffices, and SGD with SLS achieves a linear rate of convergence (for $c \geq \nicefrac{1}{2}$) when minimizing smooth, strongly-convex functions~\citep{vaswani2019painless}. 
In general, for a non-zero $\sigma$, using SGD with SLS and no step-size decay ($\alphak = 1$) results in $O\left(\exp(-T/\kappa) + \gamma_{\max} \sigma^2 \right)$ rate~\citep{vaswani2020adaptive}, implying convergence to a neighbourhood determined by the $\gamma_{\max} \sigma^2$ term.   

In order to obtain a similar rate as~\cref{thm:sc-known} but without the knowledge of $\Lmax$, we set $\gammak$ with SLS and use the same exponentially decreasing $\alphak$-sequence. We prove the following theorem in~\cref{app:sc-unknown-correlated-proof}. 
\begin{restatable}{theorem}{restatescukc}
Under the same assumptions as~\cref{thm:sc-known}, SGD (\cref{eq:sgd}) with $\alphak = \left(\frac{\beta}{T}\right)^{k/T}$, $\gammak$ as the largest step-size that satisfies $\gammak \leq \gamma_{\max}$ and~\cref{eq:armijo-ls} with $c = \nicefrac{1}{2}$ converges as,
\begin{align*}
\E \normsq{\x_{T+1} - \xopt} & \leq \normsq{\x_1 - \xopt} c_1 \, \exp\left( - \frac{T}{\kappa'} \frac{\alpha}{\ln(\nicefrac{T}{\beta})}\right) \\
& + \frac{8 \sigma^2 c_1 (\kappa')^2 \gamma_{\max}}{e^2} \frac{(\ln(\nicefrac{T}{\beta}))^2}{\alpha^2 T}  \\ 
& + \frac{2 \sigma^2 c_1 \kappa' \ln (T / \beta) \, \gamma_{\text{err}}}{e \alpha},
\end{align*}
where $\gamma_{\text{err}} := \left(\gamma_{\max} - \min \left\{\gamma_{\max}, \frac{1}{\Lmax} \right\} \right)$, \\ $\kappa' := \max\left\{\frac{\Lmax}{\mu}, \frac{1}{\mu \gamma_{\max}}\right\}$, $c_1 = \exp\left( \frac{1}{\kappa'} \cdot \frac{2\beta}{\ln(\nicefrac{T}{\beta})}\right)$.
\label{thm:sc-unknown-correlated}
\end{restatable}

\vspace{-3ex}
We observe that the first two terms are similar to those in~\cref{thm:sc-known}. For $\gamma_{\max} \geq \frac{1}{\Lmax}$, $\kappa' = \kappa$ and the above theorem implies the same $\tilde{O}\left(\exp(-T/\kappa) + \frac{\sigma^2}{T}\right)$ rate of convergence. However, as $T \rightarrow \infty$, $\x_{T+1}$ does not converge to $\xopt$, but rather to a neighbourhood determined by the last term $\frac{2 \, \sigma^2 \kappa' c_1 \ln(T/\beta)}{e \alpha} \, \left(\gamma_{\max} - \min \left\{\gamma_{\max}, \frac{1}{\Lmax} \right\}\right)$. The neighbourhood thus depends on the noise $\sigma^2$ and $\gamma_{\text{err}}$, the estimation error (in the smoothness) of the initial guess.  

When $\sigma^2 = 0$, this neighbourhood term disappears, and SGD converges to the minimizer despite the estimation error. This matches the result for SLS in the interpolation setting~\citep{vaswani2019painless}. Conversely, when the smoothness is known and $\gamma_{\max}$ can be set equal to $\frac{1}{\Lmax}$, we also obtain convergence to the minimizer and recover the result of~\cref{thm:sc-known}. In fact, if we can ``guess'' a value of $\gamma_{\max} \leq \frac{1}{\Lmax}$, it would result in the neighbourhood term becoming zero, thus ensuring convergence to the minimizer. In this case, the stochastic line-search does not decrease the step-size in any iteration, and the algorithm becomes the same as using a constant step-size equal to $\gamma_{\max}$. Finally, we contrast our result with the $\alphak = 1$ setting~\citep{vaswani2020adaptive}, and observe that instead of the dependence on $\gamma_{\max}$, our neighbourhood term depends on the estimation error in the smoothness. Next, we show the necessity of such a neighbourhood term.   
\subsubsection{Lower bound on quadratics}
\label{sec:lb-correlated}
In order to prove a lower-bound, we consider a pair of $1$-dimensional quadratics $f_i(\x) = \nicefrac{1}{2} \, (x_i \x - y_i)^2$ for $i = 1, 2$. Here, $\x$, $x_i$, $y_i$ are all scalars. The overall function to be minimized is $f(w) = (1/2) \cdot [f_1(w) + f_2(w)]$. We assume that $\norm{x_1} \neq \norm{x_2}$, and since $L_i = \normsq{x_i}$, this assumption implies different smoothness constants for the two functions. For a sufficiently large value of $\gamma_{\max}$ i.e.  $\left(\gamma_{\max} \geq \frac{1}{\min_{i \in [2]} L_i} \right)$, using SLS with $c \geq 1/2$ (required for convergence) results in $\gammak \leq \nicefrac{1}{\Lk}$\footnote{For $1$-dimensional quadratics, $\gammak = \nicefrac{1}{\Lk}$ for $c = \nicefrac{1}{2}$.} (see~\cref{lem:sls-bounds}). With these choices, we prove the following lower-bound. 

\begin{restatable}{theorem}{restatescukclb}
When using $T$ iterations of SGD to minimize the sum $f(w) = \frac{f_1(w) + f_2(w)}{2}$ of two one-dimensional quadratics, $ f_1(w) = \frac{1}{2}(w - 1)^2$ and $f_2(w) = \frac{1}{2}\left(2w + \nicefrac{1}{2} \right)^2$, setting $\gammak$ using SLS with $\gamma_{\max} \geq 1$ and $c \geq \nicefrac{1}{2}$, any convergent sequence of $\alphak$ results in convergence to a neighbourhood of the solution. Specifically, if $\xopt$ is the minimizer of $f$ and $w_1 > 0$, then,
\begin{align*}
\E (w_T - w^*) & \geq \min\left( w_1, \frac{3}{8}\right). 
\end{align*}
\label{thm:lb-correlated}
\end{restatable}
The above result (proved in~\cref{app:lb-correlated-proof}) shows that using SGD with SLS to set $\gammak$ and \emph{any convergent sequence of $\alphak$} (including the exponentially-decreasing sequence in~\cref{thm:sc-unknown-correlated}) will necessarily result in convergence to a neighbourhood. The neighbourhood term can thus be viewed as the \emph{price of misestimation} of the unknown smoothness constant. This result is in contrast to the conventional thinking that choosing an $\alphak$ sequence such that $\lim_{k \rightarrow \infty} \alphak = 0$ will always ensure convergence to the minimizer. Note that this result is not specific to SLS and would hold for other related methods~\citep{loizou2021stochastic,berrada2020training}. 

Since the lower-bound holds for any convergent $\alphak$ sequence, a possible reason for this convergence to the neighbourhood is the correlation between $i_k$ and the computation of $\gammak$. We investigate this hypothesis in the next section. 

\subsection{Offline estimation of unknown smoothness}
\label{sec:sc-unknown-decorrelated}
In this section, we consider an offline estimation of the smoothness constant. By offline, we mean that in iteration $k$, $\gammak$ is set \emph{before} sampling $i_k$ and cannot use any information about it. This ensures that $\gammak$ is decorrelated with the sampled function $i_k$. The entire sequence of $\gammak$ can even be chosen before running SGD. 

For simplicity of calculations, we consider a fixed $\gammak = \gamma$ for all iterations. Here $\gamma$ is an offline estimate of $\frac{1}{\Lmax}$, and can be obtained by any method. Without loss of generality, we assume that this offline estimate is off by a multiplicative factor $\nu$ that is $\gamma = \frac{\nu}{\Lmax}$ for some $\nu > 0$. Here $\nu$ quantifies the estimation error in $\gamma$ with $\nu = 1$ corresponding to an exact estimation of $\Lmax$. In practice, it is typically possible to obtain lower-bounds on the smoothness constant. Hence, the $\nu > 1$ regime is of practical interest. For SGD with $\gammak = \gamma = \frac{\nu}{\Lmax}$ and an exponentially decreasing $\alphak$-sequence, we prove the following theorem in~\cref{app:sc-unknown-decorrelated-proof}.  
\begin{restatable}{theorem}{restatescukdc}
Under the same assumptions as~\cref{thm:sc-known}, SGD (\cref{eq:sgd}) with $\alphak = \left(\frac{\beta}{T}\right)^{k/T}$, $\gammak = \frac{\nu}{\Lmax}$ converges as,
\begin{align*}
& \Delta_{T+1} \leq \Delta_1 \, c_2  \exp\left(- \frac{\min\{\nu,1\} \, T}{\kappa} \frac{\alpha}{\ln(\nicefrac{T}{\beta})}\right) \\ & + \max\{\nu^2,1\} \frac{8 c_2 \kappa \, \ln(T/\beta)}{\mu \, e^2 \, \alpha^2 \, T} \left[2 \sigma^2 \ln(T/\beta) + G \, [\ln(\nu)]_{+} 
\right]
\end{align*}
where $c_2 = \exp\left( \frac{1}{\kappa} \, \frac{2\beta}{\ln(\nicefrac{T}{\beta})}\right)$, $[x]_{+} = \max\{x, 0\}$, $k_0 = \floor{T \frac{[\ln(\nu)]_+}{\ln(T/\beta)}}$, $G = \max_{j \in [k_0]} \{ f(w_j) - f^* \}$ and $\Delta_k := \normsq{\x_{k} - \xopt}$.
\label{thm:sc-unknown-decorrelated}
\end{restatable}
\vspace{-2ex}
The above theorem implies an $\tilde{O}\left(\exp\left(- \frac{\min\{\nu,1\} \, T}{\kappa}\right) + \frac{\max\{\nu^2,1\} \, \left[\sigma^2 + G [\ln(\nu)]_{+} \right]}{T} \right)$ convergence to the minimizer. The first two terms are similar to that in~\cref{thm:sc-known} and imply an $\tilde{O} \left( \exp(-T/\kappa) + \frac{\sigma^2}{T} \right)$ convergence to the minimizer. Analyzing the third term, we observe that when $\nu \leq 1$, the third term is zero (since $[\ln(\nu)]_{+} = 0$), and the rate matches that of~\cref{thm:sc-known} up to constants that depend on $\nu$. The third term depends on $\left[\max_{j \in k_0} \{ f(w_j) - f^* \} \right]$ because if $\nu > 1$, the step-size $\gammak \alphak = \frac{\nu}{\Lmax} \, \alphak \geq \frac{1}{\Lmax}$ initially, and SGD diverges in this regime. Since $\alphak$ is an exponentially decreasing sequence, after $k_0 := T \frac{\ln(\nu)}{\ln(T/\beta)}$ iterations, $\frac{\nu}{\Lmax} \, \alphak \leq \frac{1}{\Lmax}$, the distance to the minimizer decreases after iteration $k_0$, eventually converging to the solution. 

\red{Furthermore, observe that the second term depends on $\tilde{O} \left(\max\{\nu^2,1\}\right)$ meaning that if we misestimate the smoothness constant by a multiplicative factor of $\nu > 1$, it can slow down the convergence rate by an $O(\nu^2)$ factor.} Finally, our theorem implies that even in the deterministic setting, misestimating $\Lmax$ can slowdown the convergence rate to $O\left(\frac{\nu^2}{T}\right)$ instead of the usual linear rate of convergence. The third term can thus be viewed as the \emph{price of misestimation} of the unknown smoothness constant. Unlike~\cref{thm:sc-unknown-correlated} where this price was convergence to a neighbourhood, here, the price of misestimation is slower convergence to the minimizer. 

\citet{moulines2011non} also considered the effect of misspecifying $L$ but in conjunction with polynomially decreasing step-sizes. \red{Specifically, they proved that using a step-size of $\frac{\nu}{L} \, \frac{1}{T^{\theta}}$ results in the following bounds that depend on $\gamma$ and $\nu$~\citep[Theorem 1]{moulines2011non}. Below, we show their bounds for three common choices of $\theta = \{0, 1/2, 1\}$ and emphasize the effect of $\nu$. 
\begin{align*}
& \Delta_{T+1} =  O\left(\exp \left((\nu^2 - \nicefrac{\nu}{\kappa}) \, T \right) \, (\Delta_1 + \sigma^2) + \nu \sigma^2  \right) \tag{When $\theta = 0$}\\    
& = O\left(\exp \left(\nu^2 \ln(T) - \nicefrac{\nu}{\kappa} \, \sqrt{T} \right) (\Delta_1 + \sigma^2) + \frac{\nu \sigma^2}{\sqrt{T}}  \right) \tag{When $\theta = \frac{1}{2}$} \\
& = O\left(  \exp \left(\nu^2 - \nicefrac{\nu}{\kappa} \, \ln(T) \right) (\Delta_1 + \sigma^2) + \frac{\nu^2 \sigma^2}{T^{\nu/2\kappa}} \right) \tag{When $\theta = 1$ and $\nu < 2 \kappa$} \\
& = O\left(  \exp\left(\nu^2 - \nicefrac{\nu}{\kappa} \, \ln(T) \right) (\Delta_1 + \sigma^2) + \frac{\nu^2 \sigma^2}{T} \right) \tag{When $\theta = 1$ and $\nu \geq 2 \kappa$} 
\end{align*}
Observe that for each regime, the convergence rate depends on $\exp(\nu)$.   
}
In contrast, the convergence rate in~\cref{thm:sc-unknown-decorrelated} depends on $O(\nu^2)$. This robustness towards misspecification can be viewed as an additional advantage of using exponentially decreasing step-sizes. \red{In the next section, we justify the dependence on $[\ln(\nu)]_{+}$ in~\cref{thm:sc-unknown-decorrelated} by proving a corresponding lower-bound. It is unclear whether the $\nu^2$ dependence in~\cref{thm:sc-unknown-decorrelated} is tight, and we leave verifying this for future work.}

\subsubsection{Lower bound on quadratics}
\label{sec:lb-decorrelated}
In this section, we consider gradient descent on a one-dimensional quadratic and study the effect of misestimating the smoothness constant by a factor of $\nu > 1$. We consider minimizing a single quadratic, ensuring that $\sigma^2 = 0$ and prove the following lower-bound in~\cref{app:lb-decorrelated-proof}.
\begin{restatable}{theorem}{restatescukdclb}
\label{thm:lb-decorrelated}
When minimizing a one-dimensional quadratic function $f(w) = \frac{1}{2}(xw - y)^2$, GD with $\alphak = \left(\frac{\beta}{T}\right)^{k/T}$, $\gammak = \frac{\nu}{\Lmax}$ for $\nu > 3$, satisfies
\begin{align*}
    w_{k+1} - w^* & = (w_1 - w^*) \prod_{i=1}^k ( 1- \nu \alpha_i).
\intertext{After $k' := \frac{T}{\ln (T / \beta)} \ln\left(\frac{\nu}{3}\right)$ iterations, we have that}
\abs{w_{k'+1} - w^*} & \geq 2^{k'} \abs{w_1 - w^*}.
\end{align*}
\end{restatable}
\vspace{-1ex}
Instantiating this lower-bound, suppose the estimate of $\Lmax$ is off by a factor of $\nu = 10$, then $\ln\left(\frac{\nu}{3}\right) \geq 1$, which implies that $k' \geq \floor{\frac{T}{\ln(\nicefrac{T}{\beta})}}$. In other words, we do not make any progress in the first $\frac{T}{\ln(\nicefrac{T}{\beta})}$ iterations, and at this point the optimality gap has been multiplied by a factor of $2^{T/\ln(\nicefrac{T}{\beta})}$ compared to the starting optimality gap. This simple example shows the slowdown in the rate of convergence by misestimating the smoothness.
\section{Towards noise \& problem adaptive ASGD}
\label{sec:acceleration}
In this section, we will first aim to use SGD with Nesterov acceleration and obtain the optimal $\tilde{O} \left(\exp \left(\frac{-T}{\sqrt{\kappa}} \right) + \frac{\sigma^2}{T} \right)$ rate without knowledge of $\sigma^2$. Subsequently, we will analyze the convergence of ASGD with offline estimates of the smoothness and strong-convexity parameters, quantifying the price of misspecification (similar to~\cref{sec:sc-unknown-decorrelated}).  

ASGD has two sequences $\{\xk, \yk \}$ and an additional extrapolation parameter $\bk$. ASGD computes the stochastic gradient at the extrapolated point $\yk$ and takes a descent step in that direction. The update in iteration $k$ is:
\begin{align}
\yk &= \xk + \bk \, (\xk - \xkp), \label{eq:extrapolation} \\
\xkk &= \yk - \gammak \, \alphak^2 \, \gradk{\yk}. \label{eq:sgd-update}
\end{align}
For analyzing the convergence of ASGD, we will assume that the variance in the stochastic gradients is bounded at any iterate, such that for all $\x$, 
\begin{align}
\E_{i} \normsq{\nabla \fj(\x) - \grad{\x}} & \leq \sigma^2. 
\label{eq:bounded-variance}
\end{align}
Note that this is a stronger condition than the growth condition in~\citet{bottou2018optimization, vaswani2019fast} and the condition in~\cref{sec:adaptivity}. Note that $\sigma = 0$ in the deterministic setting (when using the full-gradient in~\cref{eq:sgd-update}). We now characterize the convergence of ASGD. 

\begin{restatable}{theorem}{restatescacc}
Under the same assumptions of~\cref{thm:sc-known} and (iii) the bounded variance condition in~\cref{eq:bounded-variance}, ASGD (\cref{eq:extrapolation,eq:sgd-update}) with $\x_1 = y_1$, $\gammak = \frac{1}{L}$, $\alphak = \left(\frac{\beta}{T}\right)^{k/T}$, $\ak = \sqrt{\frac{\mu}{L}} \left(\frac{\beta}{T} \right)^{k/T}$ and $\bk = \frac{(1 - \akp) \, \akp \, \alpha}{\ak + \akp^2 \, \alpha}$ converges as,
\begin{align*}
\Delta_{T+1} & \leq 2c_3  \exp\left( - \frac{T}{\sqrt{\kappa}} \frac{\alpha}{\ln(\nicefrac{T}{\beta})}\right) \Delta_1 \\ & + \frac{4 \sigma^2 c_3}{\mu e^2} \frac{(\ln(\nicefrac{T}{\beta}))^2}{\alpha^2 T},
\end{align*}
where $\Delta_k := \E [f(\x_{k}) - f^*]$ and $c_3 = \exp\left( \frac{2 \beta}{ \sqrt{\kappa} \ln(\nicefrac{T}{\beta})}\right)$.
\label{thm:sc-accelerated}
\end{restatable}
\vspace{-2ex}
The above theorem implies that ASGD achieves an $\tilde{O} \left(\exp \left(\frac{-T}{\sqrt{\kappa}} \right) + \frac{\sigma^2}{T} \right)$ convergence rate. This improves over the non-accelerated $\tilde{O} \left(\exp \left(\frac{-T}{\kappa} \right) + \frac{\sigma^2}{T} \right)$ noise-adaptive rate obtained in~\cref{thm:sc-known} and~\citet{stich2019unified, khaled2020better, li2020second}. In the fully-deterministic setting ($\sigma = 0$),~\cref{thm:sc-accelerated} implies an $\tilde{O}(\exp(-T/\sqrt{\kappa}))$ convergence to the minimizer, matching the optimal rate in the deterministic setting~\citep{nesterov2013introductory}. In the general stochastic case, when $\sigma \neq 0$,~\citet{cohen2018acceleration, vaswani2019fast} use constant step-sizes ($\alphak = 1$), and prove convergence to a neighbourhood of the solution; whereas we show convergence to the minimizer at a rate governed by the $O(\sigma^2/T)$ term. To smoothly interpolate between the stochastic (batch size equal to $1$) and fully deterministic (batch size equal to $n$) setting, we generalize~\cref{eq:bounded-variance} to show an explicit dependence on the batch size (\cref{sec:additional-theory}). Comparing our result to that in~\citet{aybat2019universally}, we note that they also prove the accelerated noise-adaptive rate under the bounded variance of the stochastic gradients. In particular, they use a multi-stage algorithm and a dynamical systems perspective to prove their results. In contrast, our algorithm does not require multiple stages and is a natural stochastic extension of Nesterov's accelerated gradient. Furthermore, our proof uses the more standard estimate sequences technique. 

The result in~\cref{thm:sc-accelerated} requires the knowledge of both $\mu$ and $\Lmax$ and is thus not problem-adaptive. In the next section, we analyze the convergence of ASGD when it is used with offline estimates of $L$ and $\mu$.  

\subsection{Offline estimation of unknown smoothness \& strong-convexity}
\label{sec:acceleration-misspecified}
Similar to~\cref{sec:sc-unknown-decorrelated}, for simplicity, we will assume that $\gammak = \gamma =  \frac{1}{\tilde{L}}$ where without loss of generality, $\frac{1}{\tilde{L}} = \frac{\nul}{L}$. Similarly, we use $\tilde{\mu}$ as the offline estimate of the strong-convexity, and assume that $\tilde{\mu} = \num \mu$. We will only consider the case where we underestimate $\mu$, and hence $\num \leq 1$. This is the typical case in practice -- for example, while optimizing regularized convex loss functions in supervised learning (see~\cref{sec:experiments} for empirical results), $\tilde{\mu}$ is set to the regularization strength, and thus underestimates the true strong-convexity parameter. The following theorem (proved in~\cref{app:acceleration-thm-proof-miss}) analyzes the effect of misspecifying $L$, $\mu$ on the ASGD convergence.\\

\begin{restatable}{theorem}{restatescmisacc}
Under the same assumptions as~\cref{thm:sc-accelerated} and \red{(iv) $\nu = \nul \num\leq \kappa$}, ASGD (\cref{eq:extrapolation,eq:sgd-update}) with $\x_1 = y_1$, $\gammak = \frac{1}{\tilde{L}}=\frac{\nul}{\Lmax}$, $\alphak = \left(\frac{\beta}{T}\right)^{k/T}$, $\tilde{\mu}=\num \mu \leq \mu$, $\ak = \sqrt{\frac{\tilde{\mu}}{\tilde{L}}} \left(\frac{\beta}{T} \right)^{k/2T}=\sqrt{\frac{\nu}{ \kappa}} \left(\frac{\beta}{T} \right)^{k/2T}$ and $\bk = \frac{(1 - \akp) \, \akp \, \alpha}{\ak + \akp^2 \, \alpha}$ converges as,
\begin{align*}
& \Delta_{T+1} \leq 2c_3  \exp\left( - \frac{\sqrt{\min\{\nu,1\}}T}{\sqrt{\kappa}} \frac{\alpha}{\ln(\nicefrac{T}{\beta})}\right) \Delta_1
\\ 
& + \frac{2  c_3 (\ln(\nicefrac{T}{\beta}))^2}{e^2 \alpha^2  \mu T} \left[\sigma^2 + G^2 \min \{\frac{k_0}{T},1 \} \right] \max \{\frac{\nul}{\num},\nul^2 \},
\end{align*}
where $\Delta_k := \E[f(\x_k) - f^*]$, \red{$c_3 = \exp\left( \frac{1}{ \sqrt{\kappa}}\frac{2 \beta}{ \ln(\nicefrac{T}{\beta})}\right)$}, $[x]_+ = \max\{x,0\}$, $k_0:= \floor{T \frac{[\ln(\nul)]_+}{\ln(T/\beta)}}$ and $G = \max_{j \in [k_0]} \norm{\grad{y_j}}$. 
\label{thm:sc-misaccelerated}
\end{restatable}

The above theorem implies an \small{$\tilde O \left(\exp\left(\frac{-T\sqrt{\min\{\nu, 1\}}}{\sqrt{\kappa}}\right) + \left[\frac{\sigma^2 + G^2 [\ln(\nul)]_{+}}{T} \right] \max \{\frac{\nul}{\num},\nul^2 \} \right)$} \normalsize convergence to the minimizer. Observe that (i) when the problem-dependent parameters are known ($\num = \nul = 1$), we recover the rate of~\cref{thm:sc-accelerated}, (ii) if $\num = 1$, and we misestimate $L$, similar to SGD (\cref{thm:sc-unknown-decorrelated}), ASGD converges to the minimizer at an $O(1/T)$ rate, even in the deterministic setting (when $\sigma = 0$), (iii) if $\nul = 1$, underestimating $\mu$ matches the rate in~\cref{thm:sc-accelerated} upto (potentially large) constants, resulting in linear convergence when $\sigma = 0$, and (iv) \red{compared to~\cref{thm:sc-accelerated}, the decrease in the bias term is slowed down by an $O\left(\exp(\sqrt{\min\{\nul \num,1\}})\right)$ factor, whereas the decrease in the variance is slowed by an $O\left(\max \{\frac{\nul}{\num},\nul^2 \}\right)$ factor.}

In the next section, we design an SLS variant that ensures convergence to the minimizer while empirically controlling the misestimation for both SGD and ASGD.  
\section{Experiments}
\label{sec:experiments}
\begin{figure*}[!ht]
\centering     
\subfigure[Squared loss]{\label{fig:squared}\includegraphics[width = 0.98\textwidth]{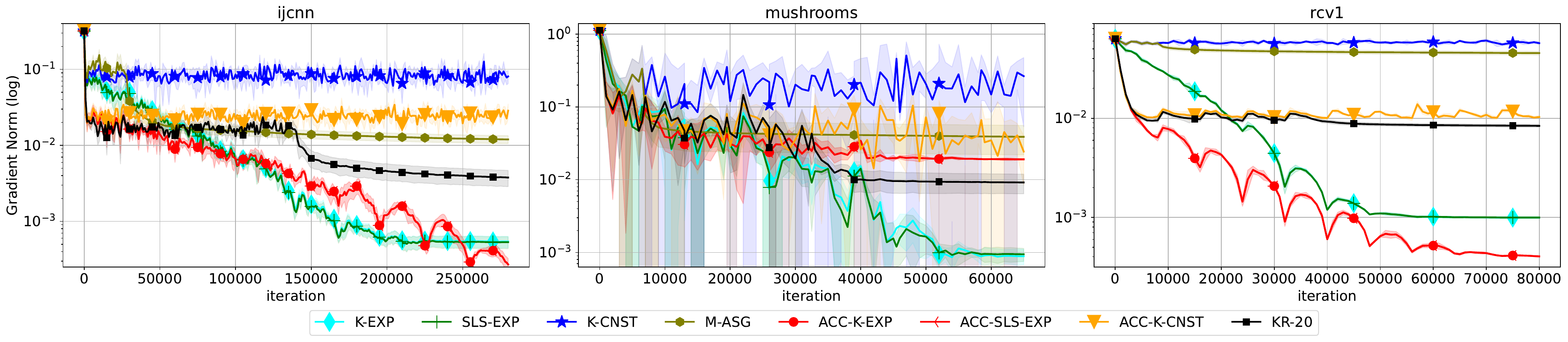}}
\subfigure[Logistic loss]{\label{fig:logistic}\includegraphics[width = 0.98\textwidth]{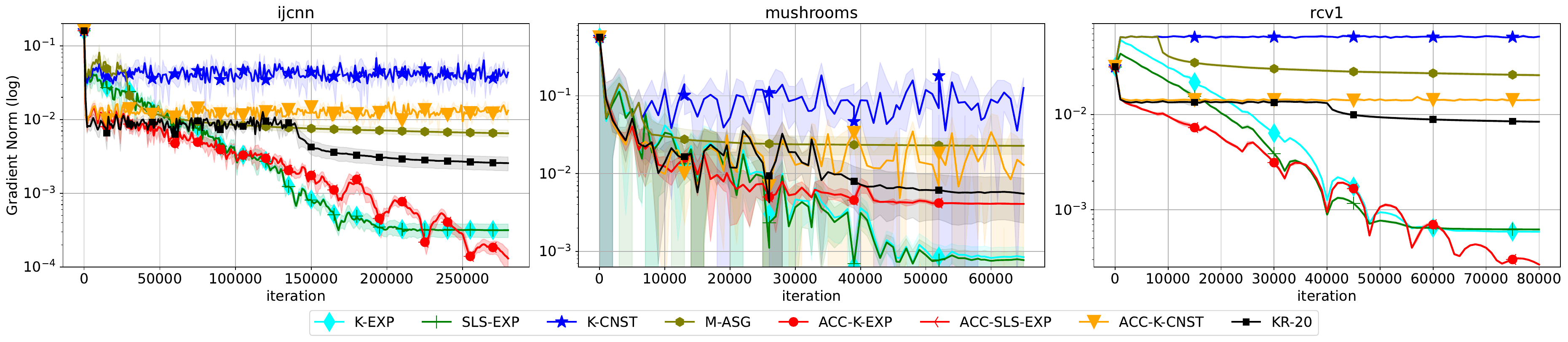}}
\vspace{-2ex}
\caption{Comparison for (a) squared loss and (b) logistic loss. Observe that exponentially decreasing step-sizes (i) result in more stable performance compared to using a constant step-size (for both SGD and ASGD) and (ii) consistently outperform the noise-adaptive methods in \texttt{KR-20} and \texttt{M-ASG}, and (iii) methods using the SLS in~\cref{eq:armijo-ls-conservative} match the performance of those with known smoothness.}
\label{fig:exps}
\end{figure*}
For comparing different step-size choices, we consider two common supervised learning losses -- squared loss for regression tasks and logistic loss for classification\footnote{The code to reproduce our experiments is available here: \url{https://github.com/R3za/expsls}}. With a linear model and an $\ell_2$ regularization equal to $\frac{\lambda}{2} \normsq{\x}$, both objectives are strongly-convex. 
We use three standard datasets from LIBSVM -- \textit{mushrooms}, \textit{ijcnn} and \textit{rcv1}, and use $\lambda = 0.01$. For each experiment, we consider $5$ independent runs and plot the average result and standard deviation. We use the (full) gradient norm as the performance measure and plot it against the number of gradient evaluations.

For each dataset, we fix $T = 10n$, use a batch-size of $1$ and compare the performance of the following optimization strategies: (i) the noise-adaptive ``constant and then decay step-size'' scheme in~\citet[Theorem 3]{khaled2020better} (denoted as \texttt{KR-20} in the plots). Specifically, for $b=\max\{\frac{2L^2}{\mu},2\rho L\}$, we use a constant step-size equal to $1/b$ when $T<b/\mu$ or $k < \ceil{T/2}$. Otherwise we set the step-size at iteration $k$ to be $ \frac{2}{\mu((2b/\mu) + k - \ceil{T/2})}$, (ii) constant step-size SGD with $\gammak = \frac{1}{\Lmax}$ and $\alphak = 1$ for all $k$ (denoted as \texttt{K-CNST} in the plots) (iii) SGD with an exponentially decreasing step-size with knowledge of smoothness~\citep{li2020second} i.e. $\gammak = \frac{1}{\Lmax}$ and $\alphak = \alpha^k$ for $\alpha = \left(\frac{\beta}{T}\right)^{1/T}$ (denoted as \texttt{K-EXP}), (iv) Accelerated SGD (ASGD) with a constant step-size ($\alphak = 1$ for all $k$)~\citep{vaswani2019fast, cohen2018acceleration} (denoted as \texttt{ACC-K-CNST}), (v) ASGD with exponentially decreasing step-sizes, (\cref{sec:acceleration}) denoted as \texttt{ACC-K-EXP} and (vi) Multistage ASGD in~\citet{aybat2019universally} (denoted as \texttt{M-ASG}) with parameters as in Corollary 3.8. Specifically we ensure that $T>2\sqrt{\kappa}$, set $T_1=T/C, T_k > 2^k[\sqrt{\kappa}\log(2^{(p+2)})], \alpha_1 = 1/L$ and $\alphak = 1/(2^{2k}L)$ where $p$ and $C$ are hyper-parameters.  

None of the above strategies are problem-adaptive, and all of them require the knowledge of the smoothness constant $\Lmax$. Additionally, the ASGD variants and $\texttt{M-ASG}$ require knowledge of $\mu$. If $x_i$ is the feature vector corresponding to example $i$, then we obtain theoretical upper-bounds on the smoothness and set $\Lmax = \max_i \normsq{x_i} + \lambda$ for the squared-loss and $\Lmax = \max_i \frac{1}{4} \normsq{x_i} + \lambda$ for the logistic loss. Similarly, we set $\mu = \lambda$ for both the squared and logistic loss. Note that this underestimates the true strong-convexity parameter, and is in line with~\cref{thm:sc-misaccelerated}. To set $p$ and $C$ for \texttt{M-ASG}, we use a grid search over $\{1,2,4\}$ and $\{2,10,100\}$ respectively. For each method, we plot the variant that results in the smallest gradient norm.  

Using a stochastic line-search (SLS) to estimate $\Lmax$ can result in convergence to the neighbourhood (\cref{sec:sc-unknown-correlated}) because of the correlations between $i_k$ and $\gammak$. To alleviate this, and still be problem-adaptive, we design a \emph{decorrelated conservative} variant of SLS: at iteration $k$ of SGD, we set $\gammak$ using a stochastic line-search on the \emph{previously sampled function} $i_{k-1}$ (we can use a randomly sampled $j_k$ as well). This ensures that there is no correlation between $i_k$ and computing $\gammak$. The overall procedure can be described as follows: starting from $\gammakp$ (the conservative aspect), with $\gamma_{0} = \gamma_{\max}$, find the largest step-size $\gammak$ that satisfies, for a random or previously sampled index $j_k$,
\begin{align}
f_{j_{k}}(\xk - \gammak \nabla f_{j_{k}}(\xk)) & \leq f_{j_{k}}(\xk) - c \gammak \normsq{\nabla f_{j_{k}}(\xk)}, 
\label{eq:armijo-ls-conservative}    
\end{align}
and update $\xk$ according to~\cref{eq:sgd}. The above procedure with $c = \nicefrac{1}{2}$ ensures that $\gammak \in \left[\min \{\gammakp, \nicefrac{1}{\Lmax}\}, \gammakp \right]$. \red{Since there is no correlation between $\gammak$ and $i_k$, we can treat $\gammak$ as an offline estimate of the smoothness, meaning that $\gammak = \nicefrac{\nu_k}{L}$ for some $\nu_k > 0$. Moreover, since we are using a conservative line-search, $\gammak \in \left[\min \{\gammakp, \nicefrac{1}{\Lmax}\}, \gammakp \right]$, meaning that $\nu_{k} \leq \nu_{k-1} \leq \nu_1$. Hence the maximum misspecification in the smoothness is given by $\nu_1 > 1$, which is governed by line-search in the first iteration. Given this, we can use a similar analysis as~\cref{thm:sc-unknown-decorrelated}, upper-bounding $\nu_k$ by $\nu_1$ for each $k$ and obtaining the corresponding result in terms of $\nu_1$. }

We use this variant of SLS with exponentially decreasing step-sizes for both SGD and ASGD, and denote the resulting variants as \texttt{SLS-EXP} and \texttt{ACC-SLS-EXP} respectively. We emphasize that this strategy is both noise-adaptive and problem-adaptive. 

From~\cref{fig:exps}, we observe that exponentially decreasing step-sizes (i) result in more stable performance compared to the constant step-size variants (for both SGD and ASGD) and (ii) consistently outperform the noise-adaptive methods, \texttt{KR-20} and \texttt{M-ASG}. We also observe that (iii) methods (\texttt{SLS-EXP} and \texttt{ACC-SLS-EXP}) using the SLS condition in~\cref{eq:armijo-ls-conservative} consistently match the performance of those with known smoothness (\texttt{K-EXP} and \texttt{ACC-K-EXP}). 

\section{Conclusion}
\label{sec:conclusion}

We used exponentially decreasing step-sizes to make SGD noise-adaptive, and considered two strategies for problem-adaptivity. Using upper and lower-bounds, we quantified the price of problem-adaptivity -- estimating the smoothness in an online fashion results in convergence to a neighbourhood of the solution, while an offline estimation results in a slower convergence to the minimizer. We then developed an accelerated variant of SGD (ASGD) and proved that it achieves the near-optimal convergence rate. We analyzed the effect of misspecifying the strong-convexity and smoothness parameters for ASGD. Finally, we empirically demonstrated the effectiveness of (A)SGD with exponential step-sizes coupled with a novel variant of SLS. 
\section{Acknowledgements}
\label{sec:ack}
We thank Chia-Yu Hsu for pointing out a mistake in a previous version of the paper. In particular, the previous version incorrectly claimed that the results in Section 4 hold under a general growth condition~\citep{bottou2018optimization} on the stochastic gradients. The corrected results only hold under the stronger though standard bounded variance assumption. 

We would like to thank Aaron Mishkin, Si Yi Meng, Yifan Sun and Frederik Kunstner for helpful feedback on the paper. Benjamin Dubois-Taine would like to acknowledge support from the European Research Council (grant SEQUOIA 724063) and funding by the French government under management of Agence Nationale de la Recherche as part of the “Investissements d’avenir” program, ANR-19-P3IA-0001 (PRAIRIE 3IA Institute).
\bibliographystyle{apalike}
\bibliography{ref}

\newpage
\appendix
\onecolumn

\newcommand{\appendixTitle}{%
\vbox{
    \centering
	\hrule height 4pt
	\vskip 0.2in
	{\LARGE \bf Supplementary material}
	\vskip 0.2in
	\hrule height 1pt 
}}

\newcommand{\makeappendixtable}[1]{%
~\\[.75em]\hskip 2em \begin{tabular}{@{}p{0.44\textwidth}p{0.2\textwidth}p{0.14\textwidth}@{}}
\toprule
#1\\
\bottomrule
\end{tabular}\\[.0em]
}

\section*{Organization of the Appendix}
\begin{itemize}

   \item[\ref{app:definitions}]
   \hyperref[app:definitions]{Definitions}
   
   \item[\ref{sec:additional-theory}]  \hyperref[sec:additional-theory]{Additional theoretical results}
   
   \item[\ref{app:ub-proofs}] 
   \hyperref[app:ub-proofs]{Upper-bound Proofs for~\cref{sec:adaptivity}}
   
   \item[\ref{app:lb-proofs}]  \hyperref[app:lb-proofs]{Lower-bound proofs for~\cref{sec:adaptivity}}
   
    \item[\ref{app:acceleration-proof}] \hyperref[app:acceleration-proof]{Proofs for~\cref{sec:acceleration}}
   
  \item[\ref{app:helper-lemmas}] \hyperref[app:helper-lemmas]{Helper Lemmas}   
  
\end{itemize}

\section{Definitions}
\label{app:definitions}
Our main assumptions are that each individual function $f_i$ is differentiable, has a finite minimum $f_i^*$, and is $L_i$-smooth, meaning that for all $v$ and $w$, 
\aligns{
    f_i(v) & \leq f_i(w) + \inner{\nabla f_i(w)}{v - w} + \frac{L_i}{2} \normsq{v - w},
    \tag{Individual Smoothness}
    \label{eq:individual-smoothness}
}
which also implies that $f$ is $\Lmax$-smooth, where $\Lmax$ is the maximum smoothness constant of the individual functions. A consequence of smoothness is the following bound on the norm of the stochastic gradients,
\aligns{
    \norm{\nabla f_i(\x)}^2 
    \leq
    2 \Lmax (f_i(\x) - f_i^*).
}
We also assume that each $f_i$ is convex, meaning that for all $v$ and $w$,
\aligns{
    f_i(v) &\geq f_i(w) - \lin{\nabla f_i(w), w-v},
    \tag{Convexity}
    \label{eq:individual-convexity}
    \\
}
Depending on the setting, we will also assume that $f$ is $\mu$ strongly-convex, meaning that for all $v$ and $\x$,
\aligns{
f(v) & \geq f(w) + \inner{\nabla f(w)}{v - w} + \frac{\mu}{2} \normsq{v - w},
\tag{Strong Convexity}
\label{eq:strong-convexity}
}

\section{Additional theoretical results}
\label{sec:additional-theory}
In this section, we relax the strong-convexity assumption to handle broader function classes in~\cref{app:relax} and prove results that help provide an explicit dependence on the mini-batch size (\cref{app:minibatch}) and in~\cref{app:polynomial} show that polynomially decreasing step-sizes cannot obtain the desired noise-adaptive rate. 

\subsection{Relaxing the assumptions}
\label{app:relax}
In this section, we extend our theoretical results to a richer class of functions - strongly quasar-convex functions~\citep{hinder2020near} in~\cref{sec:qsc-ub}, and (non-strongly) convex functions in~\cref{sec:convex-ub}. 

\subsubsection{Extension to strongly star-convex functions}
\label{sec:qsc-ub}
We consider the class of smooth, non-convex, but strongly star-convex functions~\citep{hinder2020near,gower2021sgd}, a subset of strongly quasar-convex functions. Quasar-convex functions are unimodal along lines that pass
through a global minimizer i.e. the function monotonically decreases along the line to the minimizer, and monotonically increases thereafter. In addition to this, strongly quasar-convex functions also have curvature near the global minimizer. Importantly, this property is satisfied for neural networks for common architectures and learning problems~\citep{lucas2021analyzing,kleinberg2018alternative}. 

Formally, a function is $(\zeta,\mu)$ strongly quasar-convex if it satisfies the following for all $\x$ and minimizers $\xopt$,
\begin{align}
f(\xopt) & \geq f(\x) + \frac{1}{\zeta} \langle \grad{\x},\xopt - \x \rangle + \frac{\mu}{2} \normsq{\x - \xopt}.
\label{eq:qsc-def}    
\end{align}
Strongly star-convex functions are a subset of this class of functions with $\zeta = 1$. If $\Lmax$ is known, it is straightforward to show that the results of~\cref{thm:sc-known} carry over to the strongly star-convex functions and we obtain the similar $O\left(\exp(-T/\kappa) + \frac{\sigma^2}{T} \right)$ rate. In the case when $\Lmax$ is not known, it was recently shown that SGD with a stochastic Polyak step-size~\citep{gower2021sgd} results in linear convergence to the minimizer on strongly star-convex functions under interpolation and achieves an $O \left(\exp(-T) + \gamma_{\max} \sigma^2 \right)$ convergence rate in general. The proposed stochastic Polyak step-size (SPS) does not require knowledge of $\Lmax$, and matches the rate achieved for strongly-convex functions~\citep{loizou2021stochastic}. However, SPS requires knowledge of $\fjopt$, which is usually zero for machine learning models under interpolation but difficult to get a handle on in the general case.

Consequently, we continue to use SLS to estimate the smoothness constant. Our proofs only use strong-convexity between $\x$ and a minimizer $\xopt$, and hence we can extend all our results from strongly-convex functions, to structured non-convex functions satisfying the strongly star-convexity property, matching the rates in~\cref{thm:sc-unknown-correlated} and~\cref{thm:sc-unknown-decorrelated}. Finally, we note that given knowledge of $\zeta$, there is no fundamental limitation in extending all our results to strongly quasar-convex functions. In the next section, we relax the strong-convexity assumption in a different way - by considering convex functions without curvature. 

\subsubsection{Handling (non-strongly)-convex functions}
\label{sec:convex-ub}
In this section, we analyze the behaviour of exponentially decreasing step-sizes on convex functions (without strong-convexity). As a starting point, we assume that $\Lmax$ is known, and the algorithm is only required to adapt to the noise $\sigma^2$. In the following theorem (proved in~\cref{app:c-known-proof}), we show that SGD with an exponentially decreasing step-size is not guaranteed to converge to the minimizer, but to a neighbourhood of the solution.   
\vspace{2ex}
\begin{restatable}{theorem}{restateck}
Assuming (i) convexity and (ii) $L_i$-smoothness of each $f_i$, SGD with step-size $\etak = \frac{1}{2 \Lmax} \, \alphak$ has the following convergence rate,
\begin{align}
\E [f(\bar{\x}_{T+1}) - f(\xopt)] &\leq \frac{2 \Lmax \, \normsq{\x_1 - \xopt}}{\sum_{k = 1}^{T} \alphak} + \sigma^2 \frac{\sum_{k = 1}^{T} \alphak^2}{\sum_{k = 1}^{T} \alphak} 
\label{eq:c-known-general}
\end{align}
where $\bar{\x}_{T+1} = \frac{\sum_{k = 1}^{T} \alphak \x_k}{\sum_{k= 1}^{T} \alphak}$. 
For $\alphak = \left[\frac{\beta}{T}\right]^{k/T}$, the convergence rate is given by, 
\begin{align*}
\E [f(\bar{\x}_{T+1}) - f(\xopt)] &\leq \frac{2 \Lmax \, \ln(T/\beta) \, \normsq{\x_1 - \xopt}}{\alpha T - 2 \beta} + \sigma^2 \frac{T}{T - \beta}
\end{align*}
\label{thm:c-known}
\end{restatable}
\vspace{2ex}
We thus see that even with the knowledge of $\Lmax$, SGD converges to a neighbourhood of the solution at an $O(1/T)$ rate. We contrast our result to AdaGrad~\citep{duchi2011adaptive, levy2018online} that adapts the step-sizes as the algorithm progresses (as opposed to using a predetermined sequence of step-sizes like in our case), is able to adapt to the noise, and achieves an $O\left(\frac{1}{T} + \frac{\sigma^2}{\sqrt{T}}\right)$ rate. 

In order to be noise-adaptive and match the AdaGrad rate, we can use~\cref{eq:c-known-general} to infer that a sufficient condition is for the $\alphak$-sequence to satisfy the following inequalities, (i) $\alphak \geq C_1 \, T$ and (ii) $\alphak^2 \leq C_2 \sqrt{T}$ where $C_1, C_2$ are constants. Unfortunately, in~\cref{lemma:polynomial-convex-bad,lemma:exponential-convex-bad}, we prove that it is not possible for \emph{any} polynomially or exponentially-decreasing sequence to satisfy these sufficient conditions. While we do not have a formal lower-bound in the convex case, it seems unlikely that these $\alphak$-sequences can result in the desired rate, and we conjecture a possible lower-bound. Finally, we note that to the best of our knowledge, the only predetermined (non-adaptive) step-size that achieves the AdaGrad rate is $\min \left\{\frac{1}{2 \Lmax}, \frac{1}{\sigma \sqrt{T}} \right\}$~\citep{ghadimi2012optimal}. We also conjecture a lower-bound that shows that there is no predetermined sequence of step-sizes (that does not use knowledge of $\sigma^2$) that is noise-adaptive and can achieve the $O\left(\frac{1}{T} + \frac{\sigma^2}{\sqrt{T}}\right)$ rate. 

\subsection{Dependence on the mini-batch size}
\label{app:minibatch}
In this section, we prove two results in order to explicitly model the dependence on the mini-batch size. We denote a mini-batch as $\gB$, its size as $B \in [1, n]$ and the corresponding mini-batch gradient as $\gradb{\x}= \frac{1}{B} \sum_{\fj \in \gB}  \nabla \fj(\x)$. The mini-batch gradient is also unbiased i.e. $\E_\gB[\gradb{\x}] = \grad{\x}$, implying that all the proofs remain unchanged. However, we need to use a different definition of $\sigma^2$ for both~\cref{sec:adaptivity} and~\cref{sec:acceleration}. We refine these quantities here, and show the explicit dependence on the mini-batch size. 

Note that for $\rho = 1$, the growth condition below recovers the bounded variance assumption (\cref{eq:bounded-variance}) used in~\cref{sec:acceleration}. 
\begin{thmbox}
\begin{lemmanew}
If 
$$
\E_{i} \normsq{\nabla \fj(\x)} \leq \rho \normsq{\grad{\x}} + \sigma^2,
$$
then, 
$$
\E_{\gB} \normsq{\gradb{\x)}} \leq \left( (\rho - 1) \frac{n - B}{n B} + 1 \right) \normsq{\grad{\x}} + \frac{n - B}{n B} \sigma^2. 
$$
\label{lemma:mb-growth}
\end{lemmanew}
\end{thmbox}
\begin{proof}
\begin{align*}
\E_{\gB} \normsq{\gradb{\x)}} &= \E_{\gB} \normsq{\gradb{\x} - \grad{\x} + \grad{\x}} = \E_{\gB} \normsq{\gradb{\x} - \grad{\x}} + \normsq{\grad{\x}} & \tag{Since $\E_\gB[\gradb{\x}] = \grad{\x}$} \\
\intertext{Since we are sampling the batch with replacement, using~\citep{lohr2019sampling},}
& \leq \frac{n - B}{n B} \left(E_i \normsq{\nabla \fj(\x)} - \normsq{\grad{\x}} \right) + \normsq{\grad{\x}} \\
& \leq \frac{n - B}{n B} \left((\rho - 1) \normsq{\grad{\x}} + \sigma^2 \right) + \normsq{\grad{\x}} & \tag{Using the growth condition} \\
\implies \E_{\gB} \normsq{\gradb{\x)}} & \leq \left( (\rho - 1) \frac{n - B}{n B} + 1 \right) \normsq{\grad{\x}} + \frac{n - B}{n B} \sigma^2. 
\end{align*}
\end{proof}

\begin{thmbox}
\begin{lemmanew}
If $$\sigma^2 := \E[\fj(\xopt) - \fjopt],$$ and each function $f_i$ is $\mu$ strongly-convex and $L$-smooth, then 
$$\sigma_B^2 := \E_{\gB}[f_{\gB}(\xopt) - f_{\gB}^*] \leq \frac{L}{\mu} \, \frac{n - B}{n B} \sigma^2.$$
\label{lemma:mb-sigma}
\end{lemmanew}
\end{thmbox}
\begin{proof}
\begin{align*}
\E_{\gB}[f_{\gB}(\xopt) - f_{\gB}^*] & \leq \frac{1}{2 \mu} \E_{\gB} \normsq{\gradb{\xopt)}}  \tag{By strong-convexity of $f_i$} \\
\intertext{Since we are sampling the batch with replacement, using~\citep{lohr2019sampling},}
& \leq \frac{1}{2 \mu} \frac{n - B}{n B} \E_i \normsq{\nabla \fj(\xopt)} \leq \frac{L}{\mu} \frac{n - B}{n B} \E[\fj(\xopt) - \fjopt] & \tag{By smoothness of $f_i$} \\
\implies \sigma_B^2 & \leq \frac{L}{\mu} \frac{n - B}{n B} \sigma^2.
\end{align*}
\end{proof}

 \subsection{Polynomially decaying stepsizes}
 \label{app:polynomial}
 In this section, we analyze polynomially decreasing step-sizes, namely when $\eta_k = \frac{\eta}{(k+1)^{\delta}}$ for some constants $\eta > 0$ and $0 \leq \delta \leq 1$. We argue that even with knowledge of the smoothness constant, these step-sizes fail to converge at the desired noise-adaptive rate even on simple quadratics. In particular, the next lemma shows that gradient descent (GD) applied to a strongly-convex quadratic with a polynomially decreasing step-size fails to obtain the usual linear rate of the form $O(\rho^{-T})$ for some $\rho < 1$.

 \begin{thmbox}
 \begin{lemmanew}
 When using $T$ iterations of GD to minimize a one-dimensional quadratic $f(w) = \frac{1}{2}(xw - y)^2$, setting $\eta_k = \frac{1}{L} \frac{1}{(k+1)^\delta}$ for some $0 < \delta \leq 1$ results in the following lower bounds.\\
 If $\delta = 1$,
 \begin{align*}
     w_{T+1} - w^* = \left( w_1 - w^*\right) \frac{1}{T+1}
 \end{align*}
 If $0 < \delta < 1$,  $w_1 - w^* > 0$ and $T$ is large enough,
 \begin{align*}
     w_{T+1} - w^* \geq (w_1 - w^*) \left(1 -  \frac{1}{2^\delta}\right)^{\floor{2^{1/\delta}} - 1} 4^{\frac{2\delta - 1}{1 - \delta}}4 ^{- \frac{(T+1)^{1- \delta}}{1- \delta}}
 \end{align*}
 \label{lemma:polynomial-lb-1}
 \end{lemmanew}
 \end{thmbox}
\begin{proof}
Observe that $w^* = \nicefrac{y}{x}$ and $L = x^2$. The GD iteration with $\eta_k = \frac{1}{L} \frac{1}{(k+1)^\delta}$ reads
\begin{align*}
    w_{k+1} = w_k - \frac{1}{L} \frac{1}{(k+1)^\delta} \left( x^2 w_k - xy\right) = w_k\left( 1 - \frac{1}{(k+1)^\delta}\right) + \frac{y}{x} \frac{1}{(k+1)^\delta} = w_k\left( 1 - \frac{1}{(k+1)^\delta}\right) + w^* \frac{1}{(k+1)^\delta}
\end{align*}
and thus
\begin{align*}
    w_{k+1} - w^* = \left( w_k - w^* \right) \left( 1-  \frac{1}{(k+1)^\delta}\right) \Rightarrow w_{T+1} - w^* = (w_1 - w^*) \prod_{k=1}^T \left( 1 - \frac{1}{(k+1)^\delta}\right)
\end{align*}
If $\delta = 1$,
\begin{align*}
    w_{T+1} - w^* = (w_1 - w^*) \prod_{k=1}^T \frac{k}{k+1} = (w_1 - w^*) \frac{1}{T+1}
\end{align*}
If $0 < \delta < 1$ and $w_1 - w^* > 0$,
\begin{align*}
    w_{T+1} - w^* = (w_1 - w^*) \prod_{k=1}^T \left( 1 - \frac{1}{(k+1)^\delta}\right) = (w_1 - w^*) \prod_{k=1}^T \left( 1 - \frac{2}{2(k+1)^\delta}\right)
\end{align*}
We wish to use the inequality $1 - \frac{2x}{2} \geq 2^{-2x}$ which is true for all $x \in [0, 1/2]$. In our case it holds for
\begin{align*}
    \frac{1}{(k+1)^\delta} \leq \frac{1}{2} \Rightarrow k \geq 2^{1/\delta} - 1
\end{align*}
Let $k_0 = \floor{2^{1/\delta}}$. Then for $ T \geq k_0$, 
\begin{align*}
    w_{T+1} - w^* = (w_1 - w^*) \prod_{k=1}^{k_0 - 1} \left( 1- \frac{1}{(k+1)^{\delta}}\right) \prod_{k=k_0}^{T} \left( 1- \frac{2}{2(k+1)^{\delta}}\right)
\end{align*}
Now, for $k \leq k_0 - 1$, we have that $\frac{1}{(k+1)^{\delta}} \leq \frac{1}{2^\delta}$ and thus
\begin{align*}
    \prod_{k=1}^{k_0 - 1} \left( 1- \frac{1}{(k+1)^{\delta}}\right) \geq \left( 1 - \frac{1}{2^\delta} \right)^{k_0 - 1} = \left(1 -  \frac{1}{2^\delta}\right)^{\floor{2^{1/\delta}} - 1}
\end{align*}
For $k \geq k_0$, we have $1 - \frac{2}{2(k+1)^\delta} \geq 2^{-2 \frac{1}{(k+1)^\delta}}$ and thus
\begin{align*}
    \prod_{k=k_0}^T \left( 1- \frac{2}{2(k+1)^\delta}\right) \geq 2^{-2 \sum_{k=k_0}^T \frac{1}{(k+1)^{\delta}}} = 2^{-2 \left(\sum_{k=1}^{T+1} \frac{1}{k^\delta} - \sum_{k=1}^{k_0} \frac{1}{k^\delta}\right)} \geq 2^{-2 \sum_{k=1}^{T+1} \frac{1}{k^\delta}}
\end{align*}
Using the bound in the proof of~\cref{lemma:polynomial-convex-bad}, we have
\begin{align*}
    \sum_{k=1}^{T+1} \frac{1}{k^\delta} &\leq 1 + \frac{1}{1- \delta} \left( (T+1)^{1 - \delta} - 1\right)
\end{align*}
Putting this together we have that
\begin{align*}
    2^{-2 \sum_{k=1}^{T+1} \frac{1}{k^\delta} } \geq 2^{-2 \left(1 + \frac{1}{1- \delta} \left( (T+1)^{1 - \delta} - 1\right)\right) } = \frac{4^{1/(1-\delta)}}{4}4^{- \frac{(T+1)^{1- \delta}}{1- \delta}} = 4^{\frac{2\delta - 1}{1 - \delta}}4 ^{- \frac{(T+1)^{1- \delta}}{1- \delta}}
\end{align*}
Putting everything together we get that
\begin{align*}
    w_{T+1 } - w^* \geq (w_1 - w^*) \left(1 -  \frac{1}{2^\delta}\right)^{\floor{2^{1/\delta}} - 1} 4^{\frac{2\delta - 1}{1 - \delta}}4 ^{- \frac{(T+1)^{1- \delta}}{1- \delta}}
\end{align*}
\end{proof}

 The next lemma shows that when $\delta = 0$, namely when the step-size is constant, SGD applied to the sum of two quadratics fails to converge to the minimizer.

 \begin{thmbox}
 \begin{lemmanew}
 When using SGD to minimize the sum $f(w) = \frac{f_1(w) + f_2(w)}{2}$ of two one-dimensional quadratics: $f_1(w) = \frac{1}{2}(w - 1)^2$ and $f_2(w) = \frac{1}{2}(2w + \nicefrac{1}{2})^2$ with a constant step-size $\eta = \frac{1}{\Lmax}$, the following holds: whenever $\abs{w_{k} - w^*}  < \nicefrac{1}{8}$, the next iterate satisfies $\abs{w_{k+1} - w^*} > \nicefrac{1}{8}$.
 \label{lemma:polynomial-lb-2}
 \end{lemmanew}
 \end{thmbox}
 \begin{proof}
 First observe that $w^* = 0$ and that $\Lmax = 4$. The updates then read
 \begin{align*}
     &\text{If } i_k = 1: \quad w_{k+1} = w_k - \eta (w_k - 1) = w_k(1 - \frac{1}{4}) + \frac{1}{4} = \frac{3}{4}w_k + \frac{1}{4}\\
     &\text{If } i_k = 2: \quad w_{k+1} = w_k - \eta 2(2w_k + \frac{1}{2}) = w_k(1 - \frac{4}{4}) - \frac{1}{4} =  - \frac{1}{4}
 \end{align*}
 Suppose that $\abs{w_k - w^*} = \abs{w_k} < \nicefrac{1}{8}$. We want to show that $\abs{w_{k+1}} > \nicefrac{1}{8}$. We can separate the analyses in three cases.\\
 If $w_k \in (-\nicefrac{1}{8}, 0)$ and $i_k = 1$ then
\begin{align*}
     w_{k+1} = \frac{3}{4} w_k + \frac{1}{4} > -\frac{3}{4}\times \frac{1}{8} + \frac{1}{4} = \frac{5}{32} > \frac{1}{8} 
 \end{align*}
 If $w_k \in (0, \nicefrac{1}{8})$ and $i_k = 1$ then
 \begin{align*}
     w_{k+1} = \frac{3}{4} w_k + \frac{1}{4} > \frac{1}{8} 
 \end{align*}
 If $i_k = 2$ then
 \begin{align*}
     w_{k+1} = -\frac{1}{4}  < - \frac{1}{8} 
 \end{align*}
 implying that in each case, $\abs{w_{k+1}} > \nicefrac{1}{8}$.
 \end{proof}

\section{Upper-bound Proofs for~\cref{sec:adaptivity}}
\label{app:ub-proofs}

\subsection{Proof of~\cref{thm:sc-known}}
\label{app:sc-known-proof}

\myquote{
\restatesck*
}
\begin{proof}
\begin{align*}
\normsq{\xkk-\xopt} &= \normsq{\xk-\etak \gradk{\xk} -\xopt} \\
&= \normsq{\xk-\xopt} - 2\etak\inner{\gradk{\xk}}{\xk-\xopt} +\etak^2\normsq{\gradk{\xk}}\\
& = \normsq{\xk-\xopt} - 2 \gammak \alphak \inner{\gradk{\xk}}{\xk-\xopt} + \gammak^2 \alphak^2 \normsq{\gradk{\xk}} \\
\normsq{\xkk-\xopt} & \leq \normsq{\xk-\xopt} - 2 \gammak \alphak \inner{\gradk{\xk}}{\xk-\xopt} + \gammak^2 \alphak^2 \, 2 \Lmax [\fk(\xk) - \fk^*] & \tag{Smoothness} \\
& = \normsq{\xk-\xopt} - \frac{2}{\Lmax} \alphak \inner{\gradk{\xk}}{\xk-\xopt} + \frac{2}{\Lmax} \alphak^2 \, [\fk(\xk) - \fk(\xopt)] + \frac{2}{\Lmax} \alphak^2 \, [\fk(\xopt) - \fk^*] & \tag{Since $\gammak = 1/\Lmax$.}\\
\intertext{Taking expectation w.r.t $i_k$,}
\E \normsq{\xkk - \xopt} & \leq \E \normsq{\xk-\xopt} - \frac{2}{\Lmax} \alphak \inner{\grad{\xk}}{\xk-\xopt} + \frac{2}{\Lmax} \alphak^2 \, [f(\xk) - f(\xopt)] + \frac{2}{\Lmax} \alphak^2 \sigma^2 \\
& \leq \E \normsq{\xk-\xopt} - \frac{2}{\Lmax} \alphak \inner{\grad{\xk}}{\xk-\xopt} + \frac{2}{\Lmax} \alphak \, [f(\xk) - f(\xopt)] + \frac{2}{\Lmax} \alphak^2 \sigma^2 & \tag{Since $\alphak \leq 1$}\\
\E \normsq{\xkk - \xopt} & \leq \left(1 - \frac{\mu \alphak}{\Lmax} \right) \E \normsq{\xk-\xopt} + \frac{2}{\Lmax} \alphak^2 \sigma^2 & \tag{By $\mu$-strong convexity of $f$} \\
\intertext{Unrolling the recursion starting from $\x_1$ and using the exponential step-sizes,}
\E \normsq{\x_{T+1} - \xopt} & \leq \normsq{\x_{1} - \xopt} \prod_{k = 1}^{T} \left(1 - \frac{\mu \alpha^k}{\Lmax} \right)  + \frac{2 \sigma^2}{\Lmax}  \sum_{k = 1}^{T} \left[\prod_{i = k+1}^{T} \alpha^{2k} \left(1 - \frac{\mu \alpha^i}{\Lmax} \right)  \right] \\
\intertext{Writing $\Delta_k = \E \normsq{\xk - \xopt}$}
\Delta_{T+1} & \leq \Delta_1 \exp\bigg( -\frac{\mu}{\Lmax}
    \underbrace{\sum_{k=1}^T \alpha^k}_{:= A} \bigg) + \frac{2 \sigma^2}{\Lmax} \underbrace{\sum_{k=1}^T \alpha^{2k} \exp\bigg( -\frac{\mu}{\Lmax} \sum_{i=k+1}^T \alpha^i}_{:=B_t} \bigg) 
\end{align*}
Using~\cref{lemma:A-bound} to lower-bound $A$, we obtain $A \geq \frac{\alpha T}{\ln(\nicefrac{T}{\beta})} - \frac{2\beta}{\ln(\nicefrac{T}{\beta})}$. The first term in the above expression can then be bounded as, 
\begin{align*}
    \Delta_1 \exp\left( -\frac{\mu}{\Lmax} A\right) &= \Delta_1 \, c_2  \exp\left( - \frac{T}{\kappa} \frac{\alpha}{\ln(\nicefrac{T}{\beta})}\right),
\end{align*}
where $\kappa = \frac{\Lmax}{\mu}$ and $c_2 = \exp\left( \frac{1}{\kappa} \, \frac{2\beta}{\ln(\nicefrac{T}{\beta})}\right)$.
Using~\cref{lemma:B-bound} to upper-bound $B_t$, we obtain $B_t \leq \frac{4 \kappa^2 c_2 (\ln(\nicefrac{T}{\beta}))^2}{e^2 \alpha^2 T}$, thus bounding the second term. Putting everything together, 
\begin{align*}
\Delta_{T+1} & \leq \Delta_1 \, c_2  \exp\left( - \frac{T}{\kappa} \frac{\alpha}{\ln(\nicefrac{T}{\beta})}\right) + \frac{8 \sigma^2 c_2 \kappa^2}{\Lmax e^2} \frac{(\ln(\nicefrac{T}{\beta}))^2}{\alpha^2 T}  
\end{align*}
\end{proof}

\subsection{Proof of~\cref{thm:sc-unknown-correlated}}
\label{app:sc-unknown-correlated-proof}
\myquote{
\restatescukc*
}
\begin{proof} 
\begin{align*}
\normsq{\xkk - \xopt} & \leq \normsq{\xk-\xopt} - 2 \gammak \alphak \inner{\gradk{\xk}}{\xk-\xopt} + \gammak \alphak^2 \left[\frac{\fk(\xk) - \fk^*}{c}\right] & \tag{By~\cref{lem:sls-bounds}} \\
\intertext{Setting $c = 1/2$,}
& = \normsq{\xk-\xopt} - 2 \gammak \alphak \inner{\gradk{\xk}}{\xk-\xopt} + 2 \gammak \alphak^2 \left[\fk(\xk) - \fk^*\right] \\
& = \normsq{\xk-\xopt} - 2 \gammak \alphak \inner{\gradk{\xk}}{\xk-\xopt} + 2 \gammak \alphak^2 \left[\fk(\xk) - \fk(\xopt) \right] + 2 \gammak \alphak^2 \left[\fk(\xopt) - \fk^*\right] \\
\intertext{Adding, subtracting $2 \gammak \alphak [\fk(\xk) - \fk(\xopt)]$,}
& = \normsq{\xk-\xopt} + 2 \gammak \alphak \left[-\inner{\gradk{\xk}}{\xk-\xopt} + [\fk(\xk) - \fk(\xopt)] \right] - 2 \gammak \alphak [\fk(\xk) - \fk(\xopt)] \\ & + 2 \gammak \alphak^2 \left[\fk(\xk) - \fk(\xopt) \right] + 2 \gammak \alphak^2 \left[\fk(\xopt) - \fk^*\right] \\
& \leq \normsq{\xk-\xopt} + 2 \gamma_{\min} \alphak \left[-\inner{\gradk{\xk}}{\xk-\xopt} + [\fk(\xk) - \fk(\xopt)] \right] \\ & - 2 \gammak (\alphak - \alphak^2) [\fk(\xk) - \fk(\xopt)]  + 2 \gamma_{\max} \alphak^2 \left[\fk(\xopt) - \fk^*\right] \\
\intertext{where we used convexity of $\fk$ to ensure that $-\inner{\gradk{\xk}}{\xk-\xopt} + [\fk(\xk) - \fk(\xopt)] \leq 0$. Taking expectation,}
\E \normsq{\xkk - \xopt} & \leq \normsq{\xk-\xopt} + 2 \gamma_{\min} \alphak \left[-\inner{\grad{\xk}}{\xk-\xopt} + [f(\xk) - f(\xopt)] \right] \\ & -  (\alphak - \alphak^2) \E \left[2 \gammak [\fk(\xk) - \fk(\xopt)]\right]  +  2 \gamma_{\max} \alphak^2 \sigma^2 \\
\E \normsq{\xk - \xopt} & \leq \left(1 - \alphak \gamma_{\min} \mu \right) \normsq{\xk-\xopt} -  (\alphak - \alphak^2) \E \left[2 \gammak [\fk(\xk) - \fk(\xopt)]\right] + 2 \gamma_{\max} \alphak^2 \sigma^2 \\
\end{align*}
Since $\alphak \leq 1$, and $\alphak - \alphak^2 \geq 0$, let us analyze $-\E [[\gammak [\fk(\xk) - \fk(\xopt)]]$. 
\begin{align*}
- \E [[\gammak [\fk(\xk) - \fk(\xopt)]] & = - \E [[\gammak [\fk(\xk) - \fk^*]] - \E [[\gammak [\fk^* - \fk(\xopt)]] \\
& \leq - \E [[\gamma_{\min} [\fk(\xk) - \fk^*]] - \E [[\gamma_{\max} [\fk^* - \fk(\xopt)]] \tag{$\gammak \leq \gamma_{\max}$}\\
& = - \E [[\gamma_{\min} [\fk(\xk) - \fk^*]] + \gamma_{\max} \sigma^2 \\
& = - \E [[\gamma_{\min} [\fk(\xk) - \fk(\xopt)]] - \E [[\gamma_{\min} [\fk(\xopt) - \fk^*]] + \gamma_{\max} \sigma^2 \\
& = -\gamma_{\min} [f(\xk) - f(\xopt)] - \gamma_{\min} \sigma^2 + \gamma_{\max} \sigma^2 \\
& \leq (\gamma_{\max} - \gamma_{\min}) \sigma^2
\end{align*}
Putting this relation back,
\begin{align*}
\E \normsq{\xk - \xopt} & \leq \left(1 - \alphak \gamma_{\min} \mu \right) \normsq{\xk-\xopt} 
+ 2 (\alphak - \alphak^2) \, (\gamma_{\max} - \gamma_{\min}) \sigma^2
+ 2 \gamma_{\max} \alphak^2 \sigma^2  \\
& \leq \left(1 - \alphak \gamma_{\min} \mu \right) \normsq{\xk-\xopt} + 2 \alphak \, (\gamma_{\max} - \gamma_{\min}) \sigma^2 + 2 \gamma_{\max} \alphak^2 \sigma^2.  
\end{align*}
Setting $\kappa' = \max\{\frac{L}{\mu}, \frac{1}{\mu \gamma_{\max}}\}$ we get that $1 - \alphak \gamma_{\min} \mu \leq 1 - \frac{1}{\kappa'}$. 
Writing $\Delta_k = \E \normsq{\xk - \xopt}$ and unrolling the recursion we get
\begin{align*}
    \Delta_{T+1} &\leq \left( \prod_{k=1}^T \left( 1 - \frac{1}{\kappa'} \alpha^k \right)  \right) \Delta_1 + 2 \gamma_{\max} \sigma^2 \sum_{k=1}^T \alpha^{2k} \prod_{i=t+1}^T \left( 1 - \frac{1}{\kappa'} \alpha^i \right) + 2 \, \sigma^2  \sum_{k=1}^T \alpha^{k} (\gamma_{\max} - \gamma_{\min}) \prod_{i=k+1}^T \left( 1 - \frac{1}{\kappa'} \alpha^i \right)\\
    &\leq \Delta_1 \exp\bigg( -\frac{1}{\kappa'}
    \underbrace{\sum_{k=1}^T \alpha^k}_{:= A} \bigg) + 2 \gamma_{\max}\sigma^2 \underbrace{\sum_{k=1}^T \alpha^{2k} \exp\bigg( -\frac{1}{\kappa'} 
    \sum_{i=k+1}^T \alpha^i}_{:=B_t} \bigg) \\ & + 2 \, \sigma^2 \, (\gamma_{\max} - \gamma_{\min}) \underbrace{\sum_{k=1}^T \alpha^{k} \exp\bigg( -\frac{1}{\kappa'} 
    \sum_{i=k+1}^T \alpha^i}_{:=C_t} \bigg) 
\end{align*}
Using~\cref{lemma:A-bound} to lower-bound $A$, we obtain $A \geq \frac{\alpha T}{\ln(\nicefrac{T}{\beta})} - \frac{2\beta}{\ln(\nicefrac{T}{\beta})}$. The first term in the above expression can then be bounded as, 
\begin{align*}
    \Delta_1 \exp\left( -\frac{1}{\kappa'} A\right) &\leq \Delta_1 \, c_1  \exp\left( - \frac{T}{\kappa'} \frac{\alpha}{\ln(\nicefrac{T}{\beta})}\right),
\end{align*}
where $c_1 = \exp\left( \frac{1}{\kappa'} \, \frac{2\beta}{\ln(\nicefrac{T}{\beta})}\right)$.
Using~\cref{lemma:B-bound} to upper-bound $B_t$, we obtain $B_t \leq \frac{4 (\kappa')^2 c_1 (\ln(\nicefrac{T}{\beta}))^2}{e^2 \alpha^2 T}$, thus bounding the second term. Using~\cref{lemma:C-bound} to upper-bound $C_t$, we obtain $C_t \leq c_1 \frac{\kappa' \ln (T / \beta)}{e \alpha}$, thus bounding the third term.
Finally, by~\cref{lem:sls-bounds} we have that $\gamma_{\min} \geq \min\left\{\gamma_{\max}, \frac{1}{L}\right\}$.

Putting everything together, 
\begin{align*}
\Delta_{T+1} &\leq \Delta_1 \, c_1 \exp\left( - \frac{T}{\kappa'} \frac{\alpha}{\ln(\nicefrac{T}{\beta})}\right) + \frac{8 \sigma^2 c_1 (\kappa')^2 \gamma_{\max}}{e^2} \frac{(\ln(\nicefrac{T}{\beta}))^2}{\alpha^2 T} + \frac{2 c_1 \sigma^2 \kappa' \ln (T / \beta)}{e \alpha} \, \left(\gamma_{\max} -  \min \left\{\gamma_{\max}, \frac{1}{\Lmax} \right\} \right)
\end{align*}
\end{proof}

\clearpage
\subsection{Proof of~\cref{thm:sc-unknown-decorrelated}}
\label{app:sc-unknown-decorrelated-proof}
\myquote{
\restatescukdc*
}
\begin{proof}
Following the steps from the proof of~\cref{thm:sc-known},
\begin{align*}
\normsq{\xkk - \xopt} & \leq \normsq{\xk-\xopt} - 2 \gammak \alphak \inner{\gradk{\xk}}{\xk-\xopt} + 2L \gammak^2 \alphak^2 \, [\fk(\xk) - \fk(\xopt)] + 2L \gammak^2 \alphak^2 \, [\fk(\xopt) - \fk^*] \\
\intertext{Taking expectation wrt $i_k$, and since both $\gammak$ and $\alphak$ are independent of $i_k$,}
\E \normsq{\xkk - \xopt} & \leq \normsq{\xk-\xopt} - 2 \gammak \alphak \inner{\grad{\xk}}{\xk-\xopt} + 2L \gammak^2 \alphak^2 \, [f(\xk) - f^*] + 2L \gammak^2 \alphak^2 \, \sigma^2 \\
\E \normsq{\xkk - \xopt} & \leq \left(1 - \mu \gammak \alphak \right) \normsq{\xk-\xopt} +  2L \gammak^2 \alphak^2 \, \sigma^2 + [f(\xk) - f^*] \, (2L \gammak^2 \alphak^2 - 2 \gammak \alphak) & \tag{By strong convexity} 
\end{align*}
Let us separately consider the $\nu \leq 1$ and $\nu > 1$ case. 

For the $\nu \leq 1$ case, $(2L \gammak^2 \alphak^2 - 2 \gammak \alphak) = \frac{2 \nu^2 \alphak^2}{L} - \frac{2 \nu \alphak}{L} \leq \frac{2 \nu \alphak}{L} - \frac{2 \nu \alphak}{L} = 0$. Hence, the above equation can be simplified as:
\begin{align*}
\E \normsq{\xkk - \xopt} & \leq \left(1 - \frac{\mu \nu \alphak}{L} \right) \normsq{\xk-\xopt} +  \frac{2 \nu^2 \alphak^2}{L} \, \sigma^2 
\end{align*}
Proceeding in the same way as the proof of~\cref{thm:sc-known}, define $\Delta_k = \E\normsq{\x_{k} - \xopt}$ and unroll the recursion, 
\begin{align*}
    \Delta_{T+1} &\leq \Delta_1 \prod_{k=1}^{T} (1 -  \frac{\mu \nu}{\Lmax} \alpha_k) + \left(\frac{2 \nu^2 \sigma^2}{\Lmax}  \right) \sum_{k=1}^{T} \alpha_k^2 \prod_{i= k+1}^{T} ( 1 - \frac{\mu \nu}{\Lmax} \alpha_i)
\end{align*}
Bounding the first term similar to~\cref{lemma:A-bound}, 
\begin{align*}
    \prod_{k=1}^{T} (1 -  \frac{\mu \nu}{\Lmax} \alpha_k) &\leq \exp \left( - \frac{\mu \nu}{\Lmax} \frac{\alpha - \alpha^{T+1}}{ 1- \alpha}\right) \leq \exp \left( - \frac{\mu \nu}{\Lmax} \frac{\alpha T - 2 \beta}{\ln{(T/\beta)}}\right) = \exp \left( - \frac{\nu T}{\kappa} \, \frac{\alpha}{\ln{(T/\beta)}} \right) \, \exp\left(\frac{\nu}{\kappa} \frac{2 \beta}{\ln{(T/\beta)}}\right)
\end{align*}
Bounding the second term similar to~\cref{lemma:B-bound},
\begin{align*}
    \sum_{k=1}^{T} \alpha_k^2 \prod_{i= k+1}^{T} ( 1 - \frac{\mu \nu}{\Lmax} \alpha_i) &\leq \sum_{k=1}^{T} \alpha_k^2 \exp\left( - \frac{\mu \nu}{\Lmax} \sum_{i = k+1}^{T} \alpha^i \right)\\
    &= \sum_{k=1}^{T} \alpha_k^2 \exp\left( - \frac{ \nu}{\kappa} \frac{\alpha^{k+1} - \alpha^{T+1}}{ 1 - \alpha} \right)\\
    &= \exp\left(\frac{\nu \alpha^{T+1}}{\kappa ( 1- \alpha)}\right) \sum_{k=1}^{T} \alpha_k^2 \exp\left( -\frac{\nu\alpha^{k+1}}{\kappa (1- \alpha)}\right)\\
    &\leq \exp\left(\frac{\nu \alpha^{T+1}}{\kappa ( 1- \alpha)}\right) \sum_{k=1}^{T} \alpha_k^2 \left( \frac{2(1- \alpha)\kappa}{\nu e\alpha^{k+1}}\right)^2\\
    &= \exp\left(\frac{\nu \alpha^{T+1}}{\kappa ( 1- \alpha)}\right) \frac{4(1-\alpha)^2 \kappa^2}{\nu^2e^2\alpha^2 } T\\
    &\leq \exp\left(\frac{\nu \alpha^{T+1}}{\kappa ( 1- \alpha)}\right) \frac{4 \kappa^2}{\nu^2e^2\alpha^2 } \frac{\ln(T/\beta)^2}{T} \\
    &\leq \exp\left(\frac{\nu}{\kappa} \frac{2\beta}{\ln(\nicefrac{T}{\beta})} \right) \frac{4 \kappa^2}{\nu^2e^2\alpha^2 } \frac{\ln(T/\beta)^2}{T}
\end{align*}
Putting everything together, we obtain that, 
\begin{align*}
\Delta_{T+1} &\leq \Delta_1 \exp \left( - \frac{\nu T}{\kappa} \, \frac{\alpha}{\ln{(T/\beta)}} \right) \, \exp\left(\frac{\nu}{\kappa} \frac{2 \beta}{\ln{(T/\beta)}}\right) + \frac{2 \sigma^2}{L} \exp\left(\frac{\nu}{\kappa} \frac{2\beta}{\ln(\nicefrac{T}{\beta})} \right)  \frac{4 \kappa^2}{e^2\alpha^2 } \frac{\ln(T/\beta)^2}{T} \\
\implies \Delta_{T+1} & \leq \Delta_1 c_2 \exp \left( - \frac{\nu T}{\kappa} \, \frac{\alpha}{\ln{(T/\beta)}} \right) \, + \frac{2 \sigma^2 c_2}{L} \frac{4 \kappa^2}{e^2\alpha^2 } \frac{\ln(T/\beta)^2}{T} \tag{Since $\nu \leq 1$}\\
\end{align*}

For the $\nu > 1$ case, 
\begin{align*}
\E \normsq{\xkk - \xopt} & \leq \left(1 - \frac{\mu \nu \alphak}{L} \right) \normsq{\xk-\xopt} +  \frac{2 \nu^2 \alphak^2 \, \sigma^2}{L} + [f(\xk) - f^*] \, (2L \gammak^2 \alphak^2 - 2 \gammak \alphak) \\
& \leq \left(1 - \frac{\mu  \alphak}{L} \right) \normsq{\xk-\xopt} +  \frac{2 \nu^2 \alphak^2 \, \sigma^2}{L} + [f(\xk) - f^*] \, \left(\frac{2 \nu^2 \alphak^2 }{L} - \frac{2 \nu \alphak}{L} \right) \tag{Since $\nu > 1$} 
\end{align*}
For the last term to be negative, we require $\alphak \leq \frac{1}{\nu}$. By definition of $\alphak$, this will happen after $k \geq k_0 := T \frac{\ln(\nu)}{\ln(\nicefrac{T}{\beta})}$ iterations. However, until $k_0$ iterations, we observe that $(2L \gammak^2 \alphak^2 - 2 \gammak \alphak) \leq \frac{2 \nu \, (\nu - 1)}{\Lmax} \alphak^2$. 

For the $k < k_0$ regime, 
\begin{align*}
\E \normsq{\xkk - \xopt} & \leq \left(1 - \frac{\mu \alphak}{L} \right) \normsq{\xk-\xopt} +  2L \gammak^2 \alphak^2 \, \sigma^2 + \max_{j \in [k_0]} \{ f(\x_j) - f^* \} \, \frac{2 \nu \, (\nu - 1)}{\Lmax} \alphak^2  
\end{align*}
Writing $\Delta_k = \E\normsq{\x_{k} - \xopt}$, and unrolling the recursion for the first $k_0$ iterations we get
\begin{align*}
\Delta_{k_0} &\leq \Delta_1 \prod_{k=1}^{k_0 - 1} (1 -  \frac{\mu}{\Lmax} \alpha_k) + \left(\underbrace{ 2 \frac{\nu^2}{\Lmax} \sigma^2 + \max_{j \in [k_0]} \{ f(w_j) - f^* \}  \frac{2 \nu (\nu - 1)}{\Lmax}}_{:= c_5} \right) \sum_{k=1}^{k_0 - 1} \alpha_k^2 \prod_{i= k+1}^{k_0 - 1} ( 1 - \frac{\mu}{\Lmax} \alpha_i)
\end{align*}
Bounding the first term similar to~\cref{lemma:A-bound}, 
\begin{align*}
    \prod_{k=1}^{k_0 - 1} (1 -  \frac{\mu}{\Lmax} \alpha_k) &\leq \exp \left( - \frac{\mu}{\Lmax} \frac{\alpha - \alpha^{k_0}}{ 1- \alpha}\right)
\end{align*}
Bounding the second term similar to~\cref{lemma:B-bound},
\begin{align*}
    \sum_{k=1}^{k_0 - 1} \alpha_k^2 \prod_{i= k+1}^{k_0 - 1} ( 1 - \frac{\mu }{\Lmax} \alpha_i) &\leq \sum_{k=1}^{k_0 - 1} \alpha_k^2 \exp\left( - \frac{\mu}{\Lmax} \sum_{i = k+1}^{k_0 - 1} \alpha^i \right)\\
    &= \sum_{k=1}^{k_0 - 1} \alpha_k^2 \exp\left( - \frac{1}{\kappa} \frac{\alpha^{k+1} - \alpha^{k_0}}{ 1 - \alpha} \right)\\
    &= \exp\left(\frac{\alpha^{k_0}}{\kappa ( 1- \alpha)}\right) \sum_{k=1}^{k_0 - 1} \alpha_k^2 \exp\left( -\frac{\alpha^{k+1}}{\kappa (1- \alpha)}\right)\\
    &\leq \exp\left(\frac{\alpha^{k_0}}{\kappa ( 1- \alpha)}\right) \sum_{k=1}^{k_0 - 1} \alpha_k^2 \left( \frac{2(1- \alpha)\kappa}{ e\alpha^{k+1}}\right)^2\\
    &\leq \exp\left(\frac{\alpha^{k_0}}{\kappa ( 1- \alpha)}\right) \frac{4(1-\alpha)^2 \kappa^2}{e^2\alpha^2 } k_0\\
    &\leq \exp\left(\frac{\alpha^{k_0}}{\kappa ( 1- \alpha)}\right) \frac{4 \kappa^2}{e^2\alpha^2 } \frac{k_0 \ln(T/\beta)^2}{T^2}
\end{align*}
Putting everything together, we obtain, 
\begin{align*}
\Delta_{k_0} &\leq \Delta_1 \exp \left( - \frac{\mu}{\Lmax} \frac{\alpha - \alpha^{k_0}}{ 1- \alpha}\right) + c_5 \exp\left(\frac{\alpha^{k_0}}{\kappa ( 1- \alpha)}\right) \frac{4 \kappa^2}{e^2\alpha^2 } \frac{k_0 \ln(T/\beta)^2}{T^2}   
\end{align*}
Now let us consider the regime $k \geq k_0$ where $\alphak \leq \frac{1}{\nu}$, so that we have
\begin{align*}
\E \normsq{\xkk - \xopt} & \leq \left(1 - \frac{\mu \alphak}{L} \right) \normsq{\xk-\xopt} +  \frac{2 \nu^2 \sigma^2}{\Lmax} \alphak^2 \\
\intertext{Writing $\Delta_k = \E\normsq{\xk - \xopt}$, and unrolling the recursion from $k = k_0$ to $T$,}
\Delta_{T+1} &\leq \Delta_{k_0} \prod_{k=k_0}^T (1 - \frac{\mu}{\Lmax} \alpha_k) + \frac{2 \nu^2 \sigma^2}{\Lmax} \sum_{k=k_0}^{T} \alpha_k^2 \prod_{i=k+1}^T ( 1- \frac{\mu}{\Lmax} \alpha_i)
\end{align*}
Bounding the first term similar to~\cref{lemma:A-bound}, 
\begin{align*}
    \prod_{k=k_0}^T (1 - \frac{\mu}{\Lmax} \alpha_k) &\leq \exp\left( - \frac{\mu}{\Lmax} \sum_{k= k_0}^T \alpha_k\right) = \exp \left( \frac{-\mu }{\Lmax} \frac{\alpha^{k_0} - \alpha^{T + 1}}{1 - \alpha}\right)
\end{align*}
Bounding the second term similar to~\cref{lemma:B-bound},
\begin{align*}
\sum_{k=k_0}^{T} \alpha_k^2 \prod_{i= k+1}^{T} ( 1 - \frac{\mu}{\Lmax} \alpha_i) &\leq \sum_{k=k_0}^{T} \alpha_k^2 \exp\left( - \frac{\mu}{\Lmax} \sum_{i = k+1}^{T} \alpha^i \right)\\
&= \sum_{k=k_0}^{T} \alpha_k^2 \exp\left( - \frac{1}{\kappa} \frac{\alpha^{k+1} - \alpha^{T+1}}{ 1 - \alpha} \right)\\
&= \exp\left(\frac{\alpha^{T+1}}{\kappa ( 1- \alpha)}\right) \sum_{k=k_0}^{T} \alpha_k^2 \exp\left( -\frac{\alpha^{k+1}}{\kappa (1- \alpha)}\right)\\
&\leq \exp\left(\frac{\alpha^{T+1}}{\kappa ( 1- \alpha)}\right) \sum_{k=k_0}^{T} \alpha_k^2 \left( \frac{2(1- \alpha)\kappa}{e\alpha^{k+1}}\right)^2\\
&= \exp\left(\frac{\alpha^{T+1}}{\kappa ( 1- \alpha)}\right) \frac{4(1-\alpha)^2 \kappa^2}{e^2\alpha^2 } (T - k_0 + 1)\\
&\leq \exp\left(\frac{\alpha^{T+1}}{\kappa ( 1- \alpha)}\right) \frac{4 \kappa^2}{e^2\alpha^2 } \frac{(T - k_0 + 1) \ln(T/\beta)^2}{T^2}  
\end{align*}
Putting everything together, 
\begin{align*}
\Delta_{T+1} &\leq \Delta_{k_0} \exp \left( \frac{-\mu}{\Lmax} \frac{\alpha^{k_0} - \alpha^{T + 1}}{1 - \alpha}\right) + \frac{2 \nu^2 \sigma^2}{\Lmax} \exp\left(\frac{\alpha^{T+1}}{ \kappa ( 1- \alpha)}\right) \frac{4 \kappa^2}{e^2\alpha^2 } \frac{(T - k_0 + 1) \ln(T/\beta)^2}{T^2}  
\end{align*}
Combining the above bounds for the two regimes, we get, 
\begin{align*}
\Delta_{T+1} \leq &\exp \left( \frac{-\mu}{\Lmax} \frac{\alpha^{k_0} - \alpha^{T + 1}}{1 - \alpha}\right) \left( \Delta_1 \exp \left( - \frac{\mu}{\Lmax} \frac{\alpha - \alpha^{k_0}}{ 1- \alpha}\right) + c_5 \exp\left(\frac{ \mu \alpha^{k_0}}{\Lmax( 1- \alpha)}\right) \frac{4 \kappa^2}{e^2\alpha^2 } \frac{k_0 \ln(T/\beta)^2}{T^2} \right)  \\
& + \frac{2 \nu^2 \sigma^2}{\Lmax}\exp\left(\frac{\alpha^{T+1}}{\kappa ( 1- \alpha)}\right) \frac{4 \kappa^2}{e^2\alpha^2 } \frac{(T - k_0 + 1) \ln(T/\beta)^2}{T^2} \\
& = \Delta_1 \exp \left( \frac{-\mu}{\Lmax} \frac{\alpha - \alpha^{T + 1}}{1 - \alpha} \right) + c_5 \exp \left( \frac{\mu}{\Lmax} \frac{\alpha^{T + 1}}{1 - \alpha}\right) \, \frac{4 \kappa^2}{2e^2\alpha^2 } \frac{k_0 \ln(T/\beta)^2}{T^2} \\
& + \frac{2\nu^2 \sigma^2}{\Lmax}\exp\left(\frac{\alpha^{T+1}}{\kappa ( 1- \alpha)}\right) \frac{4 \kappa^2}{e^2\alpha^2 } \frac{(T - k_0 + 1) \ln(T/\beta)^2}{T^2} 
\end{align*}
Using~\cref{lemma:A-bound} to bound the first term, and noting that $\frac{\alpha^{T+1}}{1 - \alpha} \leq \frac{2\beta}{\ln(\nicefrac{T}{\beta})}$
\begin{align*}
\Delta_{T+1} & \leq \Delta_1 \, c_2  \exp\left( -\frac{T}{\kappa} \frac{\alpha}{\ln(\nicefrac{T}{\beta})}\right) + c_5 \, \frac{4 c_2 \kappa^2}{e^2\alpha^2 } \frac{k_0 \ln(T/\beta)^2}{T^2} + \frac{2\nu^2 \sigma^2}{\Lmax} \frac{4 c_2 \kappa^2}{e^2\alpha^2 } \frac{(T - k_0 + 1) \ln(T/\beta)^2}{T^2} 
\intertext{where $\kappa = \frac{\Lmax}{\mu}$ and $c_2 = \exp\left( \frac{1}{\kappa} \, \frac{2\beta}{\ln(\nicefrac{T}{\beta})}\right)$.}
\end{align*}

Putting in the value of $c_5$ and $k_0$, and rearranging, we get
\begin{align*}
\Delta_{T+1} & \leq \Delta_1 \, c_2  \exp\left( - \frac{T}{\kappa} \frac{\alpha}{\ln(\nicefrac{T}{\beta})}\right) + \frac{4\nu^2 \sigma^2}{\Lmax \, T} \frac{4 c_2 \kappa^2 \ln(T/\beta)^2}{e^2\alpha^2} \\ & + \left[\max_{j \in [k_0]} \{ f(w_j) - f^* \}  \frac{2 \nu (\nu - 1)}{\Lmax}\right] \,\frac{4 c_2 \kappa^2}{\nu^2e^2\alpha^2 } \frac{[\ln(\nu)]_{+} \ln(T/\beta)}{T} 
\end{align*}
Combining the statements from $\nu \leq 1$ and $\nu > 1$ gives us the theorem statement. 

\end{proof}

\clearpage
\subsection{Proof of~\cref{thm:c-known}}
\label{app:c-known-proof}
\myquote{
\restateck*
}
\begin{proof}
Following the proof of~\cref{thm:sc-known}, 
\begin{align*}
\normsq{\xkk-\xopt} & \leq \normsq{\xk-\xopt} - 2 \gammak \alphak \inner{\gradk{\xk}}{\xk-\xopt} + 2L \gammak^2 \alphak^2 \, [\fk(\xk) - \fk(\xopt)] \\ & + \frac{2}{\Lmax} \gammak^2 \alphak^2 \, [\fk(\xk) - \fk^*] \\
\normsq{\xkk-\xopt} & \leq \normsq{\xk-\xopt} - \frac{\alphak}{\Lmax}  \inner{\gradk{\xk}}{\xk-\xopt} + \frac{\alphak^2}{2L} \, [\fk(\xk) - \fk(\xopt)] + \frac{\alphak^2}{2L}  \, [\fk(\xk) - \fk^*]  & \tag{$\gammak = \frac{1}{2L}$ for all $k$.} \\
& \leq \normsq{\xk-\xopt} - \frac{ \alphak}{\Lmax} [\fk(\xk) - \fk(\xopt)] + \frac{\alphak^2}{2L} \, [\fk(\xk) - \fk(\xopt)] + \frac{\alphak^2}{2L}  \, [\fk(\xk) - \fk^*] & \tag{By convexity} \\
\intertext{Taking expectation,}
\E \normsq{\xkk-\xopt} & \leq \normsq{\xk-\xopt} - \frac{ \alphak}{\Lmax} [f(\xk) - f(\xopt)] + \frac{\alphak^2}{2L} \, [f(\xk) - f(\xopt)] + \frac{\alphak^2}{2 \Lmax}  \sigma^2 \\
& \leq \normsq{\xk-\xopt} - \frac{\alphak}{2 \Lmax} [f(\xk) - f(\xopt)] + \frac{\alphak^2}{2L}  \sigma^2 & \tag{Since $f(\xk) - f(\xopt) \geq 0$ and $\alphak \leq 1$} \\
\intertext{Rearranging and summing from $k = 1$ to $T$,}
\sum_{k = 1}^{T} \alphak [f(\xk) - f(\xopt)] & \leq 2 \Lmax \normsq{\x_1 - \xopt} + \sigma^2 \sum_{k = 1}^{T} \alphak^2\\
\intertext{By averaging and using Jensen. Denote $\bar{\x}_{T+1} = \frac{\sum_{k = 1}^{T} \alphak \xk}{\sum_{k = 1}^{T} \alphak}$,}
\E [f(\bar{\x}_{T+1}) - f(\xopt)] &\leq \frac{2L \, \normsq{\x_1 - \xopt}}{\sum_{k = 1}^{T} \alphak} + \sigma^2 \frac{\sum_{k = 1}^{T} \alphak^2}{\sum_{k = 1}^{T} \alphak} 
\end{align*}
Next, we bound $\sum_{k = 1}^{T} \alphak$ and $\sum_{k = 1}^{T} \alphak^2$ for the exponentially-decreasing $\alphak$ sequence, when $\alphak = \left[\frac{\beta}{T}\right]^{k/T}$. From~\cref{lemma:A-bound}, we know that, 
\begin{align*}
\sum_{k = 1}^{T} \alphak \geq \frac{\alpha T}{\ln(\nicefrac{T}{\beta})} - \frac{2\beta}{\ln(\nicefrac{T}{\beta})}.
\end{align*}
Bounding the ratio $\frac{\sum_{k = 1}^{T} \alphak^2}{\sum_{k = 1}^{T} \alphak} = \frac{\sum_{k = 1}^{T} \alpha^{2k}}{\sum_{k = 1}^{T} \alpha^{k}}$ where $\alpha = \left[\frac{\beta}{T}\right]^{1/T}$,
\begin{align*}
\frac{\sum_{k = 1}^{T} \alpha^{2k}}{\sum_{k = 1}^{T} \alpha^{k}} & \leq \frac{\alpha^2}{1 - \alpha^2} \, \frac{1 - \alpha}{\alpha - \alpha^{T+1}} \\
& = \frac{\alpha}{1 + \alpha} \frac{1}{1 - \alpha^T} \leq \frac{1}{1 - \alpha^T} = \frac{T }{T - \beta}\\
\end{align*}
Putting everything together, 
\begin{align*}
\E [f(\bar{\x}_{T+1}) - f(\xopt)] &\leq \frac{2L \, \ln(T/\beta) \, \normsq{\x_1 - \xopt}}{\alpha T - 2 \beta} + \sigma^2 \frac{T }{T- \beta}
\end{align*}
\end{proof}

\subsection{Additional lemmas for upper-bound proofs}
\label{app:ub-lemmas}
\begin{thmbox}
\begin{lemmanew}
\label{lemma:A-bound}
\begin{align*}
A := \sum_{t=1}^T \alpha^t & \geq \frac{\alpha T}{\ln(\nicefrac{T}{\beta})} - \frac{2\beta}{\ln(\nicefrac{T}{\beta})}     
\end{align*}
\end{lemmanew}
\end{thmbox}
\begin{proof}
\begin{align*}
    \sum_{t=1}^T \alpha^t = \frac{\alpha - \alpha^{T+1}}{ 1- \alpha} = \frac{\alpha}{1 - \alpha} - \frac{\alpha^{T+1}}{1 - \alpha}
\end{align*}
We have
\begin{align}
    \label{ineq:lemma-A-bound}
    \frac{\alpha^{T+1}}{1 - \alpha} &= \frac{\alpha \beta}{T(1 - \alpha)} = \frac{\beta}{T} \cdot \frac{1}{\nicefrac{1}{\alpha} - 1} \leq \frac{\beta}{T} \cdot \frac{2}{\ln(\nicefrac{1}{\alpha})} = \frac{\beta}{T} \cdot \frac{2}{\frac{1}{T}\ln(\nicefrac{T}{\beta})} = \frac{2\beta}{\ln(\nicefrac{T}{\beta})}
    \end{align}
where in the inequality we used Lemma \ref{lem:ineq1} and the fact that $\nicefrac{1}{\alpha} > 1$. Plugging back into $A$ we get, 
\begin{align*}
    A &\geq \frac{\alpha}{1- \alpha} - \frac{2\beta}{\ln(\nicefrac{T}{\beta})}\\
    &\geq \frac{\alpha }{\ln(\nicefrac{1}{\alpha})} - \frac{2\beta}{\ln(\nicefrac{T}{\beta})} \tag{$1-x \leq \ln(\frac{1}{x})$}\\
    &= \frac{\alpha T}{\ln(\nicefrac{T}{\beta})} - \frac{2\beta}{\ln(\nicefrac{T}{\beta})}
\end{align*}
\end{proof}

\begin{thmbox}
\begin{lemmanew}
For $\alpha = \left( \frac{\beta}{T}\right)^{1/T}$ and any $\kappa > 0$,
\begin{align*}
 \sum_{k=1}^T \alpha^{2k} \exp \left(-\frac{1}{\kappa} \sum_{i=k+1}^T \alpha^i \right) & \leq \frac{4 \kappa^2 c_2 (\ln(\nicefrac{T}{\beta}))^2}{e^2 \alpha^2 T} 
\end{align*}
where $c_2= \exp\left( \frac{1}{\kappa} \frac{2\beta}{\ln(\nicefrac{T}{\beta}}\right)$
\label{lemma:B-bound}
\end{lemmanew}
\end{thmbox}
\begin{proof}
First observe that, 
\begin{align*}
\sum_{i=k+1}^T \alpha^i = \frac{\alpha^{k+1} - \alpha^{T+1}}{1 - \alpha} 
\end{align*}
We have
\begin{align*}
    \frac{\alpha^{T+1}}{1 - \alpha} &= \frac{\alpha \beta}{T(1 - \alpha)} = \frac{\beta}{T} \cdot \frac{1}{\nicefrac{1}{\alpha} - 1} \leq \frac{\beta}{T} \cdot \frac{2}{\ln(\nicefrac{1}{\alpha})} = \frac{\beta}{T} \cdot \frac{2}{\frac{1}{T}\ln(\nicefrac{T}{\beta})} = \frac{2\beta}{\ln(\nicefrac{T}{\beta})}
    \end{align*}
where in the inequality we used~\ref{lem:ineq1} and the fact that $\nicefrac{1}{\alpha} > 1$. These relations imply that, 
\begin{align*}
&\sum_{i=k+1}^T \alpha^i \geq \frac{\alpha^{k+1}}{1 - \alpha} - \frac{2\beta}{\ln(\nicefrac{T}{\beta})} \\
&\implies 
\exp\left( - \frac{1}{\kappa} \sum_{i=k+1}^T \alpha^i \right) \leq \exp\left( -\frac{1}{\kappa} \frac{\alpha^{k+1}}{1- \alpha} + \frac{1}{\kappa} \frac{2\beta}{\ln(\nicefrac{T}{\beta})}\right) = c_2 \exp\left(-\frac{1}{\kappa} \frac{\alpha^{k+1}}{1- \alpha} \right) 
\end{align*}
We then have
\begin{align*}
\sum_{k=1}^T \alpha^{2k} \exp\left( - \frac{1}{\kappa} \sum_{i=k+1}^T \alpha^i \right) & \leq c_2 \sum_{k=1}^T \alpha^{2k} \exp\left(-\frac{1}{\kappa} \frac{\alpha^{k+1}}{1- \alpha} \right) \\
& \leq c_2 \sum_{k=1}^T \alpha^{2k} \left( \frac{2 (1-\alpha) \kappa}{e \alpha^{k+1}} \right)^2 & \tag{Lemma \ref{lem:ineq2}} \\
& = \frac{4 \kappa^2 c_2}{e^2 \alpha^2} \, T (1 - \alpha)^2 \\
& \leq \frac{4 \kappa^2 c_2}{e^2 \alpha^2} \, T (\ln(1/\alpha))^2 \\
& = \frac{4 \kappa^2 c_2 (\ln(\nicefrac{T}{\beta}))^2}{e^2 \alpha^2 T}
\end{align*}
\end{proof}

\begin{thmbox}
\begin{lemmanew}
For $\alpha = \left( \frac{\beta}{T}\right)^{1/T}$ and any $\kappa > 0$,
\begin{align*}
\sum_{k=1}^T \alpha^{k} \exp \left(-\frac{1}{\kappa} \sum_{i=k+1}^T \alpha^i \right) \leq c_2 \frac{\kappa \ln (T / \beta)}{e \alpha} 
\end{align*}
for $c_2 = \exp\left( \frac{1}{\kappa} \frac{2\beta}{\ln(T/\beta)}\right)$
\label{lemma:C-bound}
\end{lemmanew}
\end{thmbox}
\begin{proof}
Proceeding in the same way as~\cref{lemma:B-bound}, we obtain the following inequality, 
\begin{align*}
\sum_{k=1}^T \alpha^{k} \exp\left( - \frac{1}{\kappa} \sum_{i=k+1}^T \alpha^i \right) & \leq c_2 \sum_{k=1}^T \alpha^{k} \exp\left(-\frac{1}{\kappa} \frac{\alpha^{k+1}}{1- \alpha} \right) 
\end{align*}
Further bounding this term, 
\begin{align*}
\sum_{k=1}^T \alpha^{k} \exp\left( - \frac{1}{\kappa} \sum_{i=t+1}^T \alpha^i \right) & \leq c_2\sum_{k=1}^T \alpha^{k} \frac{(1 - \alpha) \kappa}{e \alpha^{k+1}}  & \tag{Lemma \ref{lem:ineq2}} \\
& \leq c_2 (1 - \alpha) \frac{\kappa T}{e \alpha}  \\
& \leq c_2 \ln(1/\alpha) \frac{\kappa T}{e \alpha}  \\
& = c_2\frac{\kappa \ln (T / \beta)}{e \alpha} 
\end{align*}
\end{proof}

\begin{thmbox}
\begin{lemmanew}
\label{lem:sls-bounds}
If $f_i$ is $L_i$-smooth, stochastic lines-searches ensures that
\aligns{
    \gamma \norm{\nabla f_i(w)}^2 \leq \frac{1}{c}(f_i(w) - f_i^*),
    &&\text{and}&&
    \min\left\{
        \gamma_{\max},
        \frac{2 \, (1-c)}{L_{i}}
    \right\}
    \leq \gamma \leq \gamma_{\max}.
}
Moreover, if $f_i$ is a one-dimensional quadratic,
\begin{align*}
    \gamma = \min\left\{
        \gamma_{\max},
        \frac{2 \, (1-c)}{L_{i}}
    \right\}
\end{align*}
\end{lemmanew}
\end{thmbox}
\begin{proof}
Recall that if $f_i$ is $L_i$-smooth, then for an arbitrary direction $d$,
\begin{align*}
f_i(w - d) \leq f_i(w) - \lin{\nabla f_i(w), d} + \frac{L_i}{2} \normsq{d}.    
\end{align*}
For the stochastic line-search, $d=\gamma\nabla f_i(w)$. The smoothness and the line-search condition are then
\begin{align*}
\text{Smoothness: } f_i(w - \gamma \nabla f_i(w)) - f_i(w) &\leq \, \left(\frac{L_i}{2}\gamma^2-\gamma \right) \norm{\nabla f_i(w)}^2, \\
\text{Line-search: } f_i(w - \gamma \nabla f_i(w)) - f_i(w) &\leq -c \gamma \norm{\nabla f_i(w)}^2.
\end{align*}
The line-search condition is looser than smoothness if 
\begin{align*}
\textstyle
\left(\frac{L_i}{2}\gamma^2-\gamma \right)\norm{\nabla f_i(w)}^2 &\leq -c \gamma \norm{\nabla f_i(w)}^2.    
\end{align*}
The inequality is satisfied for any $\gamma \in [a, b]$, 
where $a, b$ are values of $\gamma$ that satisfy
the equation with equality,
$a = 0, b = \nicefrac{2(1-c)}{L_i}$,
and the line-search condition holds for $\gamma \leq \nicefrac{2(1-c)}{L_i}$. As the line-search selects the largest feasible step-size, $\gamma \geq \nicefrac{2(1-c)}{L_i}$. If the step-size is capped at $\gamma_{\max}$, we have $\eta \geq \min\{\gamma_{\max}, \nicefrac{2(1-c)}{L_i}\}$, and the proof for the stochastic line-search is complete. 

From the previous discussion, observe that if $\gamma > \frac{2(1-c)}{L_i}$, then we have
\begin{align*}
\textstyle
\left(\frac{L_i}{2}\gamma^2-\gamma \right)\norm{\nabla f_i(w)}^2 > -c \gamma \norm{\nabla f_i(w)}^2.    
\end{align*}
If $f$ is a one-dimensional quadratic, the smoothness inequality is actually an equality, and thus
\begin{align*}
    f_i(w - \gamma \nabla f_i(w)) - f_i(w) &= \, \left(\frac{L_i}{2}\gamma^2-\gamma \right) \norm{\nabla f_i(w)}^2
\end{align*}
So if $\gamma  > \frac{2(1-c)}{L_i}$, 
\begin{align*}
    f_i(w - \gamma \nabla f_i(w)) - f_i(w) &\geq -c \gamma \norm{\nabla f_i(w)}^2
\end{align*}
and the line-search condition does not hold. This implies that for one-dimensional quadratics $\gamma = \min\{ \gamma_{\max}, \frac{2(1-c)}{L_i}\}$
\end{proof}

\section{Lower-bound proofs for~\cref{sec:adaptivity}}
\label{app:lb-proofs}

\subsection{Proof of~\cref{thm:lb-correlated}}
\label{app:lb-correlated-proof}
\myquote{
\restatescukclb*
}
\begin{proof}
For SLS with a general $c \geq \nicefrac{1}{2}$ on quadratics, we know that $\gammak = \frac{2(1 - c)}{L_{i_k}}$ (see~\cref{lem:sls-bounds} for a formal proof). Recall that we consider two one-dimensional quadratics $f_i(w) = \frac{1}{2} (w x_i - y_i)^2$ for $i \in \{1,2\}$ such that $x_1 = 1$, $y_1 = 1$, $x_2 = 2$, $y_2 = -\frac{1}{2}$. Specifically, 
\begin{align*}
    f_1(w) &= \frac{1}{2}(w - 1)^2 \Rightarrow L_1 = 1\\
    f_2(w) &= \frac{1}{2}(2w + \frac{1}{2})^2  \Rightarrow L_2 = 4\\
    f(w) &= \frac{1}{4}(w - 1)^2 + \frac{1}{4}(2w + \frac{1}{2})^2 = \frac{5}{4} w^2  + \frac{1}{4} + \frac{1}{16} \Rightarrow w^* = 0
\end{align*}
If $i_k = 1$,
\begin{align*}
    w_{k+1} = w_k - \alpha_k2(1 - c)(w_k - 1) = 2(1 - c)\alpha_k + (1-2( 1- c)\alpha_k)w_k
\end{align*}
If $i_k = 2$,
\begin{align*}
    w_{k+1} = w_k - 2(1 - c)\alpha_k \frac{2}{4}(2w_k + \frac{1}{2}) = (1 - 2(1 - c)\alpha_k)w_k - \frac{1}{4} 2(1-c)\alpha_k
\end{align*}
Then
\begin{align*}
    \E w_{k+1} = (1- 2(1 - c)\alpha_k) w_k + \frac{1}{2} 2(1-c)\alpha_k - \frac{1}{8} 2(1- c)\alpha_k = (1- 2(1-c)\alpha_k) w_k + \frac{3}{8}2(1-c) \alpha_k
\end{align*}
and
\begin{align*}
    \E w_T = \E (w_T - w^*) = (w_1 - w^*) \prod_{k=1}^T (1 - 2(1-c)\alpha_k) + \frac{3}{8} \sum_{k=1}^T 2(1-c)\alpha_k \prod_{i=k+1}^T (1 - 2(1-c)\alpha_i)
\end{align*}
Using~\cref{lem:sum-prod-eq} and the fact that $2(1-c)\alpha_k \leq 1$ for all $k$, we have that if $w_1 - w^* = w_1 > 0$,
\begin{align*}
    \E (w_T - w^*) \geq \min\left( w_1, \frac{3}{8}\right) 
\end{align*}
\end{proof}

\clearpage
\subsection{Proof of~\cref{thm:lb-decorrelated}}
\label{app:lb-decorrelated-proof}
\myquote{
\restatescukdclb*
}
\begin{proof}
One has $w^* = \frac{y}{x}$ and $L = x^2$. Therefore
\begin{align*}
    w_{k+1} - w^* &= w_k - w^* - \alpha_k \eta_k x \left( xw_k -y  \right)\\
    &= w_k - w^* - \alpha_k \frac{\nu}{L} L w_k + \alpha_k \frac{\nu}{L} xy\\
    &= w_k - w^* - \alpha_k \nu w_k + \alpha_k \rho w^* = (1 - \nu \alpha_k) (w_k - w^*)
\end{align*}
Iterating gives the first part of the result. Now, for $k \leq k'$, we have
\begin{align*}
    1 - \nu \alpha^k &\leq 1 - \nu \alpha^{k'} \leq 1 - \nu \alpha^{\frac{T}{\ln (T/\beta)} (\ln \nu - \ln 3)} = 1 - \nu \left(\frac{\beta}{T}\right)^{\frac{1}{\ln (T/\beta)} (\ln \nu - \ln 3)} = 1 - \nu \left( \frac{3}{\nu}\right) = -2
\end{align*}
and thus
\begin{align*}
    \abs{w_{k' + 1} - w^*} = \abs{w_1 - w^*} \prod_{i=1}^{k'} \abs{ 1- \nu \alpha_k} \geq \abs{w_1 - w^*} 2^{k'}
\end{align*}
\end{proof}

\subsection{Lemmas for convex setting}
\label{app:lemmas-convex}

\begin{thmbox}
\begin{lemmanew}
The polynomial stepsize defined as $\alpha_k = (\nicefrac{1}{k})^{\delta}$ for some $0 \leq \delta \leq 1$ cannot satisfy $\sum_{k=1}^T \alpha_k \geq C_1 T$ and $\sum_{k=1}^T \alpha_k^2 \leq C_2 \sqrt{T}$ for positive constants $C_1 $ and $C_2$.
\label{lemma:polynomial-convex-bad}
\end{lemmanew}
\end{thmbox}
\begin{proof}
If $\delta = 0$, $\alpha_k = 1$ for all $k$, and then $\sum_{k=1}^T \alpha_k^2 = T$. If $\delta = 1$, then $\sum_{k=1}^T \alpha_k = \Theta(\ln T)$.\\
If $0 < \delta < 1$, basic calculus shows that
\begin{align*}
    \int_{1}^{T+1} \frac{1}{x^\delta} \leq \sum_{k=1}^T \frac{1}{k^\delta} \leq 1 + \int_{1}^T \frac{1}{x^\delta}
\end{align*}
and thus
\begin{align*}
    \frac{1}{1 - \delta} \left( (T+1)^{1 - \delta} - 1\right) \leq \sum_{k=1}^T \frac{1}{k^\delta} \leq 1 + \frac{1}{1 - \delta} \left( T^{1 - \delta} - 1 \right)
\end{align*}
which shows that $\sum_{k=1}^T \alpha_k = \Theta( T^{1- \delta})$, and thus we cannot have $\sum_{k=1}^T \alpha_k \geq C_1 T$ for all $T$.
\end{proof}

\begin{thmbox}
\begin{lemmanew}
The exponential stepsize defined as $\alpha_k = \alpha^k$ for some $\alpha < 1$ cannot satisfy $\sum_{k=1}^T \alpha_k \geq C_1 T$ and $\sum_{k=1}^T \alpha_k^2 \leq C_2 \sqrt{T}$ for positive constants $C_1 $ and $C_2$.
\label{lemma:exponential-convex-bad}
\end{lemmanew}
\end{thmbox}
\begin{proof}
Suppose by contradiction that the exponential stepsize satisfies the two conditions. Then
\begin{align*}
    C_2 \sqrt{T} \geq \sum_{k=1}^T \alpha_k^2 &= \sum_{k=1}^T \alpha^{2k} = \sum_{k=1}^{2T} \alpha^k - \sum_{k=1}^T \alpha^{2k - 1} = \sum_{k=1}^{2T} \alpha^k - \frac{1}{\alpha} \sum_{k=1}^T \alpha^{2k}
\end{align*}
By assumption, $\sum_{k=1}^{2T} \alpha^k \geq C_1 2T$ and $\sum_{k=1}^T \alpha^{2k} \leq C_2 \sqrt{T}$. Therefore 
\begin{align*}
    \sum_{k=1}^{2T} \alpha^k - \frac{1}{\alpha} \sum_{k=1}^T \alpha^{2k} \geq 2C_1 T - \frac{1}{\alpha} C_2 \sqrt{T}
\end{align*}
But then we obtain
\begin{align*}
    C_2 \sqrt{T} \geq 2C_1 T - \frac{1}{\alpha}C_2 \sqrt{T}
\end{align*}
which is a contradiction by taking $T$ to infinity.
\end{proof}
\clearpage
\section{Proofs for~\cref{sec:acceleration}}
\label{app:acceleration-proof}

\subsection{Reformulation}
\label{app:reformulation}
Let us consider a general ASGD update whose parameters satisfy the following conditions. 
\begin{align}
\ak^2 &= (1 - \ak) \akp^2 \frac{\etak}{\etakp} + \ak \mu \etak. \label{eq:a-def} \\
\bk &= \frac{(1 - \akp) \akp \frac{\etak}{\etakp}}{\ak + \akp^2 \frac{\etak}{\etakp}}, \label{eq:b-a} 
\end{align}
It can be verified that setting $\etak = \gammak \alphak^2 = \frac{1}{L} \left(\frac{\beta}{T} \right)^{2k/T}$, $\ak = \sqrt{\frac{\mu}{L}} \left(\frac{\beta}{T} \right)^{k/T}$ satisfies~\cref{eq:a-def}. We first show that the update in~\cref{eq:extrapolation}-\cref{eq:sgd-update} satisfying the conditions in~\cref{eq:b-a} and ~\cref{eq:a-def} can be written in an equivalent form more amenable to the analysis. 
\begin{thmbox}
\begin{lemmanew}
The following update:
\begin{align}
\yk & = \xk - \frac{\ak \qk}{\qk + \ak \mu} (\xk - \zk) \label{eq:y-update} \\
\xkk &= \yk - \etak \gradk{\yk} & \label{eq:x-update} \\
\zkk &= \xk + \frac{1}{\ak} [\xkk - \xk] & \label{eq:z-update} 
\end{align}
where,
\begin{align}
\qkk &= (1 - \ak) \qk + \ak \mu & \label{eq:q-update} \\
\ak^2 & = \qkk \etak \label{eq:a-q} \\ 
\zkk &= \frac{1}{\qkk} \left[ (1 - \ak) \qk \zk + \ak \mu \yk - \ak \gradk{\yk} \right] & \label{eq:z-def} 
\end{align}
is equivalent to the update in~\cref{eq:extrapolation}-\cref{eq:sgd-update}.
\label{lemma:reformulation}
\end{lemmanew}
\end{thmbox}
\begin{proof}
\begin{align*}
\intertext{First we check the consistency of the update (\cref{eq:z-update}) and definition (\cref{eq:z-def}) of $\zk$. Using~\cref{eq:z-def},}
\zkk &= \frac{1}{\qkk} \left[ (1 - \ak) \qk \zk + \ak \mu \yk - \ak \gradk{\yk} \right] \\
& = -\frac{(1 - \ak)}{\ak} \xk - \frac{\ak}{\qkk} \gradk{\yk} + \yk \left[\frac{(1 - \ak) (\qk + \ak \mu)}{\qkk \ak} + \frac{\ak \mu}{\qkk} \right] \tag{From~\cref{eq:y-update}} \\
& = -\frac{(1 - \ak)}{\ak} \xk - \frac{\ak}{\qkk} \gradk{\yk} + \yk \left[\frac{\qk (1 - \ak) +  (\ak \mu - \ak^2 \mu)}{\qkk \ak} + \frac{\ak^2 \mu}{\qkk \ak} \right] \\
& = -\frac{(1 - \ak)}{\ak} \xk - \frac{\ak}{\qkk} \gradk{\yk} + \yk \left[\frac{(\qkk - \ak \mu) +  (\ak \mu - \ak^2 \mu) + \ak^2 \mu}{\qkk \ak} \right] & \tag{From~\cref{eq:q-update}} \\
& = \xk - \frac{\xk}{\ak} + \frac{1}{\ak} \left[\yk - \etak \gradk{\yk} \right] & \tag{From~\cref{eq:a-q}} \\
\zkk & = \xk + \frac{1}{\ak} \left[\xkk - \xk \right] & \tag{From~\cref{eq:x-update}} 
\intertext{which recovers~\cref{eq:z-update} showing that the definition of $\zk$ and its update is consistent.} 
\intertext{Now we check the equivalence of~\cref{eq:a-def} and ~\cref{eq:q-update}-\cref{eq:a-q}. Eliminating $\qk$ using~\cref{eq:q-update}-\cref{eq:a-q},}
\frac{\ak^2}{\etak} &= (1 - \ak) \frac{\akp^2}{\etakp} + \ak \mu 
\intertext{Multiplying by $\etak$ recovers~\cref{eq:a-def}.}
\end{align*}
\begin{align*}
\intertext{Since~\cref{eq:sgd-update} and~\cref{eq:x-update} are equivalent, we need to establish the equivalence of~\cref{eq:extrapolation} and the updates in~\cref{eq:y-update}-\cref{eq:z-update}. From~\cref{eq:z-update}}
\zk & = \xkp + \frac{1}{\akp} \left[\xk - \xkp \right] \implies \zk - \xk  = \frac{1 - \akp}{\akp} (\xk - \xkp)
\intertext{Starting from~\cref{eq:y-update} and using the above relation to eliminate $\zk$,}
\yk& = \xk + \frac{\ak \qk}{\qk + \ak \mu} \frac{1 - \akp}{\akp} \left[\xk - \xkp \right] 
\intertext{which is in the same form as~\cref{eq:extrapolation}. We now eliminate $\qk$ from $\frac{\ak \qk}{\qk + \ak \mu} \frac{1 - \akp}{\akp}$. From~\cref{eq:q-update} and~\cref{eq:a-q},}
\frac{\ak^2}{\etak} &= (1 - \ak) \qk + \ak \mu \implies \qk + \ak \mu = \frac{\ak^2}{\etak} + \ak \qk 
\intertext{Using this relation,}
\frac{\ak \qk}{\qk + \ak \mu} \frac{1 - \akp}{\akp} &= \frac{\qk \etak}{\ak + \qk \etak} \frac{1 - \akp}{\akp} \\
\intertext{Using~\cref{eq:a-q}, observe that $\etak \qk = \frac{\etak}{\etakp} \etakp \qk = \frac{\etak}{\etakp} \akp^2$. Using this relation, }
\frac{\ak \qk}{\qk + \ak \mu} \frac{1 - \akp}{\akp} & = \frac{\frac{\etak}{\etakp} \akp^2}{\ak + \frac{\etak}{\etakp} \akp^2} \frac{1 - \akp}{\akp} = \frac{(1 - \akp) \akp}{\ak \frac{\etakp}{\etak} + \akp^2} = \bk 
\intertext{which establishes the equivalence to~\cref{eq:extrapolation} and completes the proof.}
\end{align*}
\end{proof}

\subsection{Estimating sequences}
\label{app:estimating-sequences}
We will use a stochastic variant of estimating sequences~\citep{nesterov2018lectures}. Specifically, the  estimating sequences $\{\phik$, $\lambdak \}_{k = 1}^{\infty}$ are defined as: for $\ak \in (0,1)$ s.t. $\sum_{k = 1}^{\infty} \ak = \infty$ and an arbitrary convex function $\phi_1(\cdot)$, for all $k$, 
\begin{align}
\lambda_1 & = 1 \quad \text{;} \quad \lambdakk = (1 - \ak) \lambdak,  \label{eq:lambda-def}  \\    
\E_{k-1}[\phik(\x)] & \leq (1 - \lambdak) f(\x) + \lambdak \phi_1(\x) \,,
\label{eq:phi-def} 
\end{align}
where $\E_k$ is defined as the expectation w.r.t the randomness in iterations $j = 1$ to $k$.

Following the proof in~\citet[Lemma 2.2.2]{nesterov2018lectures}, we first prove the following lemma. 

\begin{lemmanew}
If $f$ is $\mu$ strongly-convex and $\{\yk\}_{k = 1}^{\infty}$ is an arbitrary sequence of iterates and      
\begin{align}
\phikk(\x) & = (1 - \ak) \, \phik(\x) + \ak \left[\fk(\yk) + \langle \nabla \fk(\yk), \x - \yk \rangle + \frac{\mu}{2} \normsq{\x - \yk} \right]
\label{eq:phi-recursive}    
\end{align}
then $\E_k[\phikk(\x)]$ satisfies the relation in~\cref{eq:phi-def} 
\label{lemma:es}
\end{lemmanew}
\begin{proof}
The proof proceeds by induction. For $k = 1$, since $\lambda_1 = 1$, $\phi_1(\x) = (1 - \lambda_1) f_1(\x) + \lambda_1 \phi_1(\x)$ and hence~\cref{eq:phi-def} is satisfied. Assuming that $\E_{k-1}[\phik]$ satisfies~\cref{eq:phi-def}, we will prove that $\E_{k}[\phikk]$ also satisfies it. Taking expectation w.r.t the randomness in iteration $k$,
\begin{align*}
\E[\phikk(\x)] &= (1 - \ak) \, \phik(\x) + \ak \left[f(\yk) + \langle \nabla f(\yk), \x - \yk \rangle + \frac{\mu}{2} \normsq{\x - \yk} \right] \\
\E[\phikk(\x)] & \leq (1 - \ak) \, \phik(\x) + \ak f(\x) \tag{Since $f$ is $\mu$ strongly-convex} \\
& = (1 - (1 - \ak) \lambdak) f(\x) + (1 - \ak) (\phik(\x) - (1 - \lambdak) f(\x)) \\
\intertext{Taking expectation w.r.t randomness in iterations $j = 1$ to $k - 1$}
\E_{k}[\phikk(\x)] & \leq (1 - (1 - \ak) \lambdak) f(\x) + (1 - \ak) \E_{k-1} \left[\phik(\x) - (1 - \lambdak) f(\x)\right] \\
\E_{k}[\phikk(\x)] & \leq (1 - (1 - \ak) \lambdak) f(\x) + (1 - \ak) \, \lambdak \, \phi_1(\x) \tag{Inductive hypothesis} \\
& \leq (1 - \lambdakk) f(\x) + \lambdakk \phi_1(\x) \,, \tag{From the definition of $\lambdakk$} 
\end{align*}
completing the induction. 
\end{proof}

Next, following the proof in~\citet[Lemma 2.2.3]{nesterov2018lectures}, we prove the following lemma. 
\begin{lemmanew}
Define $\phik^* := \min \phik(\x)$. Let $\phi_1(\x) = \phi_1^* + \frac{q_1}{2} \, \normsq{\x - z_1}$ for some initialization $z_1$. The recursive definition of $\phik$ in~\cref{eq:phi-recursive} satisfy the following relation for all $k$, 
\begin{align}
\phik(\x) &= \phik^* + \frac{\qk}{2} \, \normsq{\x - \zk} \,
\label{eq:phi-relation} 
\end{align}
where, 
\begin{align}
\qkk &= (1 - \ak) \qk + \ak \mu \label{eq:q-update-es} \\
\zkk &= \frac{1}{\qkk} \left[ (1 - \ak) \qk \zk + \ak \mu \yk - \ak \gradk{\yk} \right] \label{eq:z-update-es} \\
\phikk^* & = (1 - \ak) \phik^* + \ak \left[\fk(\yk) - \frac{\ak}{2 \qkk} \normsq{\gradk{\yk}} + \frac{(1 - \ak) \, \qk}{\qkk} \left( \frac{\mu}{2} \normsq{\yk - \zk} + \inner{\gradk{\yk}}{\zk - \yk} \right)\right] 
\label{eq:phistar-update}
\end{align}
\end{lemmanew}
\begin{proof}
First we use induction to show that $\nabla^2 \phik(\x) = \qk I_d$ for all $k$. Using the definition of $\phi_1(\x)$, $\nabla^2 \phi_1(\x) = q_1 I_d$ and hence the relation holds for $k = 1$. Assuming the relation holds for $k$, let us prove it for $k+1$. Using the relation in~\cref{eq:phi-recursive}, 
\begin{align*}
\nabla^2 \phikk(\x) &= (1 - \ak) \, \nabla^2 \phik(\x) + \ak \mu I_d \\
& = ((1 - \ak) \qk + \ak \mu) \, I_d \tag{Inductive hypothesis} \\
& = \qkk I_d \tag{From~\cref{eq:q-update-es}}
\end{align*}
This completes the induction and we conclude that $\nabla^2 \phik(\x) = \qk I_d$. This justifies the form of $\phik$ in~\cref{eq:phi-relation}.

From~\cref{eq:phi-relation}, we know that $\zkk = \arg\min \phikk(\x)$ and hence we require that $\zkk$ satisfy the first-order optimality condition for $\phikk(\x)$. Hence, we verify that $\nabla \phikk(\zkk) = 0$. 
\begin{align*}
\phikk(\x) & = (1 - \ak) \, \phik(\x) + \ak \left[\fk(\yk) + \langle \nabla \fk(\yk), \x - \yk \rangle + \frac{\mu}{2} \normsq{\x - \yk} \right] \tag{From~\cref{eq:phi-recursive}} \\
& = (1 - \ak) \, \left[\phik^* + \frac{\qk}{2} \, \normsq{\x - \zk}\right] + \ak \left[\fk(\yk) + \langle \nabla \fk(\yk), \x - \yk \rangle + \frac{\mu}{2} \normsq{\x - \yk} \right] \tag{From~\cref{eq:phi-relation}} \\
\implies \nabla \phikk(\zkk) &= (1 - \ak) \, \qk (\zkk - \zk) + \ak \left[\nabla \fk(\yk) + \mu \, (\zkk - \yk) \right] \\
& = ( (1 - \ak) \, \qk + \ak u ) \, \zkk - \left[ (1 - \ak) \, \qk \, \zk  + \ak \mu \yk - \ak \nabla \fk(\yk) \right] \\
& = \qkk \, \zkk - \left[ (1 - \ak) \, \qk \, \zk  + \ak \mu \yk - \ak \nabla \fk(\yk) \right] \tag{From~\cref{eq:q-update-es}} \\
\implies \nabla \phikk(\zkk) & = 0 \tag{From~\cref{eq:z-update-es}}
\end{align*}
Since $\nabla \phikk(\zkk) = 0$ and  $\nabla \phikk(\zkk) = \qkk I_d \succ 0$, we have verified that $\zkk = \arg\min \phikk(\x)$. 

Finally, we calculate $\phikk^*$. From~\cref{eq:phi-recursive} and~\cref{eq:phi-relation}, we know that, 
\begin{align}
\phikk(\x) & =  (1 - \ak) \, [\phik^* + \frac{\qk}{2} \, \normsq{\x - \zk}] + \ak \left[f(\yk) + \langle \nabla \fk(\yk), \x - \yk \rangle + \frac{\mu}{2} \normsq{\x - \yk} \right] \\
\implies \phikk(\yk) &= (1 - \ak) \, [\phik^* + \frac{\qk}{2} \, \normsq{\yk - \zk}] + \ak \, \fk(\yk) \label{eq:phi_star-1} \\
\intertext{From~\cref{eq:phi-relation},}
\phikk(\x) &= \phikk^* + \frac{\qkk}{2} \, \normsq{\x - \zkk} \implies \phikk(\yk) = \phikk^* + \frac{\qkk}{2} \, \normsq{\yk - \zkk} \label{eq:phi_star-2} \\
\intertext{Using~\cref{eq:phi_star-1,eq:phi_star-2},}
\phikk^* &= (1 - \ak) \, [\phik^* + \frac{\qk}{2} \, \normsq{\yk - \zk}] + \ak \, \fk(\yk) - \frac{\qkk}{2} \, \normsq{\yk - \zkk} \label{eq:phistar-intermediate}
\end{align}
Using~\cref{eq:z-update-es} to calculate $\zkk - \yk$,
\begin{align}
\zkk - \yk &= \frac{1}{\qkk} \left[ (1 - \ak) \qk \zk + (\ak \mu - \qkk) \yk - \ak \gradk{\yk}\right]  \\
& = \frac{1}{\qkk} \left[ (1 - \ak) \qk \zk - (1 - \ak) \, \qk \, \yk - \ak \gradk{\yk} \right] \tag{From~\cref{eq:q-update-es}} \\
\implies \zkk - \yk & = \frac{1}{\qkk} \left[ (1 - \ak) \qk (\zk - \yk) - \ak \nabla \fk(\yk) \right]
\end{align}
Using the above equality to calculate $\frac{\qkk}{2} \, \normsq{\yk - \zkk}$,
\begin{align}
\frac{\qkk}{2} \, \normsq{\yk - \zkk} &= \frac{1}{2 \qkk} \left[(1 - \ak)^2 \qk^2 \normsq{\zk - \yk} + \ak^2 \normsq{\nabla \fk(\yk)} - 2 (1 - \ak) \qk \ak \, \langle \nabla \fk(\yk), \zk - \yk \rangle \right]    
\end{align}
Combining the above equality with~\cref{eq:phistar-intermediate},
\begin{align}
\phikk^* &= (1 - \ak) \, [\phik^* + \frac{\qk}{2} \, \normsq{\yk - \zk}] + \ak \, \fk(\yk) \\ & - \frac{1}{2 \qkk} \left[(1 - \ak)^2 \qk^2 \normsq{\zk - \yk} + \ak^2 \normsq{\nabla \fk(\yk)} - 2 (1 - \ak) \qk \ak \, \langle \nabla \fk(\yk), \zk - \yk \rangle \right] \\
& = (1 - \ak) \, \phik^* + \ak \left[\fk(\yk) - \frac{\ak}{2 \qkk} \normsq{\nabla \fk(\yk)} + \frac{(1 - \ak) \, \qk}{\qkk}  \langle \nabla \fk(\yk), \zk - \yk \rangle \right] \\
& + \frac{\normsq{\yk - \zk}}{2} \, (1 - \ak) \, \qk \, \left[ 1 - \frac{(1 - \ak) \, \qk}{\qkk}\right] \\
\phikk^* & = (1 - \ak) \phik^* + \ak \left[\fk(\yk) - \frac{\ak}{2 \qkk} \normsq{\gradk{\yk}} + \frac{(1 - \ak) \, \qk}{\qkk} \left( \frac{\mu}{2} \normsq{\yk - \zk} + \inner{\gradk{\yk}}{\zk - \yk} \right)\right] \tag{From~\cref{eq:q-update-es}}
\end{align}
\end{proof}

Finally note that the definition of $\phik$ in~\cref{eq:phi-relation} can be used to rewrite~\cref{eq:y-update} as 
\begin{align}
\yk & = \xk - \frac{\ak}{\qk + \ak \mu} \nabla \phik(\xk).  \label{eq:y-update-2}   
\end{align}

\subsection{Proof of~\cref{thm:sc-accelerated}}
\label{app:acceleration-thm-proof}
Given the definitions and relations in~\cref{app:estimating-sequences}, we first prove the descent lemma for \blue{$\etak = \frac{1}{L} \alphak^2$}, where $\alphak \leq 1$ is the exponentially decreasing step-size.
\begin{thmbox}
\begin{lemmanew}
Using the update in~\cref{eq:x-update} with $\etak = \frac{1}{L} \alphak^2$, we obtain the following inequality.
\begin{align*}
\E[f(\xkk)] & \leq \E[f(\yk)] - \frac{\etak }{2}\normsq{\grad{\yk}} + \frac{1}{2 L} \alphak^2 \sigma^2
 \end{align*}
\label{lemma:descent-lemma}
\end{lemmanew}
\end{thmbox}
\begin{proof}
By smoothness, and the update in~\cref{eq:x-update},
\begin{align*}
f(\xkk) & \leq f(\yk) - \etak \inner{\grad{\yk}}{\gradk{\yk}} + \frac{L}{2} \etak^2 \normsq{\gradk{\yk}} \\
\intertext{Taking expectation w.r.t. $i_k$,}
\E[f(\xkk)] & \leq \E[f(\yk)] - \etak \normsq{\grad{\yk}} + \frac{L}{2} \etak^2 \E[\normsq{\gradk{\yk}}] & \tag{$\etak$ is independent of $i_k$.} \\
& \leq \E[f(\yk)] - \etak \normsq{\grad{\yk}} + \frac{L}{2} \etak^2 \, \E[\normsq{\grad{\yk}}] + \frac{L}{2} \etak^2 \, \sigma^2 & \tag{From~\cref{eq:bounded-variance}} \\
&\leq \E[f(\yk)] - \etak \normsq{\grad{\yk}} + \frac{\etak}{2} \E[\normsq{\grad{\yk}}] + \frac{L}{2} \etak^2 \sigma^2 & \tag{$\etak \leq \frac{1}{L}$} \\ 
& = \E[f(\yk)] - \frac{\etak}{2} \normsq{\grad{\yk}} + \frac{1}{2 L} \alphak^2 \sigma^2 \tag{Since $\alphak \leq 1$} \\
\end{align*}
\end{proof}

The main part of the proof is to show that $\E[\phik^*]$ is an upper-bound on $\E[f(\xk)]$ (upto a factor governed by the noise term $\gNk$ depending on $\sigma^2$) for all $k$. We prove this in the following lemma. 
\begin{thmbox}
\begin{lemmanew}
For the estimating sequences defined in~\cref{app:estimating-sequences} and the updates in~\cref{eq:y-update}-\cref{eq:z-def}, for all $k$, 
\begin{align*}
\E[\phik^*] :=  \E[ \inf_{\x} \phik(\x)] & \geq \E[f(\xk)] - \gNk
\end{align*}
where $\gNk := \frac{\sigma^2}{L} \sum_{j = 1}^{k-1} \alpha_j^2 \, \prod_{i = j+1}^{k-1} (1 - r_i)$. 
\label{lemma:acceleration-induction}
\end{lemmanew}
\end{thmbox}
\begin{proof}
We will prove the lemma by induction. For $k =1$, we define $\phi^* _1 = f(\x_1)$, and since $\gN_1 = 0$, $\E[\phi^*_1] =f(\x_1) - \gN_1$, thus satisfying the base-case for the induction. For the induction, we will use the fact that $\gNkk = (1 - \ak) \gNk + \frac{\sigma^2}{L} \alphak^2$.

Assuming the induction hypothesis, $\E[\phik^*] \geq \E[f(\xk)] - \gNk$, we use~\cref{eq:phistar-update} to prove the statement for $k+1$ as follows. Taking expectations w.r.t to the randomness in iteration $k$
\begin{align*}
\E[\phikk^*] & = (1 - \ak) \E[\phik^*] + \ak \E\left[\fk(\yk) - \frac{\ak}{2 \qkk} \normsq{\gradk{\yk}} +  \frac{(1 - \ak) \, \qk}{\qkk} \left( \frac{\mu}{2} \normsq{\yk - \zk} + \inner{\gradk{\yk}}{\zk - \yk} \right)\right] \\
& = (1 - \ak) \E[\phik^*] + \ak \left[f(\yk) - \frac{\ak}{2 \qkk} \E \normsq{\gradk{\yk}} +  \frac{(1 - \ak) \, \qk}{\qkk} \left( \frac{\mu}{2} \normsq{\yk - \zk} + \inner{\grad{\yk}}{\zk - \yk} \right)\right] \tag{Since $\fk$ is unbiased} \\
\intertext{Taking expectations w.r.t to the randomness in iterations $j = 1$ to $k-1$,}
& = (1 - \ak) \E[\phik^*] + \ak \E \left[f(\yk) - \frac{\ak}{2 \qkk} \normsq{\gradk{\yk}} +  \frac{(1 - \ak) \, \qk}{\qkk} \left( \frac{\mu}{2} \normsq{\yk - \zk} + \inner{\grad{\yk}}{\zk - \yk} \right)\right] \\
& \geq (1 - \ak) \E[f(\xk) - \gNk] \\ & + \ak \E\left[f(\yk) - \frac{\ak}{2 \qkk} \normsq{\gradk{\yk}} +  \frac{(1 - \ak) \, \qk}{\qkk} \left( \frac{\mu}{2} \normsq{\yk - \zk} + \inner{\grad{\yk}}{\zk - \yk} \right)\right] & \tag{by the induction hypothesis} \\
& = (1 - \ak) \E[f(\xk)] + \ak \E[f(\yk)] - \frac{\ak^2}{2 \qkk} \E \normsq{\gradk{\yk}} \\ & +  \frac{\ak (1 - \ak) \, \qk}{\qkk} \E\left( \frac{\mu}{2} \normsq{\yk - \zk} + \inner{\grad{\yk}}{\zk - \yk} \right)  - (1 - \ak) \gNk \\
& \geq (1 - \ak) \E[f(\xk)] + \ak \E[f(\yk)] - \frac{\ak^2}{2 \qkk} \blue{\left[
\normsq{\grad{\yk}} + \sigma^2 \right]} \\ & +  \frac{\ak (1 - \ak) \, \qk}{\qkk} \E\left( \frac{\mu}{2} \normsq{\yk - \zk} + \inner{\grad{\yk}}{\zk - \yk} \right)  - (1 - \ak) \gNk \tag{\blue{Using~\cref{eq:bounded-variance}}} \\
& = (1 - \ak) \E[f(\xk)] + \ak \E[f(\yk)] - \blue{\frac{\etak}{2} \E\normsq{\grad{\yk}}} \\
& +  \frac{\ak (1 - \ak) \, \qk}{\qkk} \E\left( \frac{\mu}{2} \normsq{\yk - \zk} + \inner{\grad{\yk}}{\zk - \yk} \right) - (1 - \ak) \gNk - \blue{\frac{\etak}{2} \, \sigma^2}   & \tag{Using~\cref{eq:a-q}} \\
\intertext{By convexity, $f(\xk) \geq f(\yk) + \inner{\grad{\yk}}{\xk - \yk}$,}
& \geq (1 - \ak) \E[f(\yk) + \inner{\grad{\yk}}{\xk - \yk}] + \ak \E[f(\yk)] - \blue{\frac{\etak}{2} \E\normsq{\grad{\yk}}} \\ & + \frac{\ak (1 - \ak) \, \qk}{\qkk}  \E \left( \frac{\mu}{2} \normsq{\yk - \zk} + \inner{\grad{\yk}}{\zk - \yk} \right)  - (1 - \ak) \gNk - \blue{\frac{\etak}{2} \, \sigma^2} \\
& = \E \left[f(\yk) - \blue{\frac{\etak}{2} \E\normsq{\grad{\yk}}} \right] +  \frac{\ak (1 - \ak) \, \qk}{\qkk} \E \left( \frac{\mu}{2} \normsq{\yk - \zk} + \inner{\grad{\yk}}{\zk - \yk} \right) \\ & +  (1 - \ak) \E[\inner{\grad{\yk}}{\xk - \yk}] - (1 - \ak) \gNk - \blue{\frac{\etak}{2} \, \sigma^2} \\
\intertext{By~\cref{lemma:descent-lemma},}
& \geq \E \left[f(\xkk) - \frac{1}{2L} \alphak^2 \sigma^2 \right] +  \frac{\ak (1 - \ak) \, \qk}{\qkk} \E \left( \frac{\mu}{2} \normsq{\yk - \zk} + \inner{\grad{\yk}}{\zk - \yk} \right) \\ & + (1 - \ak) \E[\inner{\grad{\yk}}{\xk - \yk}] - (1 - \ak) \gNk - \blue{\frac{\etak}{2} \, \sigma^2} \\
& = \E \left[f(\xkk) \right] +  \frac{\ak (1 - \ak) \, \qk}{\qkk} \E \left( \frac{\mu}{2} \normsq{\yk - \zk} + \inner{\grad{\yk}}{\zk - \yk} \right) +  (1 - \ak) \E[\inner{\grad{\yk}}{\xk - \yk}] \\ & - \left[(1 - \ak) \gNk +  \frac{1}{2 L} \alphak^2 \sigma^2 + \blue{\frac{\alphak^2}{2 L} \, \sigma^2} \right]\\
\intertext{Since $\gNkk = \left[(1 - \ak) \gNk + \frac{1}{L} \alphak^2 \sigma^2 \right]$,}
\E[\phikk^*] & \geq \E \left[f(\xkk) \right] - \gNkk + (1 - \ak) \E[\inner{\grad{\yk}}{\xk - \yk}] \\ & + \frac{\ak (1 - \ak) \, \qk}{\qkk} \E \left( \frac{\mu}{2} \normsq{\yk - \zk} + \inner{\grad{\yk}}{\zk - \yk} \right)  \\
\end{align*}
Now we show that $(1 - \ak) \E\left[ \inner{\grad{\yk}}{\xk - \yk} + \frac{\ak \, \qk}{\qkk}  \left( \frac{\mu}{2} \normsq{\yk - \zk} + \inner{\grad{\yk}}{\zk - \yk} \right) \right] \geq 0$. For this, we use~\cref{eq:y-update}
\begin{align*}
\yk & = \xk - \frac{\qk \ak}{\qk + \ak \mu} (\xk - \zk) \\
\implies \zk - \yk &= \zk - \xk + \frac{\qk \ak}{\qk + \ak \mu} (\xk - \zk) = \left(1 - \frac{\qk \ak}{\qk + \ak \mu} \right) (\zk - \xk) \\
& = \left( \frac{\qk (1 - \ak) + \ak \mu}{\qk + \ak \mu} \right) (\zk - \xk) = \left( \frac{\qkk}{\qk + \ak \mu} \right) (\zk - \xk) & \tag{By~\cref{eq:q-update}} \\
\implies \frac{\ak \qk}{\qkk} \inner{\grad{\yk}}{\zk - \yk} &= \left \langle \grad{\yk} , \left( - \frac{\ak \, \qk}{\qk + \ak \mu} \right) (\xk - \zk) \right \rangle = \left \langle \grad{\yk} , \yk- \xk \right \rangle
\end{align*}
Using this relation to simplify, 
\begin{align*}
& (1 - \ak)\E \left[ \inner{\grad{\yk}}{\xk - \yk} + \frac{\ak \, \qk}{\qkk}  \left( \frac{\mu}{2} \normsq{\yk - \zk} + \inner{\grad{\yk}}{\zk - \yk} \right) \right] \\ 
&= (1 - \ak) \E \left[\frac{\ak \, \qk}{\qkk} \frac{\mu}{2} \normsq{\yk - \zk} + \left[\inner{\grad{\yk}}{\xk - \yk} + \left \langle \grad{\yk} , \yk- \xk \right \rangle \right] \right]\\
& = (1 - \ak) \E \left[\frac{\ak \, \qk}{\qkk} \frac{\mu}{2} \normsq{\yk - \zk} \right]\geq 0 & \tag{Since $\ak \leq 1$.}
\end{align*}
Putting everything together, 
\begin{align*}
\E[\phikk^*] & \geq \E \left[f(\xkk) \right] - \gNkk    
\end{align*}
and we conclude that $\E[\phik^*] \geq \E \left[f(\xk) \right] - \gNk$ for all $k$ by induction.  
\end{proof}
We now use the above lemma to prove~\cref{thm:sc-accelerated}.

\myquote{
\restatescacc*
}
\begin{proof}
Using the reformulation in~\cref{lemma:reformulation} gives us $\qk = \mu$ for all $k$ and $\z_1 = \x_1$. For the estimating sequences defined in~\cref{app:estimating-sequences}, using~\cref{lemma:acceleration-induction}, we know that the (reformulated) updates satisfy the following relation, 
\begin{align*}
\E_{T}[f(\x_{T+1})] & \leq \E_{T}[\phi^*_{T+1}] + \gN_{T+1}  \leq  \E_{T}[\phi_{T+1}(\xopt)] + \gN_{T+1}
\end{align*}
From~\cref{lemma:es}, we know that $\E_T[\phi_{T+1}(\xopt)]$ satisfies~\cref{eq:phi-def}. Hence,
\begin{align*}
\E_{T}[\phi_{T+1}(\xopt)] & \leq (1 - \lambda_{T+1}) f(\xopt) + \lambda_{T+1} \, \phi_1(\xopt)
\end{align*}
Using the above relations, 
\begin{align*}
\E[f(\x_{T+1})] & \leq (1 - \lambda_{T+1}) f^* + \lambda_{T+1} \phi_1(\xopt) + \gN_{T+1} \\
\implies \E[f(\x_{T+1}) - f^*] & \leq \lambda_{T+1} \left[\phi_1(\xopt) - f^* \right] + \gN_{T+1} \\
\intertext{By~\cref{eq:phi-relation},}
& \leq \lambda_{T+1} \left[\phi^*_1 + \frac{\q_1}{2} \normsq{\xopt- \z_1} - f^* \right] + \gN_{T+1} \\
\intertext{Choosing $\phi^*_1 = f(\x_1)$,}
& \leq \lambda_{T+1} \left[f(\x_1) - f^* + \frac{\q_1}{2} \normsq{\xopt- \z_1} \right] + \gN_{T+1} \\
\intertext{Since $\z_1 = \x_1$, $\q_1 = \mu$,}
\implies \E[f(\x_{T+1}) - f^*] & \leq \lambda_{T+1} \left[f(\x_1) - f^* + \frac{\mu}{2} \normsq{\xopt- \x_1} \right] + \frac{\sigma^2}{L} \sum_{j = 1}^{T} \alpha_j^2 \, \prod_{i = j+1}^{T} (1 - r_i) 
\end{align*}
Using the fact that $\lambda_1 = 1$ and $\lambdakk = (1 - \ak) \lambdak$, we know that that $\lambda_{T+1} = \prod_{k = 1}^{T} (1 - \ak)$, and 
\begin{align*}
\E[f(\x_{T+1}) - f^*] & \leq \left[\prod_{k = 1}^{T} (1 - \ak) \right] \, \left[f(\x_1) - f^* + \frac{\mu}{2} \normsq{\xopt- \x_1} \right] + \frac{\sigma^2}{L} \sum_{j = 1}^{T} \alpha_j^2 \, \prod_{i = j+1}^{T} (1 - r_i).
\end{align*}
Now our task is to upper-bound the $1 - \ak$ terms. From~\cref{eq:a-q}, we know that 
\begin{align*}
\ak & =  \sqrt{\qkk \etak} = \sqrt{\frac{\qkk}{L}} \alphak \tag{Since $\etak = \frac{1}{L} \, \alphak^2$} \\
\implies (1 - \ak) &=  \left( 1 - \sqrt{\frac{\qkk}{L}} \alphak \right)
\end{align*}
Since $\qk = \mu$ for all $k$, putting everything together, 
\begin{align*}
\E[f(\x_{T+1}) - f^*] & \leq \left[\prod_{k = 1}^{T} \left( 1 - \sqrt{\frac{1}{ \kappa}} \alphak \right) \right] \, \left[f(\x_1) - f^* + \frac{\mu}{2} \normsq{\xopt- \x_1} \right] + \frac{\sigma^2}{L} \sum_{j = 1}^{T} \alpha_j^2 \, \prod_{i = j+1}^{T} \left( 1 - \sqrt{\frac{1}{ \kappa}} \alpha_i \right) 
\intertext{Denoting $\Delta_k = \E[f(\x_k) - f^*]$, using the exponential step-size $\alphak = \alpha^{k}= \left(\frac{\beta}{T}\right)^{\nicefrac{k}{T}}$ and that $f(\x_1) - f^* \geq \frac{\mu}{2} \normsq{\xopt- \x_1}$,}
\Delta_{T+1} & \leq 2 \exp \left( -\sqrt{\frac{1}{ \kappa}} \sum_{k = 1}^{T}  \alpha^{k} \right) \Delta_1 + \frac{\sigma^2}{L} \sum_{k = 1}^{T} \alpha^{2k} \, \exp \left( -\sqrt{\frac{1}{ \kappa}} \sum_{i = k+1}^{T}  \alpha^i \right)
\end{align*}
Using~\cref{lemma:A-bound}, we can bound the first term as
\begin{align*}
    2 \exp \left( -\sqrt{\frac{1}{ \kappa}} \sum_{k = 1}^{T}  \alpha^{k} \right) \Delta_1 &\leq 2 \exp \left( -\sqrt{\frac{1}{ \kappa}} \left( \frac{\alpha T}{\ln(\nicefrac{T}{\beta})} - \frac{2\beta}{\ln(\nicefrac{T}{\beta})}\right) \right) \Delta_1\\
    &= 2c_3 \exp\left( - \frac{T}{\sqrt{\kappa }} \frac{\alpha}{\ln(\nicefrac{T}{\beta})}\right) [f(\x_1) - f^*]
\end{align*}
where $c_3 = \exp \left( \frac{2 \beta}{\sqrt{ \kappa} \ln (\nicefrac{T}{\beta})}\right)$. We can now bound the second term by a proof similar to~\cref{lemma:B-bound}. Indeed we have
\begin{align*}
     \sum_{k = 1}^{T} \alpha^{2k} \, \exp \left( -\sqrt{\frac{1}{ \kappa}} \sum_{i = k+1}^{T}  \alpha^i \right) &=  \sum_{k = 1}^{T} \alpha^{2k} \, \exp \left( -\sqrt{\frac{1}{ \kappa}} \frac{\alpha^{k+1} - \alpha^{T+1}}{1 - \alpha} \right)\\
     &= \exp\left( \frac{1}{\sqrt{ \kappa}}\frac{\alpha^{T+1}}{1 - \alpha}\right)\sum_{k = 1}^{T} \alpha^{2k} \, \exp \left( -\sqrt{\frac{1}{ \kappa}} \frac{\alpha^{k+1}}{1 - \alpha} \right)\\
     &\leq \exp\left( \frac{1}{\sqrt{ \kappa}}\frac{\alpha^{T+1}}{1 - \alpha}\right)\sum_{k = 1}^{T} \alpha^{2k} \left( \frac{2 (1- \alpha) \sqrt{ \kappa}}{e\alpha^{k+1}}\right)^2 &\tag{\cref{lem:ineq2}}\\
     &= \exp\left( \frac{1}{\sqrt{ \kappa}}\frac{\alpha^{T+1}}{1 - \alpha}\right) \frac{4  \kappa}{e^2 \alpha^2} T(1 - \alpha)^2 \\
     &\leq \exp\left( \frac{1}{\sqrt{ \kappa}}\frac{\alpha^{T+1}}{1 - \alpha}\right) \frac{4  \kappa}{e^2 \alpha^2} T \ln( \nicefrac{1}{\alpha})^2\\
     &= \exp\left( \frac{1}{\sqrt{ \kappa}}\frac{\alpha^{T+1}}{1 - \alpha}\right) \frac{4  \kappa \ln(\nicefrac{T}{\beta})^2}{e^2 \alpha^2 T}
\end{align*}
Finally, 
\begin{align*}
    \exp\left( \frac{1}{\sqrt{ \kappa}}\frac{\alpha^{T+1}}{1 - \alpha}\right) &\leq \exp\left( \frac{2 \beta}{\sqrt{ \kappa} \ln(\nicefrac{T}{\beta})}\right) = c_3
\end{align*}
where the inequality comes from the bound in~\cref{ineq:lemma-A-bound} in the proof of~\cref{lemma:A-bound}.
Putting everything together we obtain
\begin{align*}
    \E [f(\x_{T+1}) - f^*] & \leq 2c_3  \exp\left( - \frac{T}{\sqrt{\kappa }} \frac{\alpha}{\ln(\nicefrac{T}{\beta})}\right) [f(\x_1) - f^*]  + \frac{4 \, \sigma^2 c_3 \kappa \ln(\nicefrac{T}{\beta})^2}{ L e^2 \alpha^2 T}.
\end{align*}
\end{proof}

\clearpage
\subsection{Proof for misspecified ASGD}
\label{app:misacceleration-proof}
In this section, we assume that $L$ and $\mu$ are misspecified by positive $\nul$ and $\num$ and we set $\nu=\nul \num$. In particular, we will use $\tilde{\Lmax}$ and $\tilde{\mu}$, offline estimates of the smoothness and strong-convexity parameters. W.l.o.g we will assume that $\tilde{\mu} = \num \mu$ and $\tilde{\Lmax} = \frac{\Lmax}{\nul}$. Importantly, we will assume that $\num \leq 1$ i.e. we will only consider the more practical case where we underestimate the strong-convexity. Since $\num \leq 1$, $f$ is also $\tilde{\mu}$ strongly-convex. Hence, all the equations of~\eqref{lemma:reformulation} hold where we replace $\mu$ with $\tilde{\mu}$. Similarly in the definition of $\phik^*$ in \eqref{eq:phistar-update} we can replace $\mu$ with $\tilde{\mu}$.

With this estimated value, the extrapolation parameter $\bk$ is computed as follows:
\begin{align}
\ak^2 &= (1 - \ak) \akp^2 \frac{\etak}{\etakp} + \ak \tilde{\mu} \etak. \\
\bk &= \frac{(1 - \akp) \akp \frac{\etak}{\etakp}}{\ak + \akp^2 \frac{\etak}{\etakp}}, 
\end{align}
where $\etak = \gammak \alphak = \frac{1}{ \tilde{\Lmax}} \alphak = \frac{\nul}{ \Lmax} \left(\frac{\beta}{T} \right)^{k/T}$, $\ak = \sqrt{\frac{\tilde{\mu}}{ \tilde{L}}} \left(\frac{\beta}{T} \right)^{k/2T} = \sqrt{\frac{\nu \mu}{ \Lmax}} \left(\frac{\beta}{T} \right)^{k/2T}=\sqrt{\frac{\nu}{ \kappa}} \left(\frac{\beta}{T} \right)^{k/2T}$ satisfy the above equations. It can be verified that the reformulation in~\cref{app:reformulation} does not use the specific form of $\ak$ and $\etak$, and only relies on the consistency of the update above. Hence, the update can be be reformulated as a 3 variable sequence as in~\cref{app:reformulation} with a different choice of $\etak$ and $\ak$, but with $\qk$ and $\zk$ defined analogously in terms of $\etak$ and $\ak$.   

Similarly, it can be verified that the definition of the estimating sequences in~\cref{app:estimating-sequences} also does not depend on the specific value of $\ak$ and $\etak$, and hence we use the same definition of $\phik$.

\subsubsection{Proof of~\cref{thm:sc-misaccelerated}}
\label{app:acceleration-thm-proof-miss}
Given the definitions in~\cref{app:estimating-sequences}, we first prove the following descent lemma with the misspecified step-size. 
\begin{thmbox}
\begin{lemmanew}
Using the update in~\cref{eq:x-update} with $\etak = \frac{\nul}{ L} \alphak^2$, and defining $k_0 := T \frac{[\ln(\nul)]_{+}}{\ln(T/\beta)}$ and $G = \max_{j \in [k_0]} \norm{\grad{y_j}}$, we obtain the following descent lemma. 
\begin{align*}
\E[f(\xkk)] & \leq \E[f(\yk)] - \frac{\etak}{2} \normsq{\grad{\yk}} + \frac{\nul^2}{2  L} \alphak^2 \sigma^2 + \cI{k < k_0} \frac{G^2 \nul^2}{2  L} \alphak^2 
\end{align*}
\label{lemma:descent-lemma-miss}
\end{lemmanew}
\end{thmbox}
\vspace{-2ex}
\begin{proof}
By smoothness, and the update in~\cref{eq:x-update},
\begin{align*}
f(\xkk) & \leq f(\yk) - \etak \inner{\grad{\yk}}{\gradk{\yk}} + \frac{L}{2} \etak^2 \normsq{\gradk{\yk}} \\
\intertext{Taking expectation w.r.t. $i_k$,}
\E[f(\xkk)] & \leq \E[f(\yk)] - \etak \normsq{\grad{\yk}} + \frac{L}{2} \etak^2 \E[\normsq{\gradk{\yk}}] & \tag{$\etak$ is independent of the randomness in $i_k$.} \\
& \leq \E[f(\yk)] - \etak \normsq{\grad{\yk}} + \frac{  L}{2} \etak^2 \normsq{\grad{\yk}} + \frac{L}{2} \etak^2 \sigma^2 & \tag{From~\cref{eq:bounded-variance}} \\
& = \E[f(\yk)] - \etak \normsq{\grad{\yk}} + \frac{\etak \nul \alphak^2}{2} \normsq{\grad{\yk}} + \frac{\nul^2}{2  L} \alphak^4 \sigma^2 \\
& \leq \E[f(\yk)] - \etak \normsq{\grad{\yk}} + \frac{\etak \nul \alphak}{2} \normsq{\grad{\yk}} + \frac{\nul^2}{2  L} \alphak^2 \sigma^2 \tag{Since $\alphak \leq 1$}\\
& = \E[f(\yk)] - \frac{\etak}{2} \normsq{\grad{\yk}} - \frac{\etak}{2} \, \left(1 -  \nul \alphak \right) \normsq{\grad{\yk}} + \frac{\nul^2}{2   L} \alphak^2 \sigma^2 
\intertext{For $k \geq k_0$, $1 - \nul \alphak \geq 0$, implying that}
\E[f(\xkk)] & \leq \E[f(\yk)] - \frac{\etak}{2} \normsq{\grad{\yk}} + \frac{\nul^2}{2   L} \alphak^2 \sigma^2. 
\intertext{For $k < k_0$,}
\E[f(\xkk)] & \leq \E[f(\yk)] - \frac{\etak}{2} \normsq{\grad{\yk}} + \frac{\etak \nul \alphak}{2} \normsq{\grad{\yk}} + \frac{\nul^2}{2   L} \alphak^2 \sigma^2 \\
\E[f(\xkk)] & \leq \E[f(\yk)] - \frac{\etak}{2} \normsq{\grad{\yk}} + \frac{\alphak^2 \nul^2}{2   L} \normsq{\grad{\yk}} + \frac{\nul^2}{2   L} \alphak^2 \sigma^2 & \tag{Since $\etak = \frac{\nul}{  L} \alphak$,}
\intertext{Since $G = \max_{j \in [k_0]} \norm{\grad{y_j}}$, we can further upper-bound the RHS by}
\E[f(\xkk)] & \leq \E[f(\yk)] - \frac{\etak}{2} \normsq{\grad{\yk}} + \frac{G^2 \nul^2}{2   L} \alphak^2 + \frac{\nul^2}{2   L} \alphak^2 \sigma^2. 
\end{align*}
\end{proof}

Let us now prove the equivalent of~\cref{lemma:acceleration-induction} using the above modified descent lemma.   
\begin{thmbox}
\begin{lemmanew}
For the estimating sequences defined in~\cref{app:estimating-sequences} and the updates in~\cref{eq:y-update}-\cref{eq:z-def}, 
\begin{align*}
\E[\phik^*] :=  \E[ \inf_{\x} \phik(\x)] & \geq \E[f(\xk)] - \gNk  \end{align*}
where $\gNk := \frac{\sigma^2 (\nul^2+\nul)}{2  L} \sum_{j=1}^{k} \alpha_j^2 \, \prod_{i = j+1}^{k} (1 - r_i) + \left(\frac{G^2 \nul^2}{2 L} \right) \sum_{j=1}^{ \min\{k_0,k\} - 1} \alpha_j^2 \, \prod_{i = j+1}^{k} (1 - r_i)$
\label{lemma:acceleration-induction-miss}
\end{lemmanew}
\end{thmbox}
\begin{proof}
We will prove the lemma by induction. For $k = 1$, we define $\phi^*_1 = f(\x_1)$, and since $\gNk \geq 0$ for all $k$, $\E[\phi^*] \geq f(\x_1) - \gN_1$, thus satisfying the base-case for the induction. For the induction, we will use the fact that $\gNkk = (1 - \ak) \gNk + \frac{2\nul^2 \sigma^2}{\rho^2 L} \alphak^2 + \cI{k < k_0} \frac{G^2 \nul^2}{2 \rho L} \alphak^2$. 

Assuming the induction hypothesis, $\E[\phik^*] \geq \E[f(\xk)] - \gNk$, we use~\cref{eq:phistar-update} to prove the statement for $k+1$ as follows. Taking expectations w.r.t to the randomness in $j = 1$ to $k$,
\begin{align*}
\E[\phikk^*] & = (1 - \ak) \E[\phik^*] + \ak \E\left[\fk(\yk) - \frac{\ak}{2 \qkk} \normsq{\gradk{\yk}} +  \frac{(1 - \ak) \, \qk}{\qkk} \left( \frac{\tilde{\mu}}{2} \normsq{\yk - \zk} + \inner{\gradk{\yk}}{\zk - \yk} \right)\right] \\
& \geq (1 - \ak) \E[f(\xk) - \gNk] \\ & + \ak \E\left[\fk(\yk) - \frac{\ak}{2 \qkk} \normsq{\gradk{\yk}} +  \frac{(1 - \ak) \, \qk}{\qkk} \left( \frac{\tilde{\mu}}{2} \normsq{\yk - \zk} + \inner{\gradk{\yk}}{\zk - \yk} \right)\right] & \tag{by the induction hypothesis} \\
& = (1 - \ak) \E[f(\xk)] + \ak \E[f(\yk)] - \frac{\ak^2}{2 \qkk} \E\normsq{\gradk{\yk}} \\ & + \frac{\ak (1 - \ak) \, \qk}{\qkk} \E\left( \frac{\tilde{\mu}}{2} \normsq{\yk - \zk} + \inner{\grad{\yk}}{\zk - \yk} \right) \\ & - (1 - \ak) \gNk \\
& = (1 - \ak) \E[f(\xk)] + \ak \E[f(\yk)] - \frac{\etak}{2} \E\normsq{\gradk{\yk}}\\
& +  \frac{\ak (1 - \ak) \, \qk}{\qkk} \E\left( \frac{\tilde{\mu}}{2} \normsq{\yk - \zk} + \inner{\grad{\yk}}{\zk - \yk} \right) - (1 - \ak) \gNk & \tag{Using~\cref{eq:a-q}} \\
& = (1 - \ak) \E[f(\xk)] + \ak \E[f(\yk)] - \frac{\etak}{2} \E\normsq{\grad{\yk}}\\
& +  \frac{\ak (1 - \ak) \, \qk}{\qkk} \E\left( \frac{\tilde{\mu}}{2} \normsq{\yk - \zk} + \inner{\grad{\yk}}{\zk - \yk} \right) - (1 - \ak) \gNk- \frac{\etak}{2} \sigma^2 \tag{\blue{Using~\cref{eq:bounded-variance}}}\\
\intertext{By convexity, $f(\xk) \geq f(\yk) + \inner{\grad{\yk}}{\xk - \yk}$,}
& \geq (1 - \ak) \E[f(\yk) + \inner{\grad{\yk}}{\xk - \yk}] + \ak \E[f(\yk)] - \frac{\etak}{2} \E\normsq{\grad{\yk}} \\ & + \frac{\ak (1 - \ak) \, \qk}{\qkk}  \E \left( \frac{\tilde{\mu}}{2} \normsq{\yk - \zk} + \inner{\grad{\yk}}{\zk - \yk} \right)  - (1 - \ak) \gNk - \frac{\etak}{2} \sigma^2\\
& = \E \left[f(\yk) - \frac{\etak}{2} \normsq{\grad{\yk}} \right] +  \frac{\ak (1 - \ak) \, \qk}{\qkk} \E \left( \frac{\tilde{\mu}}{2} \normsq{\yk - \zk} + \inner{\grad{\yk}}{\zk - \yk} \right) \\ & +  (1 - \ak) \E[\inner{\grad{\yk}}{\xk - \yk}] - (1 - \ak) \gNk- \frac{\etak}{2} \sigma^2 \\
\intertext{By~\cref{lemma:descent-lemma-miss},}
& \geq \E \left[f(\xkk) - \frac{\nul^2}{2  L} \alphak^2 \sigma^2 - \cI{k < k_0} \frac{G^2 \nul^2}{2  L} \alphak^2 \right] +  \frac{\ak (1 - \ak) \, \qk}{\qkk} \E \left( \frac{\tilde{\mu}}{2} \normsq{\yk - \zk} + \inner{\grad{\yk}}{\zk - \yk} \right) \\ & + (1 - \ak) \E[\inner{\grad{\yk}}{\xk - \yk}] - (1 - \ak) \gNk - \frac{\etak}{2} \sigma^2\\
& = \E \left[f(\xkk) \right] +  \frac{\ak (1 - \ak) \, \qk}{\qkk} \E \left( \frac{\tilde{\mu}}{2} \normsq{\yk - \zk} + \inner{\grad{\yk}}{\zk - \yk} \right) +  (1 - \ak) \E[\inner{\grad{\yk}}{\xk - \yk}] \\ & - \left[(1 - \ak) \gNk +  \frac{\nul^2}{2 L} \alphak^2 \sigma^2+ \frac{\nul}{2 L} \alphak^2 \sigma^2 + \cI{k < k_0} \frac{G^2 \nul^2}{2 L} \alphak^2 \right]\\
\intertext{Since $\gNkk = \left[(1 - \ak) \gNk + \frac{(\nul^2+\nul)}{2 L} \alphak^2 \sigma^2 + \cI{k < k_0} \frac{G^2 \nul^2}{2 L} \alphak^2 \right]$,}
\E[\phikk^*] & \geq \E \left[f(\xkk) \right] - \gNkk + (1 - \ak) \E[\inner{\grad{\yk}}{\xk - \yk}] \\ & + \frac{\ak (1 - \ak) \, \qk}{\qkk} \E \left( \frac{\tilde{\mu}}{2} \normsq{\yk - \zk} + \inner{\grad{\yk}}{\zk - \yk} \right)  \\
\end{align*}
Similar to~\cref{lemma:acceleration-induction}, we show that $(1 - \ak) \E\left[ \inner{\grad{\yk}}{\xk - \yk} + \frac{\ak \, \qk}{\qkk}  \left( \frac{\tilde{\mu}}{2} \normsq{\yk - \zk} + \inner{\grad{\yk}}{\zk - \yk} \right) \right] \geq 0$. For this, we use modified~\cref{eq:y-update}
\begin{align*}
\yk & = \xk - \frac{\qk \ak}{\qk + \ak \tilde{\mu}} (\xk - \zk) \\
\implies \zk - \yk &= \zk - \xk + \frac{\qk \ak}{\qk + \ak \tilde{\mu}} (\xk - \zk) = \left(1 - \frac{\qk \ak}{\qk + \ak \tilde{\mu}} \right) (\zk - \xk) \\
& = \left( \frac{\qk (1 - \ak) + \ak \tilde{\mu}}{\qk + \ak \tilde{\mu}} \right) (\zk - \xk) = \left( \frac{\qkk}{\qk + \ak \tilde{\mu}} \right) (\zk - \xk) & \tag{By~\cref{eq:q-update}} \\
\implies \frac{\ak \qk}{\qkk} \inner{\grad{\yk}}{\zk - \yk} &= \left \langle \grad{\yk} , \left( - \frac{\ak \, \qk}{\qk + \ak \tilde{\mu}} \right) (\xk - \zk) \right \rangle = \left \langle \grad{\yk} , \yk- \xk \right \rangle
\end{align*}
Using this relation to simplify, 
\begin{align*}
& (1 - \ak)\E \left[ \inner{\grad{\yk}}{\xk - \yk} + \frac{\ak \, \qk}{\qkk}  \left( \frac{\tilde{\mu}}{2} \normsq{\yk - \zk} + \inner{\grad{\yk}}{\zk - \yk} \right) \right] \\ 
&= (1 - \ak) \E \left[\frac{\ak \, \qk}{\qkk} \frac{\tilde{\mu}}{2} \normsq{\yk - \zk} + (1 - \ak) \left[\inner{\grad{\yk}}{\xk - \yk} + \left \langle \grad{\yk} , \yk- \xk \right \rangle \right] \right]\\
& = (1 - \ak) \E \left[\frac{\ak \, \qk}{\qkk} \frac{\tilde{\mu}}{2} \normsq{\yk - \zk} \right]\geq 0 & \tag{Since $\ak \leq 1$.}
\end{align*}
Putting everything together, 
\begin{align*}
\E[\phikk^*] & \geq \E \left[f(\xkk) \right] - \gNkk    
\end{align*}
and we conclude that $\E[\phik^*] \geq \E \left[f(\xk) \right] - \gNk$ for all $k$ by induction.  
\end{proof}

\clearpage
We now use the above lemma to prove~\cref{thm:sc-misaccelerated}.
\myquote{
\restatescmisacc*
}
\begin{proof}
In order to have a valid estimate sequence, we need to restrict values that $\nu$. Since $\nu \leq \rho \kappa$, $0 \leq r_k \leq 1$ and hence $\lambdak \in (0,1)$, as required for a valid estimate sequence. For the estimating sequences defined in~\cref{app:estimating-sequences}, using~\cref{lemma:acceleration-induction}, we know that the (reformulated) updates satisfy the following relation
\begin{align*}
\E[f(\x_{T+1})] & \leq \E[\phi^*_{T+1}] + \gN_{T+1}  \leq  \E[\phi_{T+1}(\xopt)] + \gN_{T+1}
\end{align*}
From~\cref{eq:phi-relation}, we know that for all $\x$ and $k$, 
\begin{align*}
\phik(\x) &\leq (1 - \lambdak) f(\x) + \lambdak \phi_0(\x)    
\end{align*}
Using these relations, 
\begin{align*}
\E[f(\x_{T+1})] & \leq (1 - \lambda_{T}) f^* + \lambda_{T} \phi_1(\xopt) + \gN_{T} \\
\implies \E[f(\x_{T+1}) - f^*] & \leq \lambda_{T} \left[\phi_1(\xopt) - f^* \right] + \gN_{T+1} \\
\intertext{By~\cref{eq:phi-def},}
& \leq \lambda_{T} \left[\phi^*_1 + \frac{\q_1}{2} \normsq{\xopt- \z_1} - f^* \right] + \gN_{T+1} \\
\intertext{Choosing $\phi^*_1 = f(\x_1)$,}
& \leq \lambda_{T} \left[f(\x_1) - f^* + \frac{\q_1}{2} \normsq{\xopt- \z_1} \right] + \gN_{T} \\
\intertext{Since $\z_1 = \x_1$, and we set $\q_1 = \tilde{\mu}$,}
\implies \E[f(\x_{T+1}) - f^*] & \leq \lambda_{T} \left[f(\x_1) - f^* + \frac{\tilde{\mu}}{2} \normsq{\xopt- \x_1} \right] \\ & + \frac{\sigma^2 (\nul^2+\nul)}{2 L} \sum_{j=1}^{T} \alpha_j^2 \, \prod_{i = j+1}^{T} (1 - r_i) + \left(\frac{G^2 \nul^2}{2 L} \right) \sum_{j=1}^{ \min\{k_0,T\} - 1} \alpha_j^2 \, \prod_{i = j+1}^{T} (1 - r_i)
\end{align*}
Using the fact that $\lambda_1 = 1$ and $\lambdakk = (1 - \ak) \lambdak$, and since $\nu \leq \rho \kappa$, we know that $\ak \leq 1$ and $\lambda_{T} = \prod_{k = 1}^{T} (1 - \ak)$, and 
\begin{align*}
\E[f(\x_{T+1}) - f^*] & \leq \left[\prod_{k = 1}^{T} (1 - \ak) \right] \, \left[f(\x_1) - f^* + \frac{\tilde{\mu}}{2} \normsq{\xopt- \x_1} \right] + \frac{ \sigma^2 (\nul^2+\nul)}{2 L} \sum_{j=1}^{T} \alpha_j^2 \, \prod_{i = j+1}^{T} (1 - r_i) \\ & + \left(\frac{G^2 \nul^2}{2 L} \right) \sum_{j=1}^{ \min\{k_0,T\} - 1} \alpha_j^2 \, \prod_{i = j+1}^{T} (1 - r_i).
\end{align*}
Now our task is to upper-bound bound the $1 - \ak$ terms. From~\cref{eq:a-q}, we know that 
\begin{align*}
\ak & =  \sqrt{\qkk \etak} =  \sqrt{\frac{\qkk \nul}{ L}} \alphak  &   \\
\implies (1 - \ak) & = \left( 1 - \sqrt{\frac{\qkk \nul}{ L}} \alphak \right)
\end{align*}
Since $\qk = \tilde{\mu}$ for all $k$, putting everything together, 
\begin{align*}
\E[f(\x_{T+1}) - f^*] & \leq \left[\prod_{k = 1}^{T} \left( 1 - \sqrt{\frac{\nu}{ \kappa}} \alphak \right) \right] \, \left[f(\x_1) - f^* + \frac{\tilde{\mu}}{2} \normsq{\xopt- \x_1} \right] + \frac{ \sigma^2 (\nul^2+\nul)}{2 L} \sum_{j=1}^{T} \alpha_j^2 \, \prod_{i = j+1}^{T} \left( 1 - \sqrt{\frac{\nu}{ \kappa}} \alpha_i \right) \\ & + \left(\frac{G^2 \nul^2}{2  L} \right) \sum_{j=1}^{ \min\{k_0,T\} - 1} \alpha_j^2 \, \prod_{i = j+1}^{T} \left( 1 - \sqrt{\frac{\nu}{ \kappa}} \alpha_i \right).
\intertext{Denoting $\Delta_k = \E[f(\x_k) - f^*]$ and $\Delta_1 = \E[f(\x_1) - f^*]$, using the exponential step-size $\alphak = \alpha^{\nicefrac{k}{T}} = \left(\frac{1}{T}\right)^{\nicefrac{k}{T}}$ and that $f(\x_1) - f^* \geq \frac{\tilde{\mu}}{2} \normsq{\x_1 - \xopt}$. Similar to the proof of~\cref{thm:sc-unknown-decorrelated}, if $\nu > 1$, we can replace $\nu$ by 1 and get an upper-bound for the $\left( 1 - \sqrt{\frac{\nu}{ \kappa}} \alphak \right)$ term. Hence, we define $\hat{\nu}:=\min\{1,\nu\}$, and obtain the following upper-bound.}
\Delta_{T+1} & \leq 2 \exp \left( -\sqrt{\frac{\hat{\nu}}{ \kappa}} \sum_{k = 1}^{T}  \alpha^{k} \right) \Delta_1 + \frac{\sigma^2 (\nul^2+\nul)}{2  L} \sum_{k = 1}^{T} \alpha^{2k} \, \exp \left( -\sqrt{\frac{\hat{\nu}}{\kappa}} \sum_{i = k+1}^{T}  \alpha^i \right) \\ & + \frac{G^2 \nul^2}{2  L} \sum_{k = 1}^{\min\{k_0,T\} -1} \alpha^{2k} \, \exp \left( -\sqrt{\frac{\hat{\nu}}{ \kappa}} \sum_{i = k+1}^{T}  \alpha^i \right) 
\end{align*}
Using~\cref{lemma:A-bound}, we can bound the first term as
\begin{align*}
    2 \exp \left( -\sqrt{\frac{\hat{\nu}}{  \kappa}} \sum_{k = 1}^{T}  \alpha^{k} \right) \Delta_1 &\leq 2 \exp \left( -\sqrt{\frac{\hat{\nu}}{  \kappa}} \left( \frac{\alpha T}{\ln(\nicefrac{T}{\beta})} - \frac{2\beta}{\ln(\nicefrac{T}{\beta})}\right) \right) \Delta_1\\
    &\leq 2 \exp \left( \frac{2 \beta}{\sqrt{  \kappa} \ln (\nicefrac{T}{\beta})}\right)\exp\left( - \frac{T \sqrt{\hat{\nu}}}{\sqrt{\kappa  }} \frac{\alpha}{\ln(\nicefrac{T}{\beta})}\right) [f(\x_1) - f^*] \tag{since $\hat{\nu} \leq 1$}\\
    &= 2c_3 \exp\left( - \frac{T \sqrt{\hat{\nu}}}{\sqrt{\kappa  }} \frac{\alpha}{\ln(\nicefrac{T}{\beta})}\right) [f(\x_1) - f^*]
\end{align*}
where $c_3 = \exp \left( \frac{2 \beta}{\sqrt{  \kappa} \ln (\nicefrac{T}{\beta})}\right)$. We can now bound the second term by a proof similar to~\cref{lemma:B-bound}. Indeed we have
\begin{align*}
     \sum_{k = 1}^{T} \alpha^{2k} \, \exp \left( -\sqrt{\frac{\hat{\nu}}{  \kappa}} \sum_{i = k+1}^{T}  \alpha^i \right) &=  \sum_{k = 1}^{T} \alpha^{2k} \, \exp \left( -\sqrt{\frac{\hat{\nu}}{  \kappa}} \frac{\alpha^{k+1} - \alpha^{T+1}}{1 - \alpha} \right)\\
     &= \exp\left( \frac{\sqrt{\hat{\nu}}}{\sqrt{  \kappa}}\frac{\alpha^{T+1}}{1 - \alpha}\right)\sum_{k = 1}^{T} \alpha^{2k} \, \exp \left( -\sqrt{\frac{\hat{\nu}}{  \kappa}} \frac{\alpha^{k+1}}{1 - \alpha} \right)\\
     &\leq \exp\left( \frac{\sqrt{\hat{\nu}}}{\sqrt{  \kappa}}\frac{\alpha^{T+1}}{1 - \alpha}\right)\sum_{k = 1}^{T} \alpha^{2k} \left( \frac{2 (1- \alpha) \sqrt{  \kappa}}{e\alpha^{k+1} \sqrt{\hat{\nu}}} \right)^2 &\tag{\cref{lem:ineq2}}\\
     &= \exp\left( \frac{\sqrt{\hat{\nu}}}{\sqrt{  \kappa}}\frac{\alpha^{T+1}}{1 - \alpha}\right) \frac{4   \kappa}{e^2 \hat{\nu} \alpha^2} T(1 - \alpha)^2 \\
     &\leq \exp\left( \frac{\sqrt{\hat{\nu}}}{\sqrt{  \kappa}}\frac{\alpha^{T+1}}{1 - \alpha}\right) \frac{4   \kappa}{e^2 \hat{\nu} \alpha^2} T \ln( \nicefrac{1}{\alpha})^2\\
     &= \exp\left( \frac{\sqrt{\hat{\nu}}}{\sqrt{  \kappa}}\frac{\alpha^{T+1}}{1 - \alpha}\right) \frac{4   \kappa \ln(\nicefrac{T}{\beta})^2}{e^2 \hat{\nu} \alpha^2 T}
\end{align*}
Similarly, 
\begin{align*}
     \sum_{k = 1}^{\min\{k_0,T\}-1} \alpha^{2k} \, \exp \left( -\sqrt{\frac{\hat{\nu}}{  \kappa}} \sum_{i = k+1}^{T}  \alpha^i \right) & \leq \exp\left( \frac{\sqrt{\hat{\nu}}}{\sqrt{  \kappa}}\frac{\alpha^{T+1}}{1 - \alpha}\right) \frac{4   \kappa \ln(\nicefrac{T}{\beta})^2 \min\{k_0,T\}}{e^2 \hat{\nu} \alpha^2 T^2} \\ & = \exp\left( \frac{\sqrt{\hat{\nu}}}{\sqrt{  \kappa}}\frac{\alpha^{T+1}}{1 - \alpha}\right) \frac{4   \kappa \ln(\nicefrac{T}{\beta})^2 \min \left\{\frac{\ln(\nu)}{\ln(T/\beta)},1\right\}}{e^2 \hat{\nu} \alpha^2 T}
\end{align*}
Finally, 
\begin{align*}
    \exp\left( \frac{\sqrt{\hat{\nu}}}{\sqrt{  \kappa}}\frac{\alpha^{T+1}}{1 - \alpha}\right) &\leq \exp\left( \frac{2 \beta \sqrt{\hat{\nu}}}{\sqrt{  \kappa} \ln(\nicefrac{T}{\beta})}\right) = c_3
\end{align*}
where the inequality comes from the bound in~\cref{ineq:lemma-A-bound} in the proof of~\cref{lemma:A-bound}.
Putting everything together we obtain
\begin{align*}
    \E [f(\x_{T+1}) - f^*] & \leq 2c_3  \exp\left( - \frac{\sqrt{\hat{\nu}}T}{\sqrt{\kappa  }} \frac{\alpha}{\ln(\nicefrac{T}{\beta})}\right) \Delta_1  + \frac{2 c_3 \kappa \ln(\nicefrac{T}{\beta})^2}{e^2 \alpha^2 T} \frac{\sigma^2 (\nul^2+\nul)}{\hat{\nu}  L} \\
    & + \frac{2 c_3 \kappa \ln(\nicefrac{T}{\beta})^2}{e^2 \alpha^2 T} \, \min \left\{\frac{{[\ln(\nul)]_{+}}}{\ln(T/\beta)},1\right\} \, \frac{G^2 \nul^2}{L \hat{\nu}}\\
    &\leq \exp\left( - \frac{\sqrt{\hat{\nu}}T}{\sqrt{\kappa  }} \frac{\alpha}{\ln(\nicefrac{T}{\beta})}\right) \Delta_1  + \frac{2 c_3 \kappa \ln(\nicefrac{T}{\beta})^2}{e^2 \alpha^2 T} \frac{\sigma^2 \max\{(\nul^2+\nul) ,(\nul+1)/\num\}}{  L} \\
    & + \frac{2 c_3 \kappa \ln(\nicefrac{T}{\beta})^2}{e^2 \alpha^2 T} \, \min \left\{\frac{{[\ln(\nul)]_{+}}}{\ln(T/\beta)},1\right\} \, \frac{G^2 \max\{\nul^2 ,\nul/\num\}}{L}\\
    &= \exp\left( - \frac{\sqrt{\hat{\nu}}T}{\sqrt{\kappa  }} \frac{\alpha}{\ln(\nicefrac{T}{\beta})}\right) \Delta_1  + \frac{2 c_3 \ln(\nicefrac{T}{\beta})^2}{e^2 \alpha^2 T} \frac{\sigma^2 \max\{(\nul^2+\nul) ,(\nul+1)/\num\}}{  \mu} \\
    & + \frac{2 c_3 \ln(\nicefrac{T}{\beta})^2}{e^2 \alpha^2 T} \, \min \left\{\frac{{[\ln(\nul)]_{+}}}{\ln(T/\beta)},1\right\} \, \frac{G^2 \max\{\nul^2 , \nul/\num\}}{\mu}
\end{align*}
For the second inequality, we consider the term $\frac{\nul^2}{\hat{\nu}}$. If $\hat{\nu} = \nu$, then $\frac{\nul^2}{\hat{\nu}}=\frac{\nul^2}{\nu}=\frac{\nul}{\num}$. If $\hat{\nu} = 1$, then $\frac{\nul^2}{\hat{\nu}}= \nul^2$. Putting these two cases together we get $\max\{\nu^2 ,\nul/\num\}$. Similarly, we simplify the term $\frac{\nul^2+\nul}{\hat{\nu}}$. The last equality comes from the fact that $\frac{\kappa}{\Lmax}=\frac{1}{\mu}$.
\end{proof} 
\clearpage
\section{Helper Lemmas}
\label{app:helper-lemmas}
\begin{thmbox}
\begin{lemmanew}
\label{lem:ineq1}
For all $x>1$,
\begin{align*}
    \frac{1}{x - 1} \leq \frac{2}{\ln(x)}
\end{align*}
\end{lemmanew}
\end{thmbox}
\begin{proof}
For $x > 1$, we have
\begin{align*}
    \frac{1}{x - 1} \leq \frac{2}{\ln(x)} &\iff \ln(x) < 2x - 2
\end{align*}
Define $f(x) = 2x - 2 - \ln(x)$. We have $f'(x) = 2 - \frac{1}{x}$. Thus for $x \geq 1$, we have $f'(x) > 0$ so $f$ is increasing on $[1, \infty)$. Moreover we have $f(1) = 2 - 2 - \ln(1) = 0$ which shows that $f(x) \geq 0$ for all $x > 1$ and ends the proof.
\end{proof}

\begin{thmbox}
\begin{lemmanew}
\label{lem:ineq2}
For all $x, \gamma >0$,
\begin{align*}
    \exp(-x) \leq \left( \frac{\gamma}{ex }\right)^\gamma
\end{align*}
\end{lemmanew}
\end{thmbox}
\begin{proof}
Let $x > 0$. Define $f(\gamma) = \left( \frac{\gamma}{ex }\right)^\gamma - \exp(-x)$. We have
\begin{align*}
    f(\gamma) = \exp\left( \gamma\ln(\gamma) - \gamma\ln(ex)\right) - \exp(-x)
\end{align*}
and
\begin{align*}
    f'(\gamma) = \left( \gamma \cdot \frac{1}{\gamma} + \ln(\gamma) - \ln(ex)\right) \exp\left( \gamma\ln(\gamma) - \gamma\ln(ex)\right)
\end{align*}
Thus
\begin{align*}
    f'(\gamma) \geq 0 &\iff 1 + \ln(\gamma) - \ln(ex) \geq 0 \iff \gamma \geq \exp\left( \ln(ex) - 1\right) = x
\end{align*}
So $f$ is decreasing on $(0, x]$ and increasing on $[x, \infty)$. Moreover,
\begin{align*}
    f(x) = \left( \frac{x}{ex}\right)^x - \exp(-x) = \left(\frac{1}{e}\right)^x - \exp(-x) = 0
\end{align*}
and thus $f(\gamma) \geq 0$ for all $\gamma > 0$ which proves the lemma.
\end{proof}

\begin{thmbox}
\begin{lemmanew}
\label{lem:sum-prod-eq}
For any sequence $\alpha_k$
\begin{align*}
    \prod_{k=1}^T (1 - \alpha_k) + \sum_{k=1}^T \alpha_k \prod_{i=k+1}^T (1 - \alpha_i)  = 1 
\end{align*}
\end{lemmanew}
\end{thmbox}
\begin{proof}
We show this by induction on $T$. For $T = 1$,
\begin{align*}
    (1 - \alpha_1) + \alpha_1 = 1
\end{align*}
Induction step:
\begin{align*}
    \prod_{k=1}^{T+1}(1 - \alpha_k) + \sum_{k=1}^{T+1} \alpha_k \prod_{i=k+1}^{T+1} (1 - \alpha_i)  &= (1- \alpha_{T+1}) \prod_{k=1}^T (1- \alpha_k) + \left( \alpha_{T+1} + \sum_{k=1}^T\alpha_k \prod_{i=k+1}^{T+1} (1 - \alpha_i)\right) \\
    &= (1- \alpha_{T+1}) \prod_{k=1}^T (1- \alpha_k) + \left( \alpha_{T+1} + (1 - \alpha_{T+1})\sum_{k=1}^T\alpha_k \prod_{i=k+1}^{T} (1 - \alpha_i)\right)\\
    &= ( 1- \alpha_{T+1}) \left(\underbrace{ \prod_{k=1}^T (1- \alpha_k) + \sum_{k=1}^T\alpha_k \prod_{i=k+1}^{T} (1 - \alpha_i)}_{= 1}\right) + \alpha_{T+1} \tag{Induction hypothesis}\\
    &= (1 - \alpha_{T+1}) + \alpha_{T+1} = 1
\end{align*}
\end{proof}

\end{document}